%% file: thesis.tex
\documentclass[12pt,a4paper,twoside]{article}

\title{Decomposition of Geometric Set Systems and Graphs}
\author{Dömötör Pálvölgyi}
\date{}
\usepackage{amssymb,amsmath,amsthm}
\usepackage[english]{babel}
\usepackage{t1enc}
\usepackage[latin2]{inputenc}
\usepackage{epsfig}
\usepackage{subfigure, footnpag}

\newtheorem{thm}{Theorem}[section]
\newtheorem{cor}[thm]{Corollary}
\newtheorem{defi}[thm]{Definition}
\newtheorem{lem}[thm]{Lemma}
\newtheorem{prop}[thm]{Proposition}
\newtheorem{claim}[thm]{Claim}
\newtheorem{remark}[thm]{Remark}
\newtheorem{conj}[thm]{Conjecture}
\newtheorem{obs}[thm]{Observation}
\newtheorem{quest}[thm]{Question}

\usepackage{indentfirst}
\def\C{{\mathcal C}}
\def\T{{\mathcal T}}

\def\s{{\bf sl}}
\def\figsize{0.55}

\begin{document}
\pagestyle{empty}
\input{front.tex}

\cleardoublepage
\pagestyle{plain}
\pagenumbering{roman}
{\bf \large Abstract}\\
\input{abs.tex}

\cleardoublepage
{\bf \large Résumé}\\

\input{absf.tex}

\cleardoublepage

\input{thx.tex}

\cleardoublepage
\tableofcontents

\cleardoublepage
\pagenumbering{arabic}
\section{Introduction and Organization} 
\input{intro.tex}

\cleardoublepage
\part{Decomposition of Multiple Coverings}
\section{Introduction and Survey}\label{sec:introdec}
\input{introtodec.tex}

\clearpage
\section{Decomposition of Coverings by Convex Polygons}\label{sec:convex}
\input{convex.tex}

\clearpage
\section{Indecomposable Coverings by Concave Polygons}\label{sec:concave}
\input{concave.tex}

\cleardoublepage
\part{Slope Number of Graphs}
\section{Introduction and a Lower Bound}\label{sec:introslope}
\input{introslope.tex}

\clearpage
\subsection{A Lower Bound for the Slope Number of Graphs with Bounded Degree}\label{sec:slopedeg5}
\input{slopedeg5.tex}

\clearpage
\section{Drawing Cubic Graphs with at most Five Slopes}\label{sec:slopenum}
\input{slopenum.tex}

\clearpage
\section{Slope Parameter of Cubic Graphs}\label{sec:slopepar}
\input{slopepar.tex}

\clearpage
\section{Drawing Planar Graphs with Few Slopes}\label{sec:slopeplan}
\input{slopeplan.tex}

\cleardoublepage
\part{Conjectures, Bibliography and CV}
\section{Summary of Interesting Open Questions and Conjectures}\label{sec:conj}
\input{conj.tex}

\cleardoublepage

\cleardoublepage
\section{CURRICULUM VIT\AE}\label{sec:cv}
\input{CV.tex}

\end{document}

%% file: front.tex
\begin{center}
\hspace{1cm}\\

\vspace{3cm}

\textbf{\Huge Decomposition of Geometric  \\ \vspace{3mm} Set Systems and Graphs}\\

\vspace{2cm}

\textbf{\Large Dissertation}

\vspace{2cm}

\textbf{\LARGE{Dömötör Pálvölgyi}}\\

\vspace{3cm}

{\bf Jury:}\\

Prof. J. Pach, thesis director \\
Prof. F. Eisenbrand, president\\
Prof. A. Shokrollahi, rapporteur \\
Prof. G. Rote, rapporteur \\
Prof. P. Valtr, rapporteur \\

\vspace{3cm}

{\large
\begin{figure}[th]
\begin{center}
{\large \epsfig{file=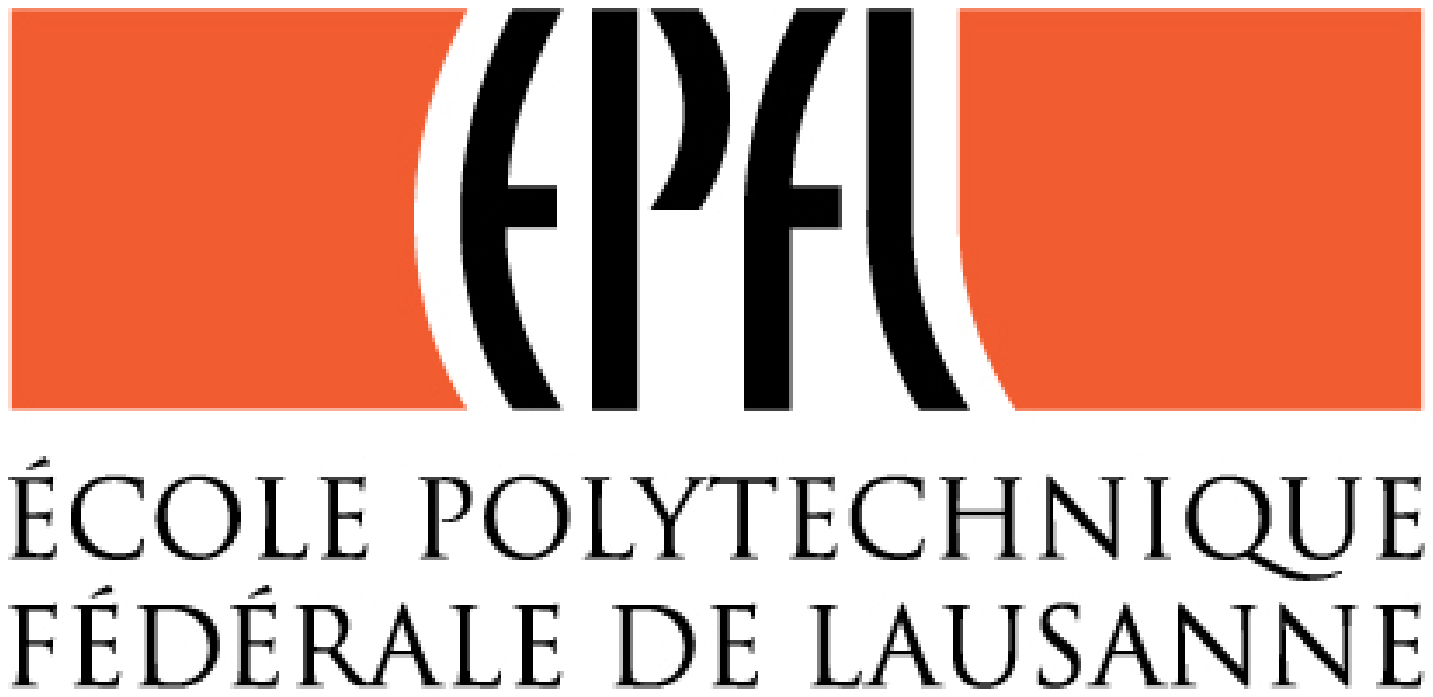,angle=0,width=0.2\textwidth}  }
\end{center}
\end{figure}
}

\end{center}

%% file: abs.tex
We study two decomposition problems in combinatorial geometry. The first part of the thesis deals with the decomposition of multiple coverings of the plane. We say that a planar set is {\it cover-decomposable} if there is a constant $m$ such that any $m$-fold covering of the plane with its translates is decomposable into two disjoint coverings of the whole plane. Pach conjectured that every convex set is cover-decomposable. We verify his conjecture for polygons. Moreover, if $m$ is large enough, depending on $k$ and the polygon, we prove that any $m$-fold covering can even be decomposed into $k$ coverings. Then we show that the situation is exactly the opposite in three dimensions, for any polyhedron and any $m$ we construct an $m$-fold covering of the space that is not decomposable. We also give constructions that show that concave polygons are usually not cover-decomposable. We start the first part with a detailed survey of all results on the cover-decomposability of polygons.\\

The second part of the thesis investigates another geometric partition problem, related to planar representation of graphs. Wade and Chu defined the {\em slope number} of a graph $G$ as the smallest number $s$
with the property that $G$ has a straight-line drawing with edges
of at most $s$ distinct slopes and with no bends.
We examine the slope number of bounded degree graphs. Our main results are that if the maximum degree is at least $5$, then the slope number tends to infinity as the number of vertices grows but every graph with maximum degree at most $3$ can be embedded with only five slopes. We also prove that such an embedding exists for the related notion called  {\em slope parameter}. Finally, we study the {\em planar slope number}, defined only for planar graphs as the smallest number $s$ with the property that the graph has a straight-line drawing in the plane without any crossings such that the edges are segments of only $s$ distinct slopes.
We show that the planar slope number of planar graphs with bounded degree is bounded.\\

{\bf Keywords. } Multiple coverings, Decomposability, Sensor networks, Hypergraph coloring, Graph drawing, Slope number, Planar graphs.

%% file: absf.tex
Nous étudions deux probl\`emes de décomposition de la géométrie combinatoire. La premi\`ere partie de cette th\`ese s'intéresse \`a la décomposition des recouvrements multiples du plan. On dit qu'un ensemble planaire est {\it recouvrement-décomposable}\footnote{{cover-decomposable}} s'il existe une constante $ m $ de telle sorte que tous les $ m $-fois recouvrements du plan avec ses translat\'ees sont décomposables en deux recouvrements disjoints du plan tout entier. Pach a conjecturé que tout ensemble convexe est recouvrement-décomposable. Nous vérifions sa conjecture pour les polygones. De plus, si $ m $ est assez grand, en fonction de $ k $ et du polygone, nous montrons que tous les $ m $-fois recouvrements peuvent \^etre décomposés m\^eme en $ k $ recouvrements. Ensuite, nous montrons qu'en trois dimensions la situation est exactement l'inverse: pour n'importe quel poly\`edre et pour tout $ m $, nous construisons une $ m $-fois recouvrement de l'espace qui n'est pas décomposable. Nous donnons également des constructions qui montrent que les polygones concaves ne sont généralement pas recouvrement-décomposables. Nous commençons la premi\`ere partie avec une étude détaillée de tous les résultats sur la recouvrement-décomposabilité de polygones.\\

La deuxi\`eme partie de la th\`ese étudie un autre probl\`eme de partition géométrique, lié \`a la représentation planaire des graphes. Wade et Chu ont défini le
{\em nombre de pente}\footnote{{slope number}} d'un graphe $ G $ comme le plus petit nombre $ s $ avec la propriété que $ G $ peut \^etre dessin\'e avec des segments ayant au plus $s$ pentes distinctes. Nous examinons le nombre de pente des graphes de degré borné. Nos principaux résultats sont que, si le degré maximum du graphe est d'au moins $ 5 $, alors le nombre de pente tend vers l'infini quand le nombre de sommets croît, mais tout graphe de degré au plus $3$ peut \^etre plongé dans le plan avec seulement cinq pentes. Nous montrons aussi qu'un tel plongement existe pour la notion appelé {\em param\`etre de pente}\footnote{{slope parameter}}. Enfin, nous étudions le {\em nombre de pente planaire}\footnote{{planar slope number}}, défini seulement pour les graphes planaires, comme le plus petit nombre $ s $ avec la propriété que le graphe admet un dessin linéaire dans le plan sans intersections et tel que les segments sont de seulement $ s $ pentes distinctes.
Nous montrons que le nombre de pente planaire des graphes planaires de degré borné est borné.\\

{\bf Mots-cl\'es. } Recouvrements multiples, Décomposabilité, Réseaux de sensors, Coloration des hypergraphes, Dessinage des graphes, Nombre de pente, Graphes planaires.

%% file: thx.tex
{\it ``I'll keep it short and sweet.  Family, religion, friendship.  These 
are the three demons you must slay if you wish to succeed in business.''
\vspace{-.3cm}
\begin{flushright} C. Montgomery Burns\\ \end{flushright} }

\vspace{0.2cm}

I would like to thank all those who in this way or that helped in the creation of the thesis. First of all, of course, my co-authors, Bal\'azs, G\'eza and Jani, without whom the thesis would consist of at most one chapter (my CV). It would be impossible to enumerate how much the three of them helped, so I have to refrain to one sentence each.
Bal\'azs, thank you for buying beer for those unforgettable research sessions at your home, and for writing the first sketch of our proofs that I could rewrite into a paper.
G\'eza, thank you for spending so much of your time with working with me, and for rewriting my always incomprehensible arguments.
J\'anos, thank you for beguiling me into the unabandonable fields of discrete geometry, for inviting my future-in-the-past wife and for rewriting G\'eza's version.

I am also indebted to my supervisor at ELTE, Zolt\'an Kir\'aly. Most papers in this thesis were started when I was your student, and although together we usually worked on other topics, you always encouraged me to also pursue this direction. I could not count how many of my conjectures (some of which were related to cover-decomposition) you have disproved. I would also thank all the others from Budapest who helped my combinatorial career, including but not restricted to Andr\'as, Bal\'azs, Cory, Dani, G\'abor, Gyula, M\'at\'e, Nathan and Zoli.

I would also like to thank Thomas for letting me copy-paste the tex files of his thesis and for showing me the very important procedure of how to attach the pages together. I am sorry that I replaced your text between with mine. I would like to thank all other members and visitors of the DCG and DISOPT groups for creating such a wonderful extraworkular environment. Thank you Adrian, Andreas, Andrew, Dave, Filip, Fritz, Genna, Jarek, Jocelyne, Judit, Laura, Martin, Nicolai, Paul, Rado, Rom, Saurabh and Scruffy. I wish you all the best and many more happy days in Lausanne!

From the third floor, I would also like to thank Dani for the time we spent together, and for simultaneously translating my abstract to french and correcting the mistakes in the english version. I would also thank here Bernadett, Panos and the anonymous referees of my papers for their valuable comments.

I would like to thank all my jury members, especially G\"unter and Pavel, for their careful reading and useful comments, most of which I implemented in this corrected version. 

Finally, I would like to thank my family for their support. Thank you anya, apa, D\'eni\footnote{\scriptsize Since I know that you will read this before you leave for the bike tour, I wanted to remind you to buy some chocolate for my birthday. You know, just because we did not meet it does not mean you can skip my present.
Damn, on the other hand, it is kinda hot, so it would melt. Or will you have a cooling box or something like that with you? You know what, just forget about it and get me something after you are back.}, Eszter, Bence.
But most importantly, I would like to thank my wife, Padmini, who fits to several above mentioned categories. Thank you for leaving me a few seconds each day to work on my research, for reading and spotting so many grammatical mistakes and for all the time we spent together, but maybe the beginning of my thesis is not the best place to discuss this in more detail.\\

And now for something completely different...

%% file: intro.tex
Partitions are one of the best studied and most important notions of combinatorial mathematics. The number theoretic partition function $p(n)$, which represents the number of possible partitions of a natural number $n$, was already studied by Euler. Later many mysterious identities about it were proved by Ramanujan and many more are still studied today, including properties of Young-tableaux. In combinatorics, partitions are often called colorings, which is just another more visual way to imagine the decomposition of a set. Coloring the vertices or edges of a graph is probably the problem that fascinated mathematicians more than any other graph theoretical question, from the four-color conjecture to Ramsey theory. The investigation of Property B\footnote{Named after Felix Bernstein who first studied this property.} was popularized by Erd\H os. A set system is said to have Property B if the elements of its ground set can be colored with two colors such that no set is monochromatic, i.e. each set contains both colors. This is strongly connected to the following geometric problem.\\

Suppose we have a finite number of sensors in a planar region $R$, each monitoring some
part of $R$, called the range of the sensor. 
Each sensor has a duration for which it can be active and once
it is turned on, it has to remain active until this duration is over, after
which it will stay inactive. %The load of a point is the sum of the durations  of all ranges that contain it, and the load of the arrangement of sensors is the minimum load of the points of $R$.
A schedule for the sensors is a starting time for each sensor that determines when it starts to be active.
The goal is to find a schedule to monitor $R$ for as long as we can.
For any instance of this problem, define a set system $\mathcal F$ as follows.
The sensors will be the elements of the ground set of $\mathcal F$ and the points of $R$ will be the sets.
An element is contained in a set if the respective sensor monitors the respective point.
In the special case when the duration of each sensor is $1$ unit of time, we can monitor $R$ for $2$ units of time if and only if $\mathcal F$ has Property B.\\

Pach posed the following related problem. Suppose that every point of the plane is covered by many translates of the same planar set. Is it always possible to decompose this covering into two coverings? Thus our goal is to partition/color the covering sets such that every point will be contained in both parts/color classes.
This question is again equivalent to asking whether certain families have Property B.
Pach conjectured that for every convex set there is a constant $m$ such that any $m$-fold covering is decomposable into two coverings. Such sets are called {\it cover-decomposable}.
The first part of this thesis is centered around this conjecture in the case when the underlying set is a polygon. We show that convex polygons are cover-decomposable. Moreover, if $m$ is large enough, depending on $k$ and the polygon, we prove that any $m$-fold covering can even be decomposed into $k$ coverings. Then we show that the situation is exactly the opposite in three dimensions. For any polyhedron and any $m$, we construct an $m$-fold covering of the space that is not decomposable. We also give constructions that show that concave polygons\footnote{A polygon is {\it concave} if it is not convex.} are not usually cover-decomposable. We start the first part with a detailed survey of all results on the cover-decomposability of polygons.\\

\newpage

In the second part we investigate another geometric decomposition problem related to planar representation of graphs. Partitioning the edges of a graph to obtain nice drawings or to show that the graph is complex in some sense was studied under various constraints.
The {\em thickness} of a graph $G$
is defined as the smallest number of planar subgraphs it can be decomposed into. It is one of the
several widely known graph parameters that measures how
far $G$ is from being planar.
The {\em geometric thickness} of $G$, defined as the
smallest number of {\em crossing-free} subgraphs of a straight-line
drawing of $G$ whose union is $G$, is another similar notion.
In this thesis we investigate a related parameter
introduced by Wade and Chu.
The {\em slope number} of a graph $G$ is the smallest number $s$
with the property that $G$ has a straight-line drawing with edges
of at most $s$ distinct slopes and with no bends.
It follows directly from
the definitions that the thickness of any graph is at most as
large as its geometric thickness, which, in turn, cannot exceed
its slope number.
Therefore the slope number is always an upper bound for the other two parameters.
The slope number is also important for the visualization of graphs.
Graphs with slope number two can be embedded in the plane using only vertical and horizontal segments. Generally, the smaller the slope number is, the simpler the visualization becomes.\\

The second part examines the slope number of bounded degree graphs. Our main results are that if the maximum degree is at least $5$, then the slope number tends to infinity as the number of vertices grow, but every graph with maximum degree at most $3$ can be embedded with only five slopes.
The degree $4$ case remains a challenging open problem.
We also prove that such an embedding exists for the related notion called slope parameter, which is defined in the second part.
Finally, we study the planar slope number of bounded degree graphs. This parameter is only defined for planar graphs. It is the smallest number $s$ with the property that the graph has a straight-line drawing in the plane without any crossings, such that the edges are segments of only $s$ distinct slopes. We show that the planar slope number of planar graphs with bounded degree is bounded.\\

In the third part we summarize the interesting open questions and conjectures about cover-decomposition and the slope number. These are followed by the bibliography and my curriculum vit\ae.

%% file: introtodec.tex
This section mainly follows our manuscript with J\'anos Pach and G\'eza T\'oth, Survey on the Decomposition of Multiple Coverings \cite{PPT10}.

Let ${\cal P}=\{\ P_i\ |\ i\in I\ \}$ be a collection of planar sets.
We say that ${\cal P}$ is an {\em $m$-fold covering} if every point in the plane
is contained in at least $m$ members of $\cal P$. 
The biggest such $k$ is called the thickness of the covering. A $1$-fold covering is simply called a {\em covering}. 

\medskip

\noindent {\bf Definition.} A planar set $P$ is said to be {\em cover-decomposable}
if there exists a (minimal) constant $m=m(P)$ such that
every $m$-fold covering of the plane with translates of $P$
can be decomposed into two coverings.

\medskip

We will also refer to the problem of decomposing a covering as the {\it Cover 
Decomposition problem}.
Pach \cite{P80} proposed the problem of determining all cover-decomposable
sets in 1980 and made the following conjecture. 
%For related problems, conjectures, see \cite{BMP05}, Chapter 2.1.

\medskip

\noindent {\bf Conjecture. (Pach)} {\em All planar convex sets are
cover-decomposable.}

\medskip

This conjecture has been verified for open polygons through a series of papers.

\medskip

\noindent {\bf Theorem A.} (i) \cite{P86} {\em Every centrally symmetric
open convex polygon is cover-decomposable.}

(ii) \cite{TT07} {\em Every open triangle is cover-decomposable.}

(iii) \cite{PT10} {\em Every open convex polygon
is cover-decomposable.}

\medskip

In fact, in \cite{PT10} a slightly stronger result is proved. In particular, it
is shown that the union of finitely many copies of the same open convex
polygon is also cover-decomposable. See also Section \ref{sec:convex} for details.
There are several recent negative results as well.

\medskip

\noindent {\bf Theorem B.} \cite{PTT05} {\em Concave quadrilaterals are not cover-decomposable.}

\medskip

In \cite{PTT05} it was also shown that certain type of concave
polygons are not cover-decomposable either.
This has been generalized to a much larger class of 
concave polygons in \cite{P10}, see
Section \ref{sec:concave} for details. 
One can ask analogous questions in higher dimensions, and in \cite{P10} it is
shown that the situation is quite different.

\medskip

\noindent {\bf Theorem B'.}
\cite{P10} {\em Polytopes are not cover-decomposable in the space and in 
 higher dimensions.}

\medskip

For a cover-decomposable set $P$, one can ask for the exact value of
$m(P)$. In most of the cases, the best known upper and lower bounds 
are very far from each other. For example, for any open triangle $T$
we have $3\le m(T)\le 19$, for the best upper bound see \'Acs \cite{A10}. 

\medskip

\noindent {\bf Definition.} Let $P$ be a planar set and $k\ge 2$ integer. 
If it exists, let $m_k(P)$ denote the smallest number $m$ with the property
that every $m$-fold covering of the plane with translates of $P$
can be decomposed into $k$ coverings.

\medskip

We conjecture that $m_k(P)$ exists for all cover-decomposable $P$, but we cannot prove it in general, see also Question \ref{quest:m3}. 
In \cite{P86} it is shown that for any centrally symmetric 
convex open polygon $P$,
$m_k(P)$ exists and $m_k(P)\le f(k, P)$ where 
$f(k, P)$ is an exponential function of $k$ for any fixed $P$.
In \cite{TT07} a similar result was shown for open triangles and in \cite{PT10} for open convex polygons.
However, all these results were improved to the optimal linear bound in a series of papers.

\medskip

\noindent {\bf Theorem C.} (i) \cite{PT07} {\em 
For any centrally symmetric open convex polygon $P$, $m_k(P)=O(k^2)$.}

(ii) \cite{A08} {\em For any centrally symmetric open convex polygon $P$, $m_k(P)=O(k)$.}

(iii) \cite{GV10} {\em For any open convex polygon $P$, $m_k(P)=O(k)$.}

\medskip

The problem of determining $m_k(P)$ can be reformulated in a slightly different
way:
we try to decompose an $m$-fold covering into as many coverings as possible. 
This problem is closely related to the Sensor Cover problem. 
Gibson and Varadarajan 
in \cite{GV10} proved their result in this more general context. See Section \ref{sec:gv} for details.

The goal of this section is to sketch the methods used in the above theorems,
and to state the most important open problems.
The first paper in this topic was \cite{P86}, and its methods were used by all
later papers.
Therefore, we first concentrate on this paper, and then we turn to the others.

\subsection{Basic Tricks}

Suppose that we have a $k$-fold covering of the plane with a family of translates of an
open polygon $P$. By a standard compactness argument, we can select a
subfamily which still forms a $k$-fold covering, and is locally finite. That is,
each point is covered finitely many times. Therefore, we will assume without
loss of generality that all coverings are locally finite.

\subsubsection{Dualization method}

In \cite{P86} the results are proved in the {\em dual setting}.
Suppose we have a collection
${\cal P}=\{\ P_i\ |\ i\in I\ \}$ of translates of $P$.
Let $O_i$ be the center of gravity of $P_i$.
The collection ${\cal P}$ is a $k$-fold covering of the plane
if and only if every translate of $\bar{P}$, the reflection of $P$ through the
origin, contains at least $k$ points of the collection
${\cal O}=\{\ O_i\ |\ i\in I\ \}$.

The collection ${\cal P}=\{\ P_i\ |\ i\in I\ \}$ can be decomposed into two
coverings
if and only if the set ${\cal O}=\{\ O_i\ |\ i\in I\ \}$ can be colored with
two colors, such that every translate of $\bar{P}$ contains a point of each of the
colors. Note that the cardinality of $I$ can be arbitrary. Using that $P$ and $\bar{P}$ are either both cover-decomposable, or none of them is, we have proved the following.

\begin{lem}\label{dual} $P$ is cover-decomposable if and only if there is an $m$, 
such that given any point set $S$, with the property that any translate of $P$ 
contains at least $m$ points of $S$, can be colored with two colors such that any translate of $P$ contains points of both colors.
\end{lem}

Note that the same argument applies if we want to decompose the covering 
into $k>2$ coverings.
All mentioned papers use the same approach, that is, they all
investigate the covering problem in the dual setting. 
Thus, from now on we will also investigate the problem in the dual setting.

\subsubsection{Divide et impera -- Reduction to wedges}

This approach is also from \cite{P86} and it is used in all papers on the topic.
Two halflines, both of endpoint $O$,
divide the plane into two parts, $W_1$ and $W_2$, which we call
{\em wedges}. 
A {\em closed wedge} contains its boundary, an {\em open wedge}
does not. 
Point $O$ is called the {\em apex} of the
wedges. 
The {\em angle} of a wedge is the angle between its two boundary
halflines, measured inside the wedge.

Let $P$ be a polygon of $n$ vertices
and we have a multiple covering of the plane with translates of $P$.
Then, the cover decomposition problem can be reduced to 
wedges as follows.

Divide the plane into small regions, say, squares, such that
each square intersects at most two consecutive sides of 
any translate of $P$. 
If a translate of $P$
contains sufficiently many points of $S$, then it contains many
points of $S$ in one of the squares, because every translate 
can only intersect a bounded number of squares.
We color the points of
$S$ separately in each of the squares such that if a translate of $P$
contains sufficiently many of them, then it contains points of both colors. 
If we focus
on the subset $S'$ of $S$ in just one of the squares, then any 
translate of $P$ ``looks like'' a wedge 
corresponding to one of the vertices of $P$.
That is, if we consider $W_1, \ldots , W_n$, the wedges corresponding to the vertices 
of $P$, then any subset of $S'$ that can be cut off from $S$ by a translate of 
$P$, can also be cut off by a translate of one of $W_1, \ldots , W_n$.
Note that $S'$ is finite because of the locally finiteness of our original covering.

\begin{lem}\label{wedge} $P$ is cover-decomposable if there is an $m$, 
such that any finite point set $S$ 
can be colored with two colors such that any translate of any wedge of $P$ 
that contains at least $m$ points of $S$, contains points of both colors.
\end{lem}

Again, the same argument can be repeated in the case when we want to decompose a covering into $k>2$
coverings.
Thus, from now on, we will be interested in coloring point sets with respect to wedges when proving positive results.
But in fact coloring point sets with respect to wedges can also be very useful to prove negative results as is shown by the next lemma.

\subsubsection{Totalitarianism}\label{sec:total}
So far our definition only concerned coverings of the {\it whole} plane, but we could investigate coverings of any fixed planar point set.

\begin{defi}\label{deftotal} A planar set $P$ is said to be {\em totally-cover-decomposable}
if there exists a (minimal) constant $m^T=m^T(P)$ such that
every $m^T$-fold covering of ANY planar point set with translates of $P$
can be decomposed into two coverings. Similarly, let $m^T_k(P)$ denote the smallest number $m^T$ with the property
that every $m^T$-fold covering of ANY planar point set with translates of $P$
can be decomposed into $k$ coverings.
\end{defi}

This notion was only defined in \cite{P10}, however, the proofs in earlier papers all work for this stronger version because of Lemma \ref{wedge}. Sometimes, when it can lead to confusion, we will call cover-decomposable sets {\em plane-cover-decomposable}. By definition, if a set is totally-cover-decomposable, then it is also plane-cover-decomposable. On the other hand, there are sets (maybe even polygons) which are plane-cover-decomposable, but not totally-cover-decomposable. E.g. the disjoint union of a concave quadrilateral and a far enough halfplane is such a set. For these sets the following stronger version of Lemma \ref{wedge} is true.

\begin{lem}\label{totallywedge} The open polygon $P$ is totally-cover-decomposable if and only if there is an $m^T$ 
such that any finite point set $S$ 
can be colored with two colors such that any translate of any wedge of $P$ 
that contains at least $m^T$ points of $S$, contains points of both colors.
\end{lem}

Note that if we want to show that a set is not plane-cover-decomposable, then we can first show that it is not totally-cover-decomposable using this lemma for a suitable point set $S$ and then adding more points to $S$ and using Lemma \ref{dual}. Of course, we have to be careful not to add any points to the translates that show that $P$ is not totally-cover-decomposable. This is the path followed in \cite{PTT05} and also in \cite{P10}, but there the point set $S$ cannot always be extended. This will be discussed in detail in Section \ref{sec:concave}.

\subsection{Boundary Methods}

%This part mainly follows \cite{PT07}.

Let $W$ be a wedge, and $s$ be a point in the plane. A translate of $W$
such that its apex is at $s$, is denoted by $W(s)$.
More generally, if $W$ is convex, then for points $s_1, s_2, \ldots s_k$,
$W(s_1, s_2, \ldots s_k)$ denotes the {\em minimal} translate of $W$
(for containment) 
%whose closure
which contains $s_1, s_2, \ldots s_k$.

Here we sketch the proof of Theorem A (i) from \cite{P86}, in the special case
when $P$ is an axis-parallel square.
This square has an {\em upper-left}, {\em lower-left}, {\em upper-right},
and {\em lower-right} vertex.
To each vertex there is a corresponding wedge, whose apex is at this vertex and whose sides contain the sides of the square incident to this vertex.
Denote the corresponding wedges by
$W_{ul}$, $W_{ll}$, $W_{ur}$,
and $W_{lr}$, respectively. We refer to these four wedges as $P$-wedges.
Let $S$ be a finite point set.
By Lemma \ref{wedge} it is sufficient to prove the following.

\begin{lem}\label{negyzet} 
$S$ 
can be colored with two colors such that any translate of 
a $P$-wedge
%$W_{ul}$, $W_{ll}$, $W_{ur}$, and $W_{lr}$ 
which contains at least five points of $S$,
contains points of both colors.
\end{lem}

It will be very useful to define the {\em boundary} of $S$ with respect to the 
wedges of $P$. It is a generalization 
of the convex hull; a point $s$ of $S$ is on the convex hull 
if there is a halfplane which contains $s$ on its boundary,
but none of the points of $S$ in its interior.

\begin{defi}
The {\em boundary} of $S$ with respect to a wedge $W$,
$Bd^W(S)=\{s\in S:W(s)\cap S=\emptyset\}.$
Two $W$-boundary vertices, $s$ and $t$ are {\em neighbors}
if $W(s, t)\cap S=\emptyset$.
\end{defi}

It is easy to see that $W$-boundary points have a natural order
where two vertices are consecutive if and only if 
they are neighbors. Observe also that any translate of $W$ intersects
the $W$-boundary in an interval.
Now the boundary of $S$ with respect to the four $P$-wedges
%$W_{ul}$, $W_{ll}$, $W_{ur}$, $W_{lr}$ 
is the union of the four boundaries. 

The $W_{lr}$- and 
a $W_{ll}$-boundary meets at the highest point of $S$ (the point of maximum
$y$-coordinate, which does not have to be unique, but for simplicity let us suppose it is),
the $W_{ll}$- and 
a $W_{ul}$-boundary meets at the rightmost point,
the $W_{ul}$- and 
a $W_{ur}$-boundary meets at the lowest point,
and the $W_{ur}$- and 
a $W_{lr}$-boundary meets at the leftmost point.
See Figure \ref{hatar}. For simplicity, translates of $P$-wedges, 
$W_{ul}$, $W_{ll}$, $W_{ur}$, $W_{lr}$, are denoted by $W_{ul}$, $W_{ll}$,
$W_{ur}$, $W_{lr}$, respectively.

\begin{figure}[htb]
\begin{center}
\scalebox{0.45}{\includegraphics{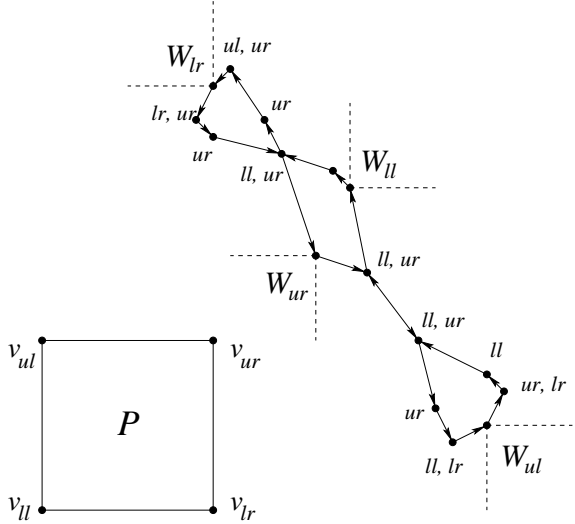}}
\caption{The boundary of a point set.}\label{hatar}
\end{center}
\end{figure}

\smallskip

Points of $S$ which are not boundary vertices, are called {\em interior} points.
The main difference between the convex hull and the boundary,
with respect to $P$, is that
in the cyclic enumeration of all boundary vertices obtained by joining the natural orders on the four parts of the boundary together,
a vertex could occur twice. These are called {\em singular} vertices, the others 
are called {\em regular} vertices.
However, it can be shown that no vertex can appear three times in the 
cyclic enumeration, and all singular vertices have the same type: either all
of them belong to $W_{ul}$ and $W_{lr}$, or all of them belong to $W_{ur}$ and
$W_{ll}$.
This also holds for any centrally symmetric convex polygon, singular boundary
vertices
all belong to the same two opposite boundary pieces.

The most important observation is the following. 

\begin{obs}\label{obs1}
If a translate of a $P$-wedge, 
say, $W_{ll}$, contains some points of $S$, then it is the
union of three subsets: (i) an interval of the boundary which contains at least
one point from the $W_{ll}$-boundary, (ii) an interval of the boundary which contains at least
one point from the $W_{ur}$-boundary, (iii) interior points. 
Note that (i) is non-empty, while (ii) and (iii) could be empty.
Analogous statements hold for the other three wedges, and also for other 
symmetric polygons.
\end{obs}

A first naive attempt for a coloring could be to color all the boundary blue, and 
the interior red. Clearly, it is possible that there is a wedge that contains 
lots of boundary vertices and no interior vertices, so this coloring is not always
good. 
Another naive attempt could be to
color boundary vertices alternatingly red and blue. 
There is an obvious parity-problem here, and a problem with the singular
vertices.
But there is another, more serious problem, that a translate of a wedge could contain just one
boundary vertex, and lots of interior vertices. 
So, we have to say something about the colors of the interior vertices but
this leads to further complications.
It turns out that a ``mixture'' of these approaches works.

\begin{defi}
We call a boundary vertex $s$ {\em $r$-rich} 
if there is a translate $W$ of a $P$-wedge, 
%$W_{ul}$, $W_{ll}$, $W_{ur}$, or $W_{lr}$, 
such that $s$ is the only $W$-boundary vertex in $W$ but 
$W$ contains at least $r$ points of $S$.\footnote{In \cite{P86} and \cite{PT07} a slightly different definition is used, there $s$ is required to be the only vertex from the whole (ant not only from the $W$-) boundary in the translate of $W$. For symmetric polygons both definitions work, but, for example, for triangles only the above given definition can be used.}
\end{defi}

This definition is used in different proofs with a different constant $r$, but when it leads to no confusion, then we simply write rich instead of $r$-rich. In this proof rich means $5$-rich, thus a boundary vertex $s$ is rich if there is a wedge that intersects the $W$-boundary in $s$ and contains at least four other points.\footnote{Note that instead of $r=5$ we could also pick $r=4$ to define rich points in this proof and only the last line would require a little more attention.}

Our general coloring rule will be the following.\\ 
(1) Rich boundary vertices are blue.\\
(2) There are no two red neighbors.\\
(3) Color as many points red as possible, that is, let the set of red points
$R\subset S$ be {\em maximal} under condition (1) and (2).\\
\indent
Note that from (3) we can deduce\\
(4) Interior points are red.\\
\indent
A coloring that satisfies these conditions is called a {\em proper coloring}.
There could be many such proper colorings of the same point set, and for centrally
symmetric 
polygons, each of them is good for us. 
In \cite{P86} an explicit proper coloring is given.

Now we are ready to finish the argument.

\begin{proof}[Proof of Lemma \ref{negyzet}.]
Suppose that $S$ is colored properly and $W$ is a 
translate of a $P$-wedge, 
%$W_{ul}$, $W_{ll}$, $W_{ur}$, or $W_{lr}$, 
such that 
it contains at least five points of $S$.
We can assume without loss of generality that 
$W$ contains exactly five points of $S$.
By Observation \ref{obs1}, $W$ intersects the $W$-boundary of 
$S$ in an interval. 

First we find a blue point in $W$.
If the above interval contains just one point then this point is rich as the wedge contains at least five points, and rich points are blue according to (1).
If the interval contains at least two points, then one of them should be blue according to (2).

Now we show that there is also a red point in $W$. If $W$ contains any interior
point, then we are done according to (4).
So we can assume by Observation \ref{obs1} that $W\cap S$ is the union of two
intervals, and all points in $W$ are blue. 
Since we have five points, one of them, say, $x$, is not the endpoint of any
of the intervals. If it is not rich, then, according to (3),
it or one of its neighbors, all contained in $W$, is red. So, $x$ should be rich. But then there is a translate $W'$
of a $P$-wedge,  
%$W_{ul}$, $W_{ll}$, $W_{ur}$, or $W_{lr}$ 
which contains only $x$ as a boundary vertex, and contains
five points. Using that $S$ is centrally symmetric, it can be shown that
$S\cap W'$ is a proper subset of $S\cap W$, a contradiction, since both contain exactly five points.
This concludes the proof of Lemma \ref{negyzet}.\end{proof}%\hfill $\Box$\\

If we only consider wedges with more points, we can guarantee 
more red points in them. 

\begin{lem}\label{negyzet-k} 
In a proper coloring of $S$, 
any translate of a $P$-wedge
which contains at least $5k$ points of $S$,
contains at least one blue point and at least $k$ red points.
\end{lem}

The proof is very similar to the proof of Lemma \ref{negyzet}, the difference is that now we color $5k$-rich points red and we have to be a little more careful when counting red points, especially because of the possible singular points.
Then, we can recolor red points recursively by Lemma \ref{negyzet-k}, 
and we obtain an exponential upper bound on $m_k(P)$. 
Analogous statement holds for any centrally symmetric open convex polygon,
therefore, we have

\begin{thm}\label{negyzet-expon} 
For any  symmetric open convex polygon $P$, there is a $c_P$ such that
any $c_P^k$-fold covering of the plane with translates of $P$ can be
decomposed into $k$ coverings.
\end{thm}

\subsubsection{Decomposition to $\Omega(\sqrt{m})$ parts for symmetric polygons}
%{PT07}

Here we sketch the proof of Theorem C (i), following the proof of \cite{PT07}, which is a modification of the previous proof. 
We still assume for simplicity that
$P$ is an axis parallel square. 
The basic idea is the same as in the previous proof.
Let $k\ge 2$. We will color $S$ 
with $k$ colors such that any $P$-wedge that contains 
at least $m=18k^2$ points contains all $k$ colors.
We define $k$ boundary layers and denote them by $B_1, B_2, \ldots , B_k$,
respectively.
That is, 
denote the boundary of $S$ by $B_1$ and let $S_2=S\setminus
B_1$.
Similarly, for any $i< k$,
once we have $S_{i}$,  let $B_i$ be the boundary of $S_{i}$ and let $S_{i+1}=S\setminus
B_{i}$.
Boundary layer $B_i$ will be ``responsible'' for color $i$.
Color $i$ takes the role of blue from the previous proof,
while red points are distributed ``uniformly'' among the other $k-1$ colors.

Slightly more precisely, 
a vertex $v\in B_i$ is rich if there is a translate of a 
$P$-wedge that intersects
$S_{i}$ in at least $18k^2-18ki$ points, and $v$ is the only boundary vertex
in it. We color rich vertices of $B_i$ with color $i$, and color first the remaining singular, then the regular points periodically: $1, i, 2, i, \ldots, k, i, 1, \ldots$
The main 
observation is that 
if a $P$-wedge intersects $B_i$ (for any $i$) in at least $18k$ points, then it contains a
long interval which contains a 
point of each color.\footnote{This $18$ could be improved with a more careful analysis.}
Otherwise, it has to intersect 
each of the boundary layers, but then for each $i$, its intersection with
$B_i$ contains a rich point of color $i$.

\subsubsection{Triangles}
%{TT07}

The main difficulty with non-symmetric polygons is that Observation \ref{obs1} does
not 
hold here; the intersection with a translate of a $P$-wedge is not the union
of two boundary intervals and some interior points. 
In the case of triangles Tardos and T\'oth \cite{TT07} managed to overcome 
this difficulty, with a particular version of a proper coloring,
thus proving Theorem A (ii), 
we sketch their proof in this section. For other polygons a
different approach was necessary, we will see it later.

Suppose that $P$ is a triangle with vertices $A$, $B$, $C$. 
There are three $P$-wedges, $W_A$, $W_B$, and $W_C$.
We define the boundary just like before, it has three parts, the $A$-, $B$-, and
$C$-boundary, each of them is an interval in the cyclic enumeration of the 
boundary vertices. 
Here comes the first difficulty, there could be a singular
boundary vertex which appears three times in the cyclic enumeration of
boundary vertices, once in each boundary.
It is easy to see that there is at most one such vertex, and we can get rid of
it by 
decomposing $S$ into at most four subsets, such that in each of them
singular boundary points all belong to the same two boundaries, just like in
the case of centrally symmetric polygons. 
For simplicity of the description, assume that $S$ has only regular boundary vertices.

\begin{figure}[ht]
\centering
		\includegraphics[scale=0.36]{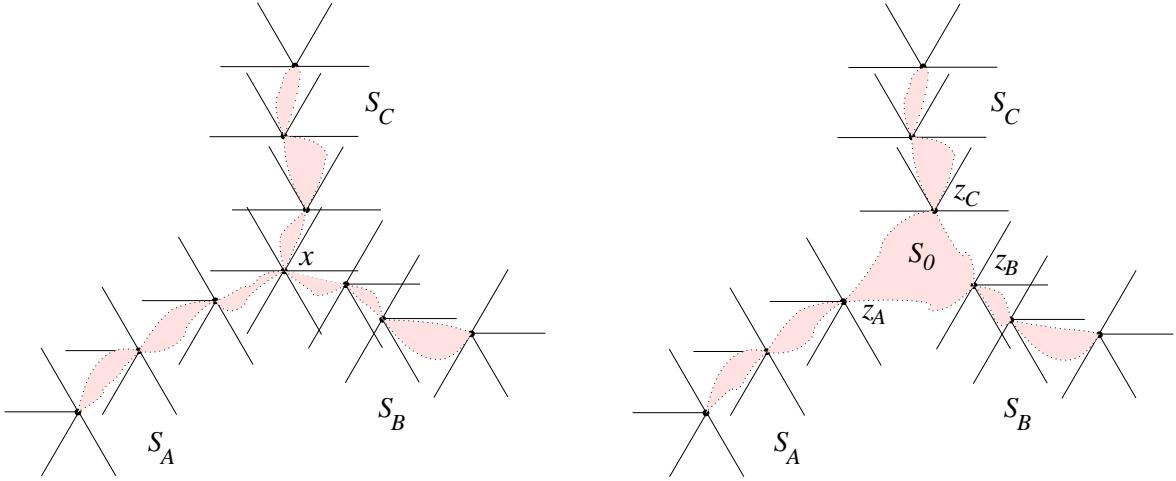}              
		 \caption{On the left, $x$ is singular, on the right, there are only regular boundary vertices.}
		\label{fig:sing}
\end{figure}

Again, we call a boundary vertex $s$ rich 
if there is a translate $W$ of a $P$-wedge, 
such that $s$ is the only $W$-boundary vertex in $W$ but 
$W$ contains at least five points of $S$.

Our coloring will still be a proper coloring, that is\\ 
(1) Rich boundary vertices are blue.\\
(2) There are no two red neighbors.\\
(3) Color as many points red as possible, that is, let the set of red points
$R\subset S$ be {\em maximal} under condition (1) and (2).\\
(4) Interior points are red.

But in this case, we will describe explicitly, how to obtain the set of red
points.
The coloring will be a kind of greedy algorithm.
Consider the 
linear order on the lines of the plane that are parallel to the side $BC$,
so that 
the line through $A$ defined smaller than the line $BC$. 
We define the partial
order $<_A$ on the points with $x<_Ay$ if the line through $x$ is smaller than
the line through $y$. We have $A<_AB$ and $A<_AC$.
Similarly define the partial order $<_B$ according to the lines parallel
to $AC$ with $B<_BC$ and  $B<_BA$, 
and the partial order $<_C$ according to the lines
parallel to $AB$ with $C<_CA$ and $C<_CB$.

First, color all rich boundary vertices blue.
Now 
take the $A$-boundary vertices of $S$ 
and consider them in
{\em increasing} order according to $<_A$. If we get to a point that is not colored,
we color it red and we color every neighbor of it blue. 
These neighbors may have already been
colored blue (because they are rich, or because of an earlier red neighbor) but
they are not colored red since 
any neighbor of any red point is immediately
colored blue. 
Continue, until all of the $A$-boundary is colored.
Color the $B$- and $C$-boundaries similarly, using the other two partial
orders.

Suppose that $W$ is a 
translate of a $P$-wedge, 
such that 
it contains at least five points of $S$.
We can assume without loss of generality that 
$W$ contains exactly five points of $S$.
Assume that $W$ is a translate of $W_A$. The other two cases 
are exactly the same. 
To find a blue point, we proceed just like in the previous section, and it
works for any proper coloring. 
We know that $W$ intersects the $A$-boundary of 
$S$ in an interval. 
If this interval contains just one point, then it is rich, so it is blue.
It the interval contains at least two points, then one of them should be blue.

Now we show that there is also a red point in $W$. If $W$ contains any interior
point, then we are done. Therefore, we assume that all five points in $W$ are
boundary vertices. 
Since there are five points in $W$, one of them, say, $x$,
is {\em not} (i) the first or last $A$-boundary vertex in $W$, and 
(ii) not the $<_A$-minimal $B$-boundary point in $W$, and (iii)
     not the $<_A$-minimal $C$-boundary point in $W$.

Suppose that $x$ is rich. Then there is a translate $W'$
of a $P$-wedge,  
which contains only $x$ as a boundary vertex, and contains
five points. It can be shown by some straightforward geometric observations that
$S\cap W'$ is a proper subset of $S\cap W$, a contradiction, since $S\cap W$ both contain five points. So, $x$ can not be rich. 
But then why would it be blue? 
The only reason could be that in the coloring process one of its neighbors on
the boundary,
$y$,
was colored red earlier. But then again,  some geometric observations 
show that $y\in W$, which shows that there is a red point in $W$.
This concludes the proof.

The same idea works if we have 
singular boundary vertices which all belong to, say, to the $A$- and $B$-boundaries. The only difference is that we have to synchronize the coloring
processes on the $A$- and $B$-boundaries, so that we get to the common vertices at the same time. 

By a slightly more careful argument we obtain

\begin{lem}\label{haromszog-k} 
The points of $S$ can be colored with red and blue such that
any translate of a $P$-wedge which contains at least $5k+3$ of the points,
contains a
blue point and at least $k$ red points.
\end{lem}

If we apply Lemma \ref{haromszog-k} recursively, we get an exponential 
bound on $m_k(P)$.

\begin{lem}\label{expon-haromszog} 
For any open triangle $P$, every $\frac{7\cdot\: 5^k - 15}{20}$-fold covering 
of the plane with translates of $P$ can be decomposed into $k$ 
coverings. 
\end{lem}

\subsection{Path Decomposition and Level Curves}

In this chapter we present two generalizations of the boundary method that are used to prove the other positive theorems, Theorem A (iii), C (ii) and C (iii).

\subsubsection{Classification of wedges}\label{sec:class}
%{PT10}

In order to prove Theorem A (iii), that says all open convex polygons are cover-decomposable, in \cite{PT10} some new ideas were developed. 
In the previous results we colored a point set 
with respect to $P$-wedges, for some polygon $P$. 
In this paper, point sets are colored with respect to an {\em arbitrary} 
set of wedges. 

\medskip

\noindent {\bf Definition.}
Suppose that ${\cal W}=\{\ W_i\ |\ i\in I\ \}$ is a collection of wedges.
$\cal W$ is said to be {\em non-conflicting} or simply NC, if there is a
constant $m$ with the
following property.
Any finite set of points $S$ can be colored with two colors such that
any translate of a wedge $W\in {\cal W}$ that contains at least $m$ points of
$S$, contains points of both colors.\footnote{Note that if a collection of wedges, ${\cal W}=\{\ W_i\ |\ i\in I\ \}$ is NC, then so is ${\cal W}=\{\ W_i^*\ |\ i\in I\ \}$ where $W_i^*$ is the closure or interior of the wedge $W_i$. This is true because if we perturbate any $S$ such that the segment determined by any two points becomes non-parallel to any side of any of the wedges, then the collection of sets of points that can be cut off from $S$ by a translate of a wedge from ${\cal W}$ will not decrease.}

\medskip

It turns out that a single wedge is always NC. 
Then pairs of wedges which are NC are characterized.
Finally, it is shown that a set of wedges is NC if and only if 
each pair is NC. From this characterization it follows directly that
for any convex polygon $P$, the set of $P$-wedges is NC.

\begin{lem}\label{1wedge} 
A single wedge is NC.
\end{lem}

A very important tool in the the proof of Lemma \ref{1wedge},
and the following lemmas, is the {\em path decomposition}
which is the generalization of the concept of the boundary.
We give the proof of  Lemma \ref{1wedge} to illustrate this method.

\medskip

\noindent {\bf Proof.}
Let $S$ be a finite point set and $W$ a wedge. 
We prove the statement with $k=3$, that is, 
$S$ can be colored with two colors such that
any translate of $W$ that contains at least $3$ points of
$S$, contains a point of both colors. 
Suppose first that the angle of $W$ is at least $\pi$. Then $W$ is the union
of two halfplanes, $A$ and $B$.
Take the translate of $A$ (resp. $B$) that contains exactly two points of $S$,
say,
$A_1$ and $A_2$ (resp. $B_1$ and $B_2$).
There might be coincidences between $A_1$, $A_2$ and $B_1$, $B_2$, but 
still, 
%therefore the set $\{ A_1, A_2, B_1, B_2 \}$ contains two, three, or four
%different points. But in any case, 
we can color the set $\{ A_1, A_2, B_1, B_2 \}$
such that $A_1$ and $A_2$ (resp. $B_1$ and $B_2$) are of different colors.
Now, if a translate of $W$ contains three points, it contains either
$A_1$ and $A_2$, or $B_1$ and $B_2$, and we are done.
%\dom{egy abra ami mutatja, hogy harom kell}

\smallskip

Suppose now that the angle of $W$ is less than $\pi$. 
%We give two proofs in
%this case, since we will apply the ideas of both proofs later.
We show that in this case the NC property holds with $k=2$. 
We can assume that the positive $x$-axis is in $W$, %, and that no two points have the same $y$-coordinate. Both of these
this can be achieved by an appropriate rotation. For simplicity, also suppose that no direction determined by two points of $S$ is parallel to the sides of $W$ as with a suitable perturbation this can be achieved. 

For any fixed $y$, let $W(2;y)$ be the translate of $W$
which\\
(1) contains at most two points of $S$,\\
(2) its apex has $y$-coordinate $y$, and\\
(3) its apex has minimal $x$-coordinate.\\
It is easy to see that for any $y$, $W(2;y)$ is uniquely defined.
Examine, how $W(2;y)$ changes as $y$ runs over the real numbers.
If $y$ is very small (smaller than the $y$-coordinate of the points of $S$),
then  $W(2;y)$ contains two points, say $X$ and $Y$, and one more, $Z$, on its boundary.
As we increase $y$, the apex of $W(2;y)$ changes continuously.
How can the set $\{ X, Y\}$, of the two points in $W(2;y)$ change?
For a certain value of $y$, one of them, say, $X$, moves to the boundary. At
this point we have $Y$ inside, and two points, $X$, and $Z$ on the boundary.
If we slightly further increase $y$, then $Z$ {\em replaces} $X$, that is,
$Y$ and $Z$ will be in $W(2;y)$ (see Figure \ref{fig:replace}). As $y$ increases to infinity,
the set $\{ Z, Y\}$ could change several times, but each time it changes in the above described manner.
Define a directed graph whose vertices are the points of $S$, and there is an edge from $u$ to 
$v$ if $v$ replaced $u$ during the procedure.
We get two paths, $P_1$ and $P_2$. The pair $(P_1, P_2)$ is called
the {\em path decomposition of $S$ with respect to $W$, of order two} (see Figure \ref{fig:path2}).

\begin{figure}[htb]
\begin{center}
\scalebox{0.4}{\includegraphics{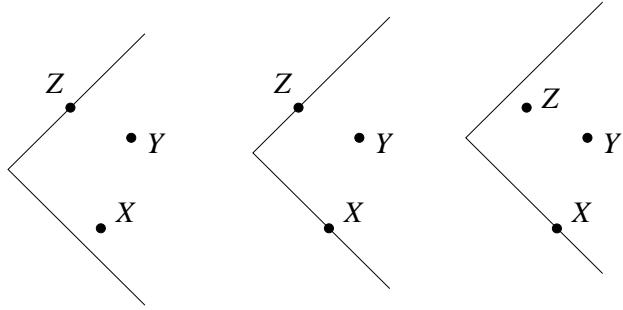}}
\caption{$Z$ replaces $X$ in $W(2; y)$.}
\label{fig:replace}
\end{center}
\end{figure}

Color the vertices of $P_1$ red, the vertices of $P_2$ blue.
Observe that each translate of $W$ that
contains at least two points, contains at least one vertex of both $P_1$ and
$P_2$. This completes the proof. \hfill $\Box$

\begin{figure}[htb]
\begin{center}
\scalebox{0.8}{\includegraphics{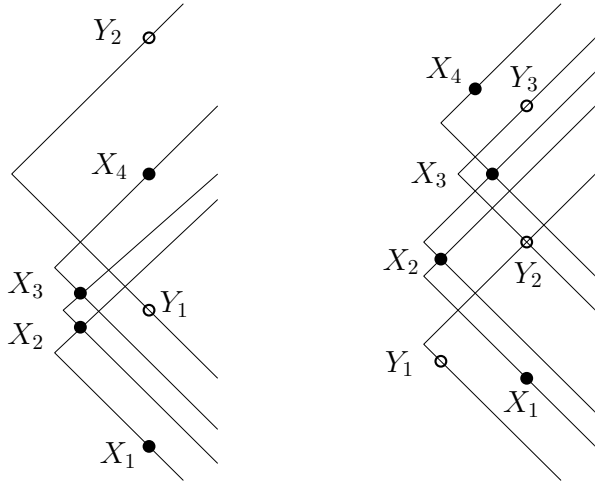}}
\caption{Path decompositions of order two. $P_1=X_1X_2\ldots $,
  $P_2=Y_1Y_2\ldots $.}
  \label{fig:path2}
\end{center}
\end{figure}

We can define the {\em path decomposition of $S$ with respect to $W$, of order
$k$} very similarly.
Let $W(k;y)$ be the translate of $W$
which (1) contains at most $k$ points of $S$, (2) its apex has $y$-coordinate
$y$, and (3) its apex has minimal $x$-coordinate.
Suppose that for $y$ very small, $W(k;y)$ contains the 
points $r_1, r_2, \ldots, r_k$, and at least one more on its boundary.
Just like in the previous description, as we increase $y$, the set 
$\{ r_1, r_2, \ldots, r_k \}$ changes several times, such that 
one of its elements is {\em replaced} by some other vertex.
Define a directed graph on the vertices of $S$ such that there is 
an edge from $r$ to $s$ if $s$ replaced $r$ at some point. 
We get the union of $k$ directed paths, $P^W_1$, $P^W_2$, $\ldots$, $P^W_k$, 
which
is called 
the {\em order $k$ path decomposition of $S$ with respect to $W$}.
Note that the order $1$ path decomposition is just the $W$-boundary of $S$, so this notion is a generalization of the boundary.\footnote{But in general, $P_1(r)$ is not the boundary for higher order path decompositions! Although the union of the paths contains the boundary, the points of the boundary do not necessarily form a path.}

\begin{obs}\label{obs2}
(i) Any translate of $W$ contains an interval of each of 
$P^W_1$, $P^W_2$, $\ldots$, $P^W_k$, and (ii) if a translate of $W$ 
contains $k$ points of $S$, then it contains exactly one point of each of
$P^W_1$, $P^W_2$, $\ldots$, $P^W_k$.
\end{obs}

Now we investigate the case when we have
{\em two} wedges.
We distinguish several cases according to the relative
position of the two wedges, $V$ and $W$.

\smallskip

{\bf Type 1 (Big):} One of the wedges has angle at least $\pi$.

\smallskip

For the other cases, we can assume without loss of generality that
$W$ contains the positive $x$-axis.
Extend the boundary halflines of $W$ to lines, they divide the plane
into four parts, Upper, Lower, Left, and Right, which latter is $W$ itself.
See Figure \ref{fig:wedgeW}.

\begin{figure}[p]
\begin{center}
\scalebox{0.4}{\includegraphics{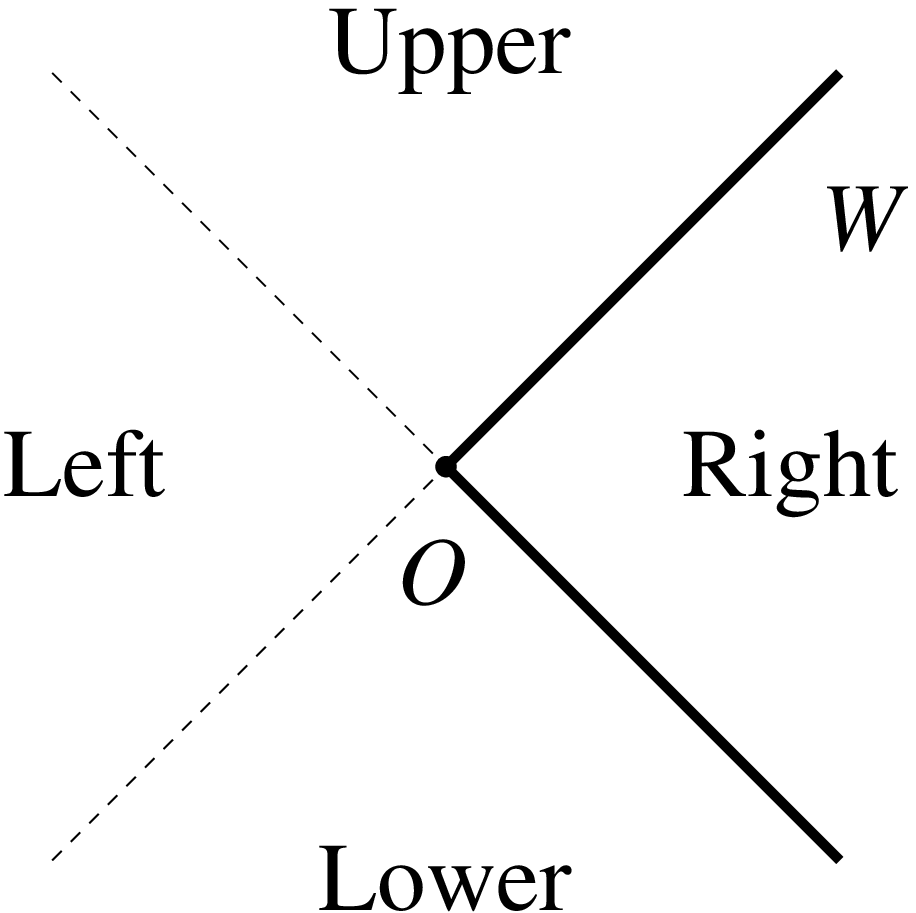}}
\caption{Wedge $W$}
\label{fig:wedgeW}
\end{center}
\end{figure}

\smallskip

{\bf Type 2 (Halfplane):} One side of $V$ is in Right and the other one is in
Left. That is, the union of the wedges cover a halfplane. See Figure \ref{fig:type2}.

\begin{figure}[p]
\begin{center}
\scalebox{0.5}{\includegraphics{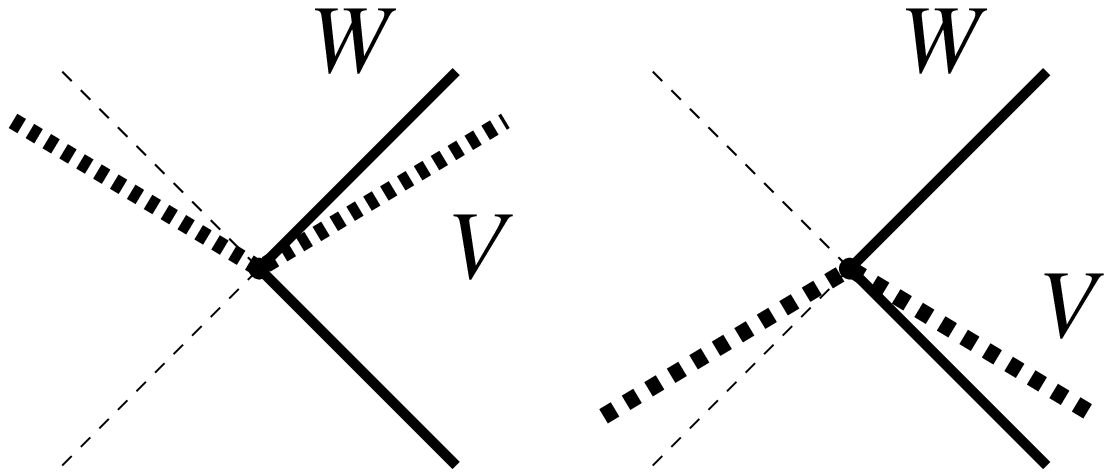}}
\caption{Type 2 (Halfplane)}
\label{fig:type2}
\end{center}
\end{figure}

\smallskip

{\bf Type 3 (Contain):}
Either (i) one side of $V$ is in Upper, the other one is in Lower, or
(ii) both sides are in Right or (iii) both sides are in Left. See Figure \ref{fig:type3}. 

\begin{figure}[p]
\begin{center}
\scalebox{0.5}{\includegraphics{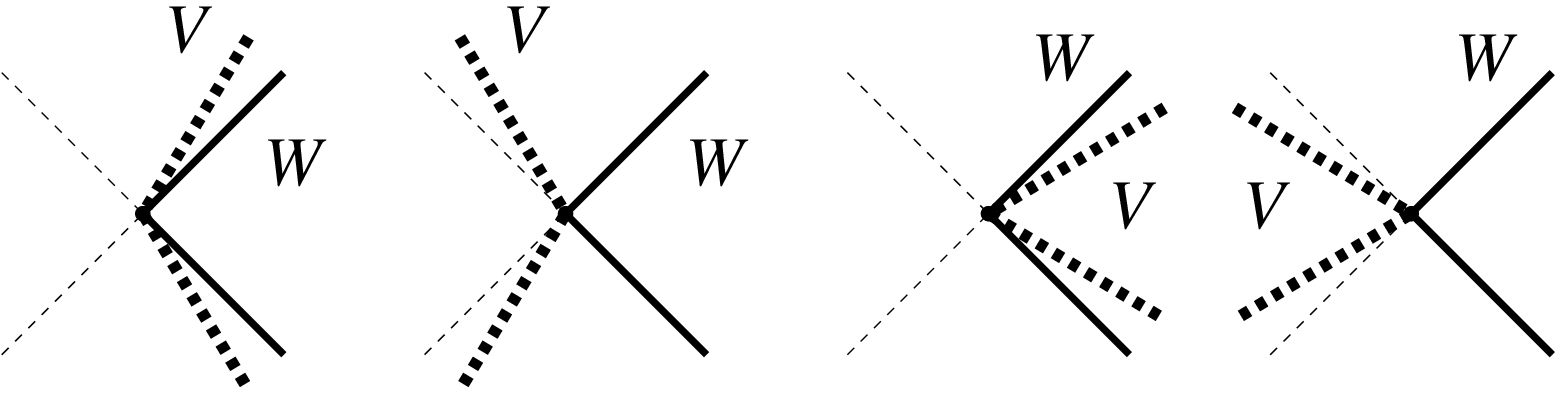}}
\caption{Type 3 (Contain)}
\label{fig:type3}
\end{center}
\end{figure}

\smallskip

{\bf Type 4. (Hard):} One side of $V$ is in Left and the other one is in Upper or
Lower.
This will be the hardest case. See Figure \ref{fig:type4}. 

\begin{figure}[p]
\begin{center}
\scalebox{0.5}{\includegraphics{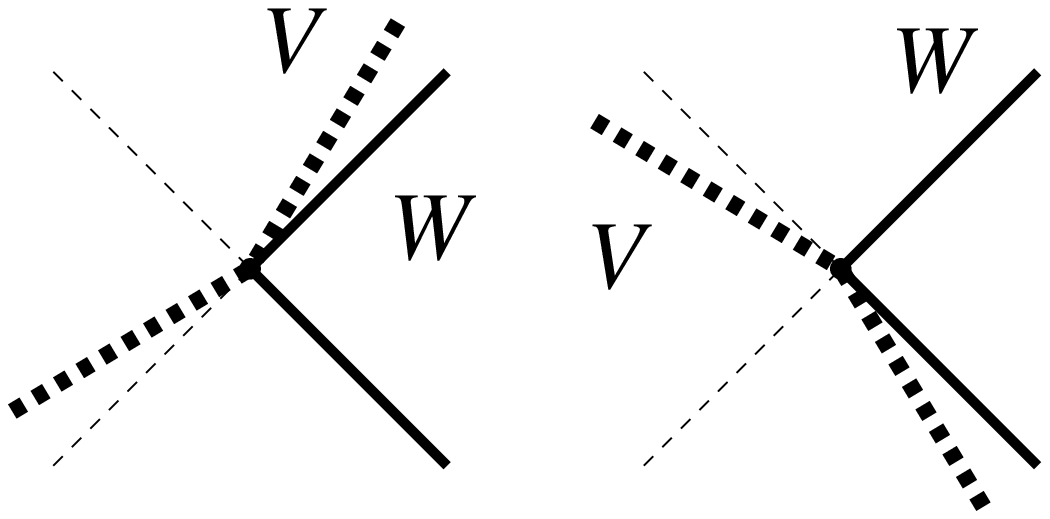}}
\caption{Type 4 (Hard)}
\label{fig:type4}
\end{center}
\end{figure}

\smallskip

{\bf Type 5. (Special):} Either (i)
one side of $V$ is in Right and the other one is in Upper or Lower,
or (ii) both sides are in Upper, or (iii) both sides are in Lower.
That is, the union of the wedges is in an open
halfplane whose boundary contains the origin, but
none of them contain the other. See Figure \ref{fig:type5}. 

\begin{figure}[p]
\begin{center}
\scalebox{0.5}{\includegraphics{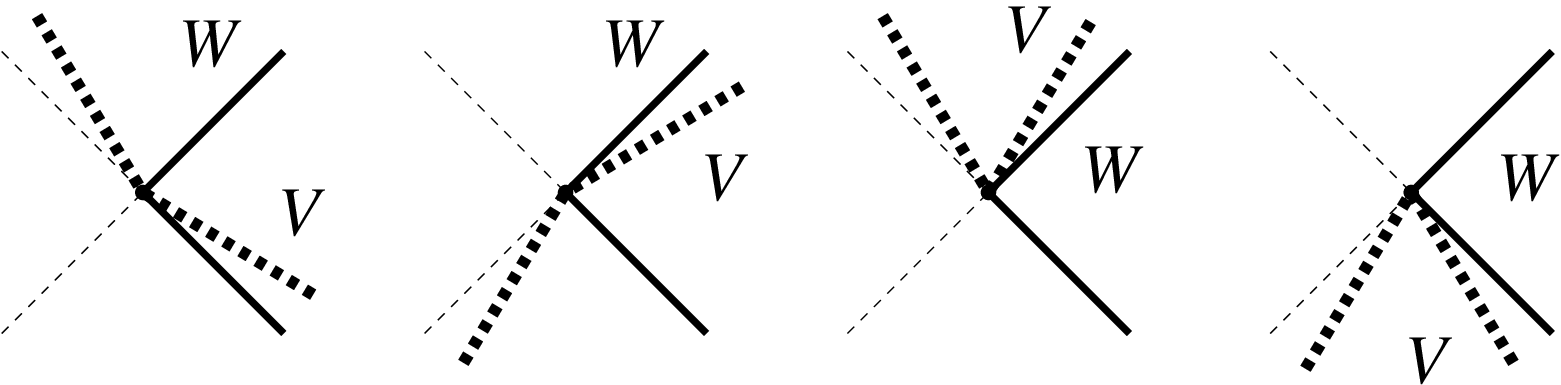}}
\caption{Type 5 (Special)}
\label{fig:type5}
\end{center}
\end{figure}

\smallskip

It is not hard to see that there are no other possibilities.

\begin{lem}\label{2wedge}
Let ${\cal W}=\{ V, W\}$ be
a set of two wedges, of Type 1, 2, 3, or 4.
Then $\cal W$ is NC.
\end{lem}

This lemma is proved in Section \ref{sec:convex}, for each case separately.
It is also shown in \cite{P10} that if ${\cal W}=\{ V, W\}$ is a set of two wedges
of Type 5 (Special), then $\cal W$ is {\em not} NC. For the proof and its
consequences see Section \ref{sec:concave}.
In case of several wedges we have

\begin{lem}\label{sokwedge}
A set of wedges ${\cal W}=\{ W_1, W_2, \ldots ,
W_t \}$ is NC if and only if any pair $\{W_i, W_j \}$ is NC.
\end{lem}

It is obvious that if two wedges are not NC then ${\cal W}$ 
can not be NC. Therefore, a set of wedges is NC if and only if none of the pairs is of Type 5 (Special).
The proof of Lemma \ref{sokwedge} can again be found in Section \ref{sec:convex}. In fact, a somewhat stronger statement is true. At the end of Section \ref{sec:convex} it is shown that if ${\cal W}$ is NC, then for any $k$ there is an $m_k$ such that any finite point set can be colored with $k$ colors such that if a translate of a wedge from ${\cal W}$ contains at least $m_k$ points, then it contains all $k$ colors.

To finish the proof of Theorem A (iii), 
observe that two wedges corresponding to the vertices 
of a convex polygon cannot be of Type 1 (Big) or of Type 5 (Special).
A summary of the whole proof of the theorem can be found at the end of Section \ref{sec:convex}.

\subsubsection{Level curves and decomposition to $\Omega(k)$ parts for symmetric polygons}

The {\em level curve} method was invented by Aloupis et. al. \cite{A08} at the same time and independently from the path decomposition. 
Again suppose that the angle of $W$ is less than $\pi$ and $W$ contains the positive $x$-axis.
Now define the level curve of depth $r+1$, $\C(r)$, as the collection of the apices of $W(r;y)$.\footnote{In \cite{A08} they denote this by $\C(r+1)$ because they work with closed wedges.}
Another equivalent way to define $\C(r)$ is as the boundary of the union of all the translates of $W$ containing at most $r$ points. 

Note that this curve consists of straight line segments that are parallel to the sides of $W$.
$\C(0)$ goes through all the boundary points, this shows that this notion is a generalization of the boundary.
If $p\in\C(r)$, then $|W(p)\cap S|$ is either $r$ or $r-1$ and it is $r-1$ only a finite number of times, when the respective translate has a point on both of its sides.

Now with the level curve method, we prove Theorem C (ii), as was done in \cite{A08}.

Suppose our symmetric polygon $P$ has $2n$ vertices.
We denote the $P$-wedges belonging to them in clockwise order by $W_0,\ldots, W_{2n-1}$. All the indices should be considered in this section modulo $2n$.
We call two wedges $W_i$ and $W_j$ {\em antipodal} if $i+n\equiv j$ modulo $2n$, that is, if they are the wedges belonging to two opposite vertices of the polygon.
A crucial observation is (already used in \cite{P86}) that any two $P$-wedges that are not antipodal, cover a half-space.

For every side of $P$, take two lines parallel to it that cut off $2r+2$ points from each side of $S$. Denote the intersection of the $n$ stripes formed by these lines by $\T$. Any large enough wedge has to intersect $\T$, thus it is enough to care about the wedges whose apex lies in $\T$. Now if we consider the level curves $\C^{W_i}(r)$, a simple geometric observation shows that only level curves belonging to antipodal wedges may cross inside $\T$ and some further analysis shows that in fact there can be only one such pair (note the similarity to the singular points in case of symmetric polygons). This means that the regions cut off from $\T$ by the curves $\C^{W_i}(r)$ are all disjoint with the possible exception of one pair. Without loss of generality, these are the curves of $W_0$ and $W_n$.

Another easy observation shows that any translate of $W_i$ that contains at least $3r+5$ points, must contain a point from $\C^{W_i}(r)\cap\T$, thus also a translate of $W_i$ whose apex is on the level curve inside $\T$, containing $r$ points from $S$. Therefore it is enough to care about these wedges, whose apex lies on the respective level curve.
It is possible to parametrize these wedges with the circle parameterized by $[0,2n)$ such that $W(t)$ is a translate of $W_{\left\lfloor t\right\rfloor}$.
A crucial geometric observation is that if $p\in W(\left\lfloor t\right\rfloor+x)\cap W(\left\lfloor t\right\rfloor+z)$, where $0\le x\le 1$ and $0\le z\le n$, then $p\in W(\left\lfloor t\right\rfloor+y)$ for all $x\le y\le z$.
If $p\in W(\left\lfloor t\right\rfloor+x)\cap W(\left\lfloor t\right\rfloor+z)$, where $0\le x\le 1$ and $n\le z\le n+1$, then $p$ is contained in two antipodal wedges implying that it is contained in translates of $W_0$ and $W_n$ but in no translates of any other other wedge from $W(t)$.
Therefore, every $p$ is contained either in an interval of the circle $[0,2n)$, or in two intervals, one of which is a subinterval of $[0,1]$, the other of $[n,n+1]$.
The simplest is if we take care of these two types separately, as any big wedge contains a lot of points from one of these groups. The first type forms a circular interval graph, if every point of the circle is covered $m'$-fold, then we can decompose this to $m'/3$ coverings with a simple greedy algorithm. In the second type, we want to color points with respect to a wedge and its rotation with $180$ degrees. The greedy algorithm again gives a good decomposition from an $m''$-fold covering into $m''/3$ coverings. Putting the numbers together this implies that $m_k\le 18k+5$ for any system of wedges derived from a symmetric polygon. This has to be multiplied by a constant depending on the shape of the polygon that comes from Lemma \ref{wedge} to get a bound for the multiple-cover-decomposability function $m_k$ of the polygon.

\subsubsection{Decomposition to $\Omega(k)$ parts for all polygons}\label{sec:gv}
%{GV10}
The decomposition to multiple coverings is also motivated by the following problem, called {\it Sensor Cover problem}.

Suppose we have a finite number of sensors in a region $R$, each monitoring some
part of $R$, which is called the range of the sensor. 
Each sensor has a duration for which it can be active and once
it is turned on, it has to remain active until this duration is over, after
which it will stay inactive. The load of a point is the sum of the durations 
of all ranges that contain it, and the load of the 
arrangement of sensors is the minimum load of the points of $R$.
A schedule for the sensors is a starting time for each sensor that determines when it starts to be active.

The goal is to find a schedule to 
monitor the given area, $R$, for as long
as we can. Clearly, the cover decomposability problem is a special case of 
the Sensor Cover problem, when the duration of each sensor is the same. Gibson and Varadarajan 
in \cite{GV10} proved their result in this more general context.

\medskip

\noindent {\bf Theorem D.} \cite{GV10} {\em For any open convex polygon $P$ there is a $c(P)$ such that for any instance of the Sensor Cover problem with load $k\cdot c(P)$
where each range is a translate of $P$,
there is a polynomial time computable schedule such that every point is monitored for $k$ time units.}

\medskip

In the special case where the duration of each sensor is $1$ unit of time and $R$ is the whole plane, this is equivalent to Theorem C (iii). %$m_k(P)=O(k)$ where the constant hidden in the $O$ notation depends on $P$.
As the proof is essentially the same, we will only sketch the proof for this special case to avoid changing terminology. In their proof they use the usual dualization and reduction to wedges, because of which it is enough to prove the following theorem (for the special case).

\medskip

\noindent {\bf Theorem D'.} \cite{GV10} {\em If ${\cal W}=\{\ W_i\ |\ i\in [n]\ \}$ is a system of $P$-wedges, then there is an $\alpha_n$ depending only on $n$, such that any point set $S$ can be colored with $k$ colors such that any translate of a wedge from ${\cal W}$ that contains at least $\alpha_n k$ points, contains all $k$ colors.}

\medskip

Note that any two $P$-wedges are of Type 2 (Halfplane), 4 (Hard) or a special case of 3 (Contain). Their main lemma is the following easy observation. 
%(can be already found in \cite{A08}).

\begin{lem}\label{partialcolor} For any point set $Q\subset P$, any wedge $W$, any $k$ and any $L\ge 2k$, we can partially color the points of $Q$ with $k$ colors such that any translate of $W$ that contains $L$ points of $P$ and at least $2k$ points of $Q$ contains at most $2k$ colored points but contains all $k$ colors. Moreover, if a point $z$ is colored, then all points in $Q\cap W(z)$ are colored.
\end{lem}
\begin{proof} In every step take a point from $Q$ that covers a maximal, yet uncovered interval of $\C^{W}(L)$\footnote{Here and later, the level curves are always with respect to $P$ and not to $Q$.} until the whole curve is covered, then color these points with one color and repeat.
\end{proof}

The trick is that we obtain a partial coloring using this lemma for a carefully chosen subset of $P$, any one of the wedges, $W_1\in{\cal W}$, $k$ and an $L=f(n) k$ (constant depending on $n$ to be specified later) such that $\alpha_{n-1} k$ points remain uncolored in any translate of any wedge from ${\cal W'}={\cal W}\setminus W_1$ that previously contained at least $\alpha_n k$ points. After applying this partial coloring $n$ times, we are done.

Before we can specify $Q$, we need to define an order on the plane for every line that is parallel to the side of a wedge. (Thus together this gives at most $2n$ orders for general wedges, for $P$-wedges it would give $n$.) The order is very similar to the one used for triangles in \cite{TT07}. For a wedge $W$ and a line $\ell$ parallel to one of its sides, define $p<_\ell q$ if the line parallel to $\ell$ through $q$ intersects $W(p)$. (So in the special case of $P$-wedges, $p_i$ and $p_{i+1}$ are the minimal vertices of $P$ according to $<_{p_ip_{i+1}}$.) For simplicity, we just refer to these orders as the $2n$ orders defined by the system ${\cal W}$.

Now we can define $Q$.
A point $p\in P$ is in $Q$ if there is a translate of $W_1$ containing exactly $L$ points from $P$ in which $p$ is not among the first $\alpha_{n-1} k$ points in any of the $2n-2$ orders defined by ${\cal W'}$.
Now let us apply Lemma \ref{partialcolor} to this $Q$, $W_1$, $k$ and $L=((2n-1)\alpha_{n-1}+6) k$.
Note that each translate of $W_1$ whose apex lies on the curve $\C^{W}(2k)$ will contain at least $2k$ points of $Q$ as $L\ge ((2n-2)\alpha_{n-1}+2) k$.%, we will choose $L$ this way.

\begin{claim} If $W_i\in {\cal W'}$ contains $\alpha_n k$ points from $P$, where $\alpha_n\ge 3\alpha_{n-1}+6$, then it contains $\alpha_{n-1} k$ uncolored points after applying the coloring of Lemma \ref{partialcolor} to $Q$, $W_1$, $k$ and $L=((2n-1)\alpha_{n-1}+6) k$.
\end{claim}
\begin{proof}
The proof depends on the type of $W_1$ and $W_i$. First suppose they are of Type 2 (Halfplane). Take a translate of $W_i$, $W_i(x)$ containing $\alpha_n k$ points of $P$. If it does not intersect the level curve $\C^{W_1}(L)$, then we did not color any of its points, we are done. If it intersects this level curve, then the intersection can be only one point, $z$. Moreover, $W_1(z)$ contains all the colored points contained in $W_i(x)$. Since $W_1(z)$ contains at most $2k$ colored points, we are done if $\alpha_n k\ge (\alpha_{n-1}+2) k$.

The second case we consider is, if they are of Type 3 (Contain), such that a translate of $W_i$ is (not necessarily properly) contained in $-W_1$ (the wedge obtained by reflecting $W_1$ to the origin). 
Take a translate of $W_i$, $W_i(x)$ containing $\alpha_n k$ points of $P$.
If it does not intersect the level curve $\C^{W_1}(L)$, then we did not color any of its points, we are done.
If there is a $z\in \C^{W_1}(L)$ for which $W_1(z)\cap W_i(x)$ contains at least $(\alpha_{n-1}+2) k$ points, then we are done as only $2k$ of these can be colored.
Otherwise, for any $z\in \C^{W_1}(L)\cap W_i(x)$ denote by $a_z$ and $b_z$ the points where the boundary of $W_1(z)$ and $W_i(x)$ meet. So if $z$ is one of the two ends of the interval $\C^{W_1}(L)\cap W_i(x)$, we have $a_z=z$ or, respectively, $b_z=z$.
We also know that for any $z\in \C^{W_1}(L)$, the wedge $W_1(z)$ contains at least $L-((2n-2)\alpha_{n-1}) k$ points of $Q$.
A continuity argument shows that there is a $z'$ for which both $W_1(a_{z'})$ and $W_1(b_{z'})$ contain at least $\left(L-((2n-2)\alpha_{n-1}) k-(\alpha_{n-1}+2) k\right)/2$ points.
If this number is at least $2k$, then $W_i(x)\setminus W_1(z')$ cannot contain any colored points because of the moreover part of Lemma \ref{partialcolor}.
This implies that $W_i(x)$ can contain only the at most $2k$ colored points of $W_1(z')$.
So for this case we need the additional condition $L\ge ((2n-1)\alpha_{n-1}+6) k$.

Finally, notice that in the remaining cases, $W_i$ can be cut into three parts, $W_i^1$, $W_i^2$ and $W_i^3$, such that $W_i^1$ and $W_i^3$ have a side parallel to one of the sides of $W_1$ and for each there is a halfplane that contains it with $W_1$, while $W_i^2$ is contained in $-W_1$. If $W_i(x)$ contains at least $\alpha_n k$ points of $P$, then at least one of these three wedges must contain at least $\alpha_n k/3$ points. If it is $W_i^2$, then we are done as in the previous case if $\alpha_n k\ge 3(\alpha_{n-1}+2) k$. If it is one of the other two wedges, then we arrive to our last case.

Suppose $W_i(x)$ contains at least $\alpha_n k/3$ points of $P$, and there is a triangle such that $W_i$ and $W_1$ are among its three wedges.
Without loss of generality, suppose that their parallel side is the horizontal, they are contained in the ``upward'' halfplane and $W_1$ looks right, $W_i$ left. % and the other side of $W_i$ is vertical.
Again, if $W_i(x)$ does not intersect the level curve $\C^{W_1}(L)$, then we did not color any of its points, we are done. 
If it does, then consider the colored points in $W_i(x)$ in increasing order with respect to the order defined by the non-horizontal side of $W_1$.
If there are at most $2k$ colored points in $W_i(x)$, then we are done.
Otherwise, denote the $2k+1^{st}$ colored point according to this order by $y$.
Since $y$ is colored, there is a $z\in W_i(x)\cap \C^{W_1}(L)$ for which $y\in W_1(z)$. 
If $z\notin W_i(x)$, then $W_1(z)$ must contain all the points that are smaller than $y$ in the order, a contradiction, as it can contain at most $2k$ colored points.
If $z\in W_i(x)$, then we can use the property that $y\in Q$. A point was selected to $Q$ from $P$ only if there is a translate of $W_1$ containing exactly $L$ points from $P$ in which $p$ is not among the first $\alpha_{n-1} k$ points in any of the $2n-2$ orders defined by ${\cal W'}={\cal W}\setminus W_1$. The apex of this translate of $W_1$ must be on $W_i(x)\cap \C^{W_1}(L)$. But then $W_i(x)$ also contains the $\alpha_{n-1} k$ points smaller than $y$ in the order defined by the non-horizontal side of $W_i$, which are necessarily uncolored, thus we are done.
\end{proof}

Therefore we are also done with the proof of the theorem.
Note that the bound that we get for $\alpha_n$ grows superexponentially with $n$ because apart from $\alpha_n\ge 3\alpha_{n-1}+6$, we must also guarantee $\alpha_n\ge L=((2n-1)\alpha_{n-1}+6) k$ to make sure that also the translates of $W_1$ contain all $k$ colors.
We would like to remark that this bound can be made exponential by introducing a more sophisticated notation and demanding a different ``$\alpha$'' for each wedge in each step (so when there are $j$ wedges left, then the ``$\alpha$'' of $W_i$ should be approximately $2^i3^j$).

\subsection{Indecomposable Constructions}
In this section we survey results about coverings that cannot be decomposed into two coverings. The first such example was given in \cite{MP86}, where it was shown that the unit ball is not cover-decomposable. Thus for any $k$ there is a covering of $\mathbb R^3$ with unit balls such that every point is covered by at least $k$ balls, but the covering cannot be decomposed into two coverings. Later in \cite{PTT05} several other constructions were given, all based on the geometric realization of the same hypergraph not having Property B\footnote{We say that a hypergraph has Property B if the elements of the ground set can be colored with two colors such that any hyperedge contains both colors.}.
It was shown by Erd\H os \cite{E63} that the smallest number of sets of size $k$ that do not have Property B is at least $2^{k-1}$, so any indecomposable construction must be exponentially big. With a standard application of the Lov\'asz Local Lemma \cite{EL75} it can also be shown for ``nice'' geomteric sets that if every point is covered by less than exponentially many translates, then the covering is decomposable.

We start by presenting the construction of \cite{PTT05} using concave quadrilaterals proving Theorem B. Then we briefly preview the results of Section \ref{sec:concave}.

\subsubsection{Concave quadrilaterals}
We present the construction in the dual case.
We suppose that the vertices of the quadrilateral, $Q$, are $A, B, C$ and $D$ in this order, the obtuse angle being at $D$. This implies that $W_A$ and $W_C$ are of Type 5 (Special), moreover, they belong to an even more special subclass: When we translate the wedges such that their apices are in the origin, then they are disjoint and there is an open halfplane that contains both of their closures (see the two right examples in Figure \ref{fig:type5}). For simplicity, let us suppose that $W_A$ is a very thin wedge that contains a horizontal segment and $W_C$ is a very thin wedge that contains a vertical segment, the construction would work for any other two wedges that are derived from a concave quadrilateral.

First we give a finite set of points and a finite number of translates that show that $Q$ is not totally-cover-decomposable. Then we show how this construction is extendable to give a covering of the whole plane. The construction is based on a construction using translates of the wedges $W_A$ and $W_C$. We will use these wedges to realize the following $k$-uniform hypergraph, $\mathcal H$. The vertices of the hypergraph are sequences of length less than $k$ consisting of the numbers from $1$ through $k$: $V(\mathcal H)=[k]^{< k}$. There are two kinds of hyperedges. The first kind contains sequences of length $l$ whose restriction to their first $l-1$ members is the same. The second kind consists of a length $k-1$ sequence and all its possible restrictions. So $\mathcal H$ has roughly $k^k$ vertices and edges. 

\begin{figure}[htb]
\begin{center}
\scalebox{0.5}{\includegraphics{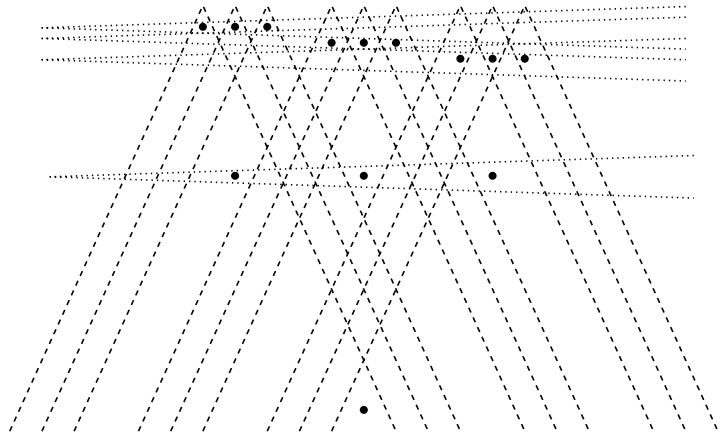}}
\caption{Indecomposable covering with two special wedges of a concave quadrilateral.}
\label{fig:quadrilateral}
\end{center}
\end{figure}

The hyperedges of the first kind are realized by translates of $W_A$, the second kind by translates of $W_C$. The vertices of the hypergraph are all very close to a vertical line. Also, vertices that belong to a hyperedge of the first kind are all on a horizontal line, for each edge on a different one (see Figure \ref{fig:quadrilateral}). It is easy to see that this is indeed a geometric realization of $\mathcal H$, so the points cannot be colored with two colors such that every translate of $W_A$ and $W_C$ of size $k$ contains both colors.

\begin{figure}[htb]
\begin{center}
\scalebox{0.5}{\includegraphics{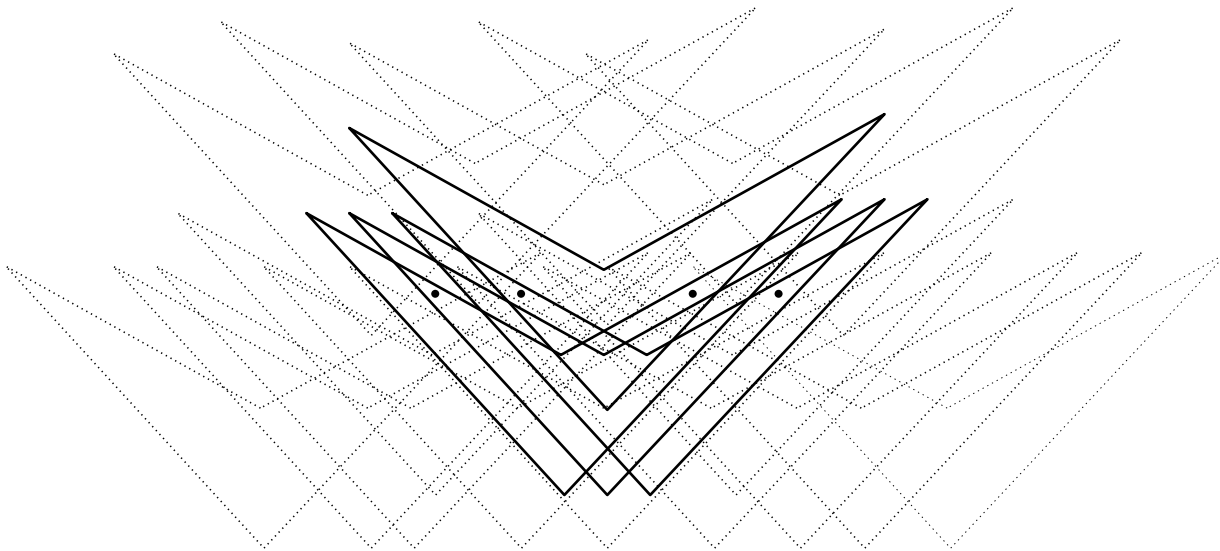}}
\caption{Extending the original $2$-fold covering of the four points by the solid quadrilaterals to a $2$-fold covering of the whole plane by adding the dotted quadrilaterals.}
\label{fig:quadriplane}
\end{center}
\end{figure}

Now we need to extend the corresponding covering to the whole plane. Before we do, notice that it can be achieved that the centers of all translates of $Q$ used in the construction lie on the same line. After going back from the dual to the primal, this means that we have a set of points, $S$, on a line, $\ell$, and an indecomposable, $k$-fold covering of them with translates of $Q$. Add all translates of $Q$ to our covering that are disjoint from $S$ (see Figure \ref{fig:quadriplane}). It is clear that the resulting covering remains indecomposable. Moreover, now every point not in $S$ will be covered infinitely many times\footnote{The construction can of course be modified to get a locally finite covering using a standard compactness argument.}, because $Q$ cannot have two sides that are parallel to $\ell$, so we can ``go in'' between any two points of $S$. Note that this statement is not necessarily true for arbitrary concave polygons, this is why the construction of \cite{P10} is not always extendable this way.

\subsubsection{General concave polygons and polyhedra}
The construction for concave polygons differs from the quadrangular case because it is no longer true that any pair of Type 5 (Special) wedges have the property that they can be translated such that their apices are in the origin and they are disjoint (see the two left examples in Figure \ref{fig:type5}). Because of this a different hypergraph (also not having Property B) is realized. This construction has less points (about $4^k$) and is more general as it can be realized by any pair of Type 5 (Special) wedges. The details can be found in Section \ref{sec:concave}.

However, this construction is not always extendable to give an indecomposable covering of the whole plane. Different notions of cover-decomposability and their connections are also studied in Section \ref{sec:concave}.
Finally, as a corollary of the construction, it is also shown at the end of Section \ref{sec:concave} that polyhedra (both convex and concave) are not cover-decomposable. This construction is extandable, thus we obtain that polyhedra are not space-cover-decomposable.

%% file: convex.tex
This section is based on our paper with G\'eza T\'oth, Convex polygons are cover-decomposable \cite{PT10}.

Our main result is the strongest statement, (iii), of Theorem A, which claims

\medskip

\noindent {\bf Theorem A.} {\em Every open convex polygon
is cover-decomposable.}

\medskip

We start by recalling some old definitions and making some new ones. Then we establish the earlier unproved Lemma \ref{2wedge} and \ref{sokwedge}, and finally, we summarize the proof of Theorem A.

\subsection{Preparation}

Now let $W$ be a wedge, and $X$ be a point in the plane. A translate of $W$
such that its apex is at $X$, is denoted by $W(X)$.
More generally, for points $X_1, X_2, \ldots X_k$,
$W(X_1, X_2, \ldots X_k)$ denotes the {\em minimal} translate of $W$
(for containment) whose
closure
contains $X_1, X_2, \ldots X_k$.
The set of all translates
of $W$ is denoted by $Tr^W$ and the set of those translates that contain exactly $k$ points from a point set $S$ is denoted by $Tr^W_k(S)$.
The reflection of $W$ about the origin is denoted by $-W$.

We can assume without loss of generality that the positive $x$-axis is in $W$, and that no two points from our point set, $S$, have the same $y$-coordinate.
Both of these can be achieved by an appropriate rotation.
We say that $X<_yY$ if the $y$-coordinate of $X$ is smaller than the
$y$-coordinate of $Y$. This ordering is called the $y$-ordering.
A subset $I$ of $S$ is an {\em interval} of $S$
if $\forall X<_y Y<_y Z\in S: X,Z\in I \rightarrow Y\in I.$

The {\em boundary} of $S$ with respect to $W$,
$Bd^W(S)=\{X\in P:W(X)\cap S=\emptyset\}.$
Note that a translate of $W$ always intersects the boundary in an interval.
For each $X\in Bd^W(S)$ the {\em shadow} of $X$ is
$Sh^W(X)=\{Y\in S: W(Y)\cap Bd^W(S)=X\}.$
Observe that $\forall X,Y\in Bd^W(S): Sh^W(X)\cap Sh^W(Y)=\emptyset.$

Now we give another proof using these notions for Lemma \ref{1wedge} that claims that any single wedge, $W$, is NC.
In fact, we show that if the angle of $W$ is less than $\pi$, then the points of $S$ can be colored with two colors such that any wedge that has at least two points contains both colors (so the NC property holds with $k=2$).

\medskip

\noindent {\bf Proof.} Color the points of the boundary alternating, according to the order $<_y$.
For every boundary point $X$,
color every point in the shadow of $X$ to the other color than $X$.
Color the rest of the points arbitrarily.
Any translate of $W$ that contains at least two points,
contains one or two boundary points.
If it contains one boundary point, then the other point is in its shadow,
so they have different colors.
If it contains two boundary points, then they are consecutive points according
to the $y$-order, so they have different colors again. \hfill $\Box$.

\subsection{NC wedges - Proof of Lemma \ref{2wedge} and \ref{sokwedge}}

Now we can turn to the case when we have
translates of two or more wedges at the same time.
Remember that any pair of wedges belong to a Type determined by their relative
position. The different types are classified in Section \ref{sec:class} and are depicted in Figures \ref{fig:type2}, \ref{fig:type3}, \ref{fig:type4}, \ref{fig:type5}.
It is shown in \cite{P10} that if ${\cal W}=\{ V, W\}$ is a set of two wedges
of Type 5 (Special), then $\cal W$ is {\em not} NC. In a series of lemmas we show that all
other pairs are NC, thus proving Lemma \ref{2wedge}.

\begin{lem}\label{convlem2} Let ${\cal W}=\{ V, W\}$ be
a set of two wedges, of Type 3 (Contain).
Then $\cal W$ is NC.
\end{lem}

\begin{figure}[htb]
\begin{center}
\scalebox{0.5}{\includegraphics{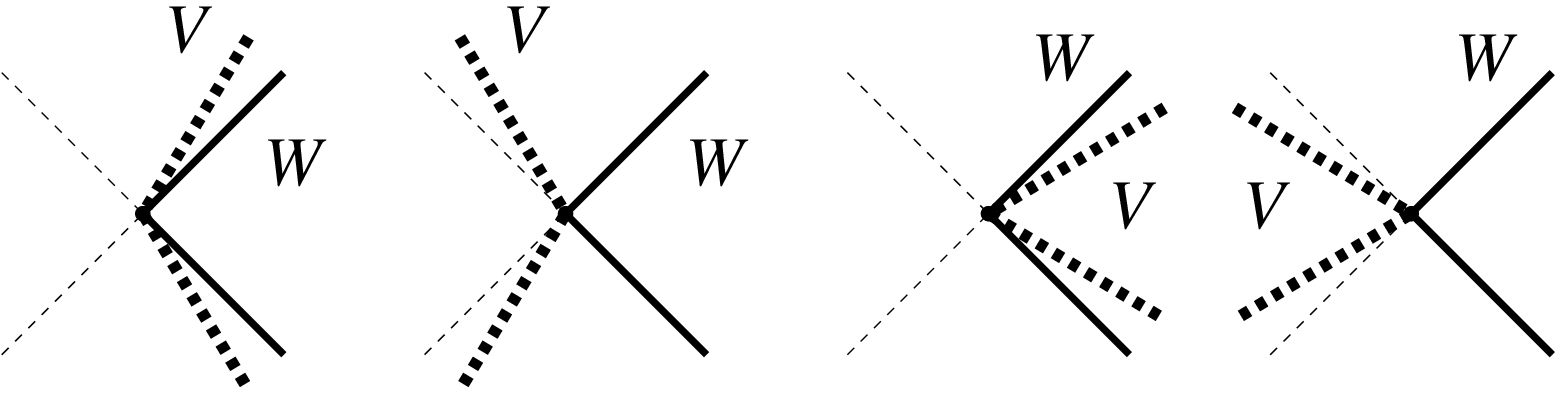}}
\caption{Type 3 (Contain)}
\label{localcontain}
\end{center}
\end{figure}

Note that this proof could be made slightly simpler with an argument similar to the one used in the proof of Lemma \ref{partialcolor}, here we reproduce the original proof.

\medskip

\noindent {\bf Proof.}
We can assume that $W\supset V$ or  $W\supset -V$
and $W$ contains the positive $x$-axis, just like on the two right diagrams of Figure \ref{localcontain}.
Let $(P^W_1, P^W_2, \ldots, P^W_k)$ be
the path decomposition of $S$ with respect to $W$, of order $k$.

Observe that any translate of $V$ intersects any $P^W_i$ in an interval of
it. Indeed, if
$X_1<_yX_2<_yX_3\in P^W_i$, then $X_2\in W(X_1, X_3)\cap -W(X_1, X_3)$, which
is a subset of $V(X_1, X_3)\cap -V(X_1, X_3)$. See Figure \ref{fig:vw}.

\begin{figure}[htb]
\begin{center}
\scalebox{0.5}{\includegraphics{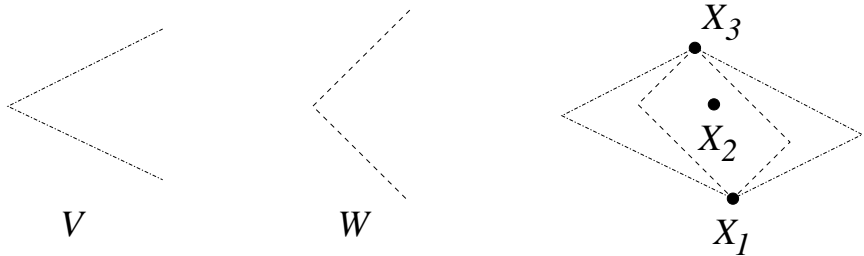}}
\caption{$W(X_1, X_3)\cap -W(X_1, X_3)\subset V(X_1, X_3)\cap -V(X_1, X_3)$.}
\label{fig:vw}
\end{center}
\end{figure}

We show that we can color the points of $S$ with red and blue such that
any translate of $W$ which contains at least 4 points, and any translate of $V$
which contains at least 14 points, contains points of both colors.
Consider $(P^W_1, P^W_2, P^W_3, P^W_4)$,
the path decomposition of $S$ with respect to $W$, of order $4$.
We color $P^W_1$ and $P^W_2$ such that every $W'\in Tr^W_4(S)$
contains a blue point of them, and every $V'\in Tr^V_7(P^W_1\cup P^W_2)$
contains points of both colors.
Similarly, we color $P^W_3$ and $P^W_4$ such that every $W'\in Tr^W_4(S)$
contains a red point of them, and every $V'\in Tr^V_7(P^W_3\cup P^W_4)$
contains points of both colors.
Finally, we color the rest of the points
$R=S\setminus(P^W_1\cup P^W_2\cup P^W_3\cup P^W_4)$
such that every $V'\in Tr^V_2(R)$ contains points of both colors.

Recall that for any $W'\in Tr^W_4(S)$,
$|W'\cap P^W_1|=|W'\cap P^W_2|=|W'\cap P^W_3|=|W'\cap P^W_4|=1$.
For any $X\in P^W_1$,  $Y\in P^W_2$, if there is a $W'\in Tr^W_4(S)$ with
$W'\cap P^W_1=\{X\}$ and
$W'\cap P^W_2=\{Y\}$, then
we say that $X$ and $Y$ are {\em friends}.
If $X$ (resp. $Y$) has only one friend $Y$ (resp. $X$), 
then we call it a {\em fan} (of $Y$, resp. of $X$).
If $X$ or $Y$ has at least one fan, then we say that it is a {\em star}.
Those points that are neither fans, nor stars are called {\em regular}.

For an example, see Figure \ref{fig:path2}. On the left figure, $Y_1$ is a star, its fans are $X_2$ and $X_3$, the other points are regular. On the right, $Y_2$ is a star, its fan is $X_2$, the other points are regular.

Suppose first that all points of $P^W_1$ and $P^W_2$ are regular.
Color every third point of $P^W_1$,
red and
the others blue. In $P^W_2$, color the friends of the red points blue,
and color the rest of the points of $P^W_2$ (every third) red.
For any $W'\in Tr^W_4$, $W'\cap P^W_1$ and
$W'\cap P^W_2$ are friends, therefore, at least one of them is blue.
On the other hand,
any $V'\in Tr_7^V(P^W_1\cup P^W_2)$ contains three consecutive
points of $P^W_1$ or $P^W_2$, and they have both colors.

Suppose now that not all the points of  $P^W_1$ and $P^W_2$ are regular.
Color all stars blue.
The first and last friend of a star, in the $y$-ordering,
is either a star or a regular vertex,
the others are fans.
Color the friends of each star alternatingly, according to the
$y$-ordering, starting with blue, except the last two friends; color the last
one
blue, the previous one red.
The so far uncolored regular points of
$P^W_1$ and $P^W_2$ form pairs of intervals.
We color each such pair of interval the same way as we did in the all-regular
case,
coloring the first point of each pair of intervals red. See Figure \ref{fig:friends}.

Clearly,
if $W'\in Tr^W_4$ then it contains at least one blue point of $P^W_1\cup P^W_2$.
If $V'\in Tr_7^V(P^W_1\cup P^W_2)$, then it contains four consecutive
points of $P^W_1$ or $P^W_2$, say, $X_1, X_2, X_3, X_4$, in $P^W_1$.
If $X<_y Y<_y Z\in P^W_1\cap V'$ and $Y$ is a
star,
then $V'$ must contain all fans of $Y$ as well.
Indeed, the fans of $Y$ are in $W(X,Z)\setminus(W(X)\cup W(Z))$,
and by our earlier observations, this is in
$V(X,Z)\subset V'$.
So, if either $X_2$ or $X_3$ is a star, then
$V'$ contains
a red point, since every star has a red fan.
Since the star itself is blue, we are done in this case.
If $X_1, X_2, X_3, X_4$ contains three consecutive regular
vertices then we are done again, by the coloring rule for the regular
intervals.
So we are left with the case when $X_1$ and $X_4$ are stars, $X_2$ and $X_3$
are regular.
But in this case $V'$ also contains the common friend $Y$ of $X_2$ and $X_3$ in
$P^W_2$, which is also a regular vertex.
By the coloring rule for the regular
intervals, one of $Y$, $X_2$ and $X_3$ is red, the other two are blue, so we
are done.

\begin{figure}[htb]
\begin{center}
\scalebox{0.5}{\includegraphics{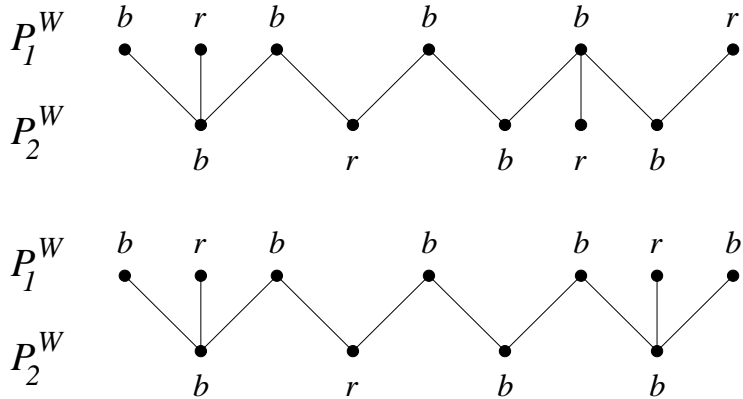}}
\caption{Two examples of coloring of $P^W_1\cup P^W_2$. Friends are connected
 by edges.}
 \label{fig:friends}
\end{center}
\end{figure}

\vskip 0.2cm

For $P^W_3\cup P^W_4$ we use the same coloring rule as for
$P^W_1\cup P^W_2$ but we switch the roles of the colors.
So any $W'\in Tr_4^W$ contains at least one red point of
$P^W_3\cup P^W_4$ and any $V'\in Tr_7^V(P^W_3\cup P^W_4)$ contains both
colors.

Finally, we have to color the rest of the points
$R=S\setminus(P^W_1\cup P^W_2\cup P^W_3\cup P^W_4)$
such that every $V'\in Tr^V_2(R)$ contains points of both colors.
This can be achieved by the first proof of Lemma \ref{1wedge}.

Now
any $W'\in Tr_4^W$ contains at least one blue and at least one red point.
If $V'\in Tr_{14}^V$, then either it contains at least two points of $R=P\setminus
(P^W_1\cup P^W_2\cup P^W_3\cup P^W_4)$,
or at least seven points of $P^W_1\cup P^W_2$, or at least seven points of
$P^W_3\cup P^W_4$, and in all cases it contains points of both colors.
This completes the proof of Lemma \ref{convlem2}. \hfill $\Box$

\begin{defi}\label{convdefi2} Suppose that ${\cal W}=\{ V, W \}$ is a pair of wedges.
$\cal W$ is said to be {\em asymmetric non-conflicting} 
or simply ANC, if there is a
constant $k$ with the
following property.
Any finite set of points $S$ can be colored with red and blue such that
any translate of $V$
that contains at least $k$ points of
$S$, contains a red point, and
any translate of $W$
that contains at least $k$ points of
$S$, contains a blue point.
\end{defi}

The next technical result allows us to simplify all following proofs.

\begin{lem}\label{convlem3} If a pair of wedges is not of Type 5 (Special),
and ANC, then it is also NC.
\end{lem}

\begin{figure}[htb]
\begin{center}
\scalebox{0.5}{\includegraphics{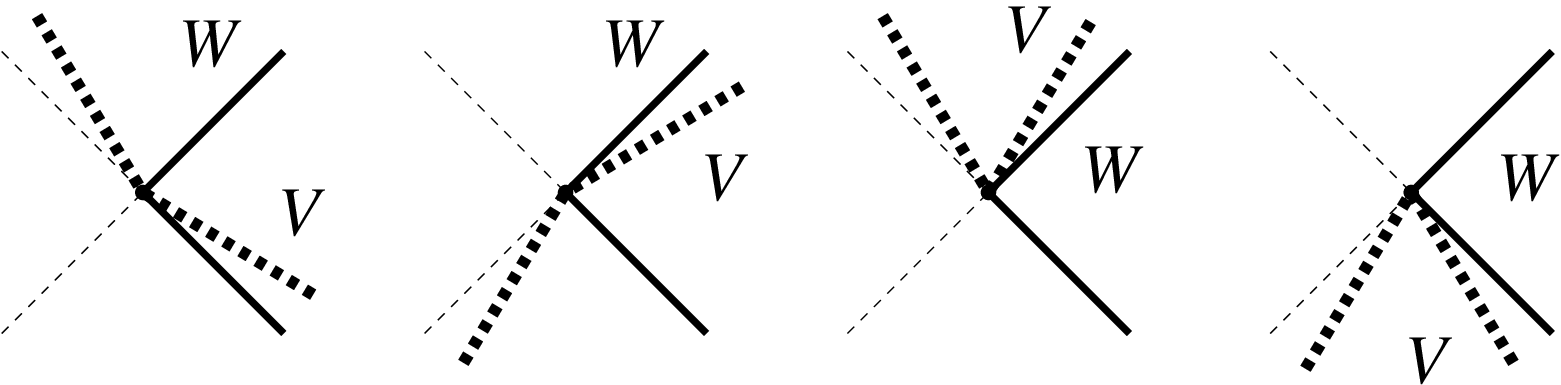}}
\caption{Type 5 (Special)}
\end{center}
\end{figure}

\noindent {\bf Proof.}
We can assume without loss of generality that
$V$ contains the positive $x$-axis, and $W$ contains either the positive or
the negative $x$-axis.
Suppose that $\{V, W\}$ is ANC, let $k>0$ arbitrary, and let $S$ be a set
of points.
First we color $Bd^V(S)$.
Let $U$ be a wedge that also contains the
positive $x$-axis, but has a very small angle.
Then translates of $V$ and translates of
$U$ both intersect  $Bd^V(S)$ in its intervals.
Clearly, the pair $\{U, W\}$ is of Type 3 (Contain), therefore, by Lemma \ref{convlem2},
we can color  $Bd^V(S)$ such that any translate of $W$,
$W'\in Tr_4^W(Bd^V(S))$ and any translate of $U$,
$U'\in Tr_{14}^U(Bd^V(S))$ contains both colors.
But then any translate of $V$,
$V'\in Tr_{14}^V(Bd^V(S))$ contains both colors as well.

Now we have to color $S\setminus Bd^V(S)$.
We divide it into three parts as follows.
%$A_B=\cup_{q\in Bd^V(P) blue} Sh^V(q)$
%$A_R=\{p\in P\setminus Bd^V(P): \exists q\in Bd^V(P) \textit{ red}: p\in Sh^V(q)\}$,\\
$$S_{\mbox{b}}=\{ X\in S\setminus Bd^V(S)\ |\  \forall Y\in V(X)\cap
Bd^V(S), Y \ {\mbox {is blue}} \},$$
$$S_{\mbox{r}}=\{ X\in S\setminus Bd^V(S)\ |\  \forall Y\in V(X)\cap
Bd^V(S), Y \ {\mbox {is red}} \},$$
$$S_0=S\setminus (Bd^V(S)\cup S_{\mbox{b}}\cup S_{\mbox{r}}).$$

Any translate $V'\in Tr^V$
that intersects $S_{\mbox{b}}$ in at least one point,
must contain at least one blue point, from $Bd^V(S)$,
so we only have to make sure that it contains a red point too.
Similarly, any
$V'\in Tr^V$ that intersects $S_{\mbox{r}}$ in at least one point,
must contain a red point,
and any $V'\in Tr^V$ that intersects $S_0$ must contain points of both colors.

Thus, we can simply color $S_0$ such that any $W'\in Tr_2^W(S_0)$ contains
both colors, which can be done by Lemma \ref{1wedge}.

With  $S_{\mbox{b}}$, and with $S_{\mbox{r}}$, respectively, we proceed
exactly the same way as we did with $S$ itself, but now we change the roles of
$V$ and $W$.
We get the (still uncolored)
subsets $S_{\mbox{b,b}}$, $S_{\mbox{b,r}}$, $S_{{\mbox{b}},0}$,
$S_{\mbox{r,b}}$, $S_{\mbox{r,r}}$, $S_{{\mbox{r}},0}$
with the following properties.

\begin{itemize}
\item Any translate $V'\in Tr^V$ or $W'\in Tr^W$,
that intersects  $S_{\mbox{b,b}}$
(resp. $S_{\mbox{r,r}}$)
in at least one point,
must contain at least one blue (resp. red) point.

\item Any translate $V'\in Tr^V$
that intersects  $S_{\mbox{b,r}}$
(resp. $S_{\mbox{r,b}}$)
contains a blue (resp. red) point,
and any translate $W'\in Tr^W$
that intersects  $S_{\mbox{b,r}}$
(resp. $S_{\mbox{r,b}}$)
contains a red (resp. blue) point.

\item Any translate $V'\in Tr^V$
that intersects  $S_{{\mbox{b}},0}$
(resp. $S_{{\mbox{r}},0}$)
contains a blue (resp. red) point,
and any translate $W'\in Tr^W$
that intersects  $S_{{\mbox{b}},0}$
(resp. $S_{{\mbox{r}},0}$)
contains points of both colors.

\end{itemize}

Color all points of
$S_{\mbox{b,b}}$ and $S_{{\mbox{b}},0}$ red, color
all points of $S_{\mbox{r,r}}$ and $S_{{\mbox{r}},0}$ blue.
Finally, color $S_{\mbox{b,r}}$ using the ANC property of the pair
$(V, W)$, and similarly, color $S_{\mbox{r,b}}$ also using the ANC
property, but the roles of red and blue switched.
Now it is easy to check that in this coloring any translate of $V$ or $W$ that
contains sufficiently many points of $S$, contains a point of both colors. \hfill $\Box$

\begin{remark}  In \cite{P10} it has been proved that if $\{V, W\}$
is a Special pair,
then $\{V, W\}$ is {\em not} ANC. So, the following statement
holds as well.
\end{remark}

\noindent {\bf Lemma \ref{convlem3}'.} {\em If a pair of wedges is
ANC, then it is also NC.}

\bigskip

%%%%%%%%%%%%%%%%%%%%%%%%%%%%%%%%%%%%%%%%%%%%%%%%%%%%%%%%%%%%%%%%%%%%%%%%%%%%%%%%%%%%%%%%%%%%

\begin{lem}\label{convlem4} Let ${\cal W}=\{ V, W\}$ be
a set of two wedges, of Type 1 (Big).
Then $\cal W$ is NC.
\end{lem}

\noindent {\bf Proof.} By Lemma \ref{convlem3}, it is enough to show that $\{ V, W\}$ is ANC. Let $W$ be the
wedge whose angle is at least $\pi$. Then $W$ is the union of two halfplanes,
say, $H_1$ and $H_2$. Translate both halfplanes such that they contain exactly
one point of $S$, denote them by
$X_1$ and $X_2$, respectively. Note that $X_1$ may coincide with $X_2$.
Color $X_1$ and $X_2$ red, and all the other points blue.
Then any translate of $W$ that contains at least one point, contains a red
point, and any translate of $V$ that contains at least three points, contains
a blue point.
\hfill $\Box$

\begin{lem}\label{convlem5} Let ${\cal W}=\{ V, W\}$ be
a set of two wedges, of Type 2 (Halfplane).
Then $\cal W$ is NC.
\end{lem}

\begin{figure}[htb]
\begin{center}
\scalebox{0.5}{\includegraphics{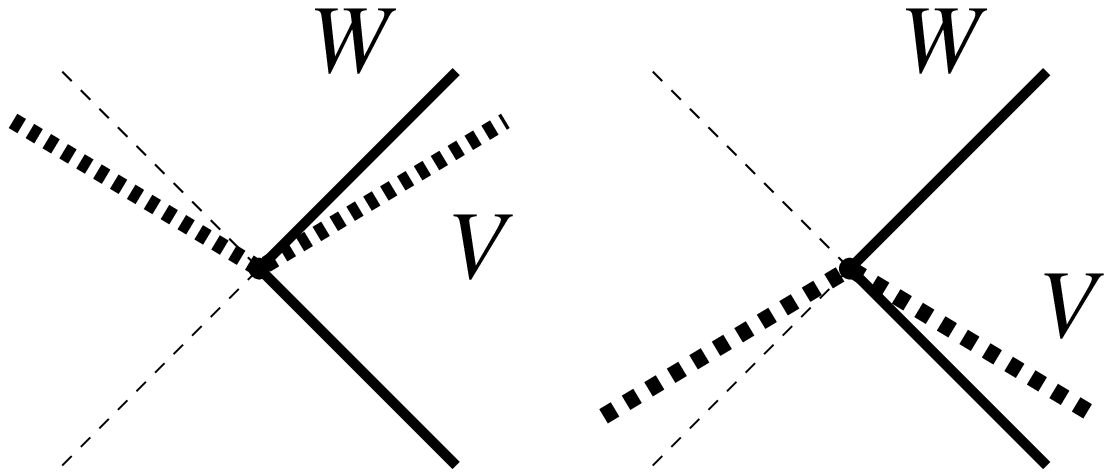}}
\caption{Type 2 (Halfplane)}
\end{center}
\end{figure}

\noindent {\bf Proof.}
Again, it is enough to show that they are ANC. Since $\{ V, W\}$ is
of Type 2 (Halfplane),
$Bd^V(S)$ and $Bd^W(S)$ have at most one point in
common.
If $Bd^V(S)$ and $Bd^W(S)$ are disjoint, then color  $Bd^V(S)$ blue,
$Bd^W(S)$ red, and the other points arbitrarily.
Then any nonempty translate of $V$ (resp. $W$) contains a blue (resp. red) point.

Otherwise, let $X$ be their common point.
Let $P=Bd^V(S)\cup Bd^W(S)\setminus X$, and consider its $V$-boundary,
$Bd^V(P)$, and $W$-boundary,
$Bd^W(P)$.
Clearly, each point in $P=Bd^V(S)\setminus X$ belongs to $Bd^V(P)$,
and each point in $P=Bd^W(S)\setminus X$ belongs to $Bd^W(P)$.

If $Bd^V(P)$ and
$Bd^W(P)$ are disjoint, then color  $Bd^V(S)$ blue,
$Bd^W(P)$ and the other points red.
Then any nonempty translate of $V$ contains a blue point.
Suppose that we have a translate of $W$ with two points,
both blue. Then it should contain $X$, and a point
of  $Bd^V(P)$. But this contradicts our assumption that
$Bd^V(P)$ and
$Bd^W(P)$ are disjoint.
So, any
translate of $W$ which contains at least two points of
$S$,
contains a red point.

If $Bd^V(P)$ and
$Bd^W(P)$ are not disjoint,
then they have one point in common, let
$Y$ be their common point.
If $Y$ belongs to $Bd^W(P)$, then color
$Bd^V(S)$ blue,
$Bd^W(P)$ and the other points red.
Then, by the same argument as before,
any nonempty translate of $V$ contains a blue point, and
any
translate of $W$ which contains at least two points of
$S$,
contains a red point.
Finally, if $Y$ belongs to $Bd^V(P)$, then we proceed
analogously, but the roles of $V$ and $W$, and
the colors, are switched. \hfill $\Box$

\begin{lem}\label{convlem6} Let ${\cal W}=\{ V, W\}$ be
a set of two wedges, of Type 4 (Hard).
Then $\cal W$ is NC.
\end{lem}

\begin{figure}[htb]
\begin{center}
\scalebox{0.5}{\includegraphics{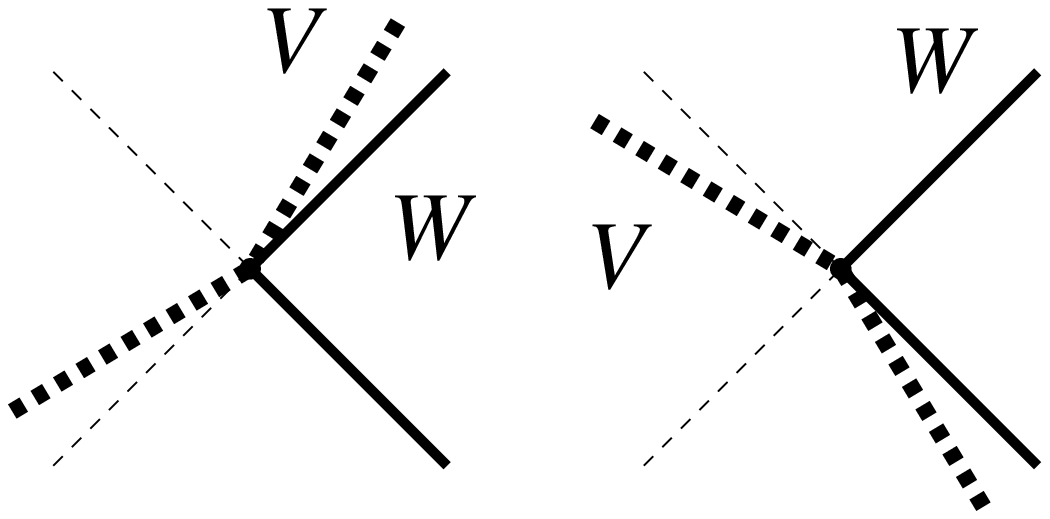}}
\caption{Type 4 (Hard)}
\label{figlochard}
\end{center}
\end{figure}

\noindent {\bf Proof.}
As usual, we only prove that $\{ V, W\}$ is ANC.
Assume that
$W$ contains the positive $x$-axis.
Just like in the definition of the different types,
extend the boundary halflines of $W$ to lines,
they divide the plane
into four parts, Upper, Lower, Left, and Right, latter of which is $W$ itself.
We can assume without loss of generality that
$V$ contains the negative $x$-axis, one side of $V$ is in Upper, and one side
is in Left, just like on the left of Figure \ref{figlochard}.

Observe that if a translate of $V$ and a translate of $W$ intersect each
other, then one of them contains the apex of the other one.

\begin{claim}\label{convclaim1} For any point set $P$ and $X\in P$,
either $Bd^V(P\setminus X)\setminus Bd^V(P)=\emptyset$ or $Bd^W(P\setminus X)\setminus Bd^W(P)=\emptyset$.
\end{claim}

{\bf Proof.} Suppose on the contrary that
$Y\in Bd^V(P\setminus X)\setminus Bd^V(P)$ and $Z\in Bd^W(P\setminus X)\setminus Bd^W(P)$.
Then $X\in V(Y)$ and $X\in W(Z)$, so $V(Y)$ and $W(Z)$ intersect each other,
therefore, one of them contains the other one's apex, say,
$Z\in V(Y)$. But this is a contradiction, since $Y$ is a boundary point of
$P\setminus X$. \hfill $\Box$

\smallskip

Return to the proof of Lemma \ref{convlem6}.
Color
$Bd^V(S)\setminus Bd^W(S)$ red, and
$Bd^W(S)\setminus Bd^V(S)$ blue, the interior points arbitrarily.
Now consider the points of $Bd^V(S)\cap Bd^W(S)$.
For any $X\in Bd^V(S)\cap Bd^W(S)$,
if
$Bd^V(S\setminus X)\setminus Bd^V(S)\ne\emptyset$, then color it red,
if $Bd^W(S\setminus X)\setminus Bd^W(S)\ne\emptyset$, then color it blue.
For each of the remaining points $Y$ we have
$Bd^V(S\setminus Y)\setminus Bd^V(S)=
 Bd^W(S\setminus Y)\setminus Bd^W(S)=\emptyset$.
Color each of these points such that they have
the {\em opposite} color than the the previous point
of $Bd^V(S)\cap Bd^W(S)$,
in the $y$-ordering.

To prove that this coloring is
good, let
$V'\in Tr_2^V$, $V'\cap S = \{ X, Y\}$.
If it intersects $Bd^V(S)\setminus Bd^W(S)$, we are done.
So assume that $V'\cap Bd^V(S)\subset Bd^V(S)\cap Bd^W(S)$.
Let $X\in V'\cap Bd^V(S)$. If $X$ is red, then by the coloring rule,
$Bd^V(S\setminus X)\setminus Bd^V(S)=\emptyset$. But then $Y$ is also a
$V$-boundary point, so
we have $Y\in  Bd^V(S)\cap Bd^W(S)$. Again we can assume that $Y$ is
red, so
$Bd^V(S\setminus Y)\setminus Bd^V(S)=\emptyset$.
Suppose that $X<_yY$. Since $V'\cap S = \{ X, Y\}$, $X$ and $Y$ are
consecutive points of $Bd^V(S)\cap Bd^W(S)$.
Now it is not hard to see that
$Bd^W(S\setminus Y)\setminus Bd^W(S)=\emptyset$. Therefore, by the coloring
rule, $X$ and $Y$ have different colors.
For the translates of $W$ the argument is analogous, with the colors switched.\hfill $\Box$

\medskip

Now we turn to the case when we have more than two wedges.

\begin{lem}\label{convlem7} For any $s, t>0$ integers, there is a number $f(s, t)$ with the
following property.

Let ${\cal W}=\{ W_1, W_2, \ldots , W_t \}$ be a set of $t$ wedges, such that
any pair $\{W_i, W_j \}$ is NC, and let $S$ be a set of points.
Then $S$ can be decomposed into $t$ parts, $S_1, S_2, \ldots , S_t$, such that
for $i=1, 2, \ldots , t$, for any translate $W'_i$ of $W_i$,
if $|W'_i\cap S|\ge f(s, t)$ then
$|W'_i\cap S_i|\ge s$.
\end{lem}

\noindent {\bf Proof.}
The existence of $f(1,2)$ is equivalent to the property that
the corresponding two wedges are ANC.
Now we show that $f(s, 2)$ exists for every $s$.
Let $V$ and $W$ be two wedges that form a NC pair.
Let $P^V_1, P^V_2, \ldots , P^V_{s^2f(1,2)}$ be the path decomposition
of $S$ of order $s^2f(1,2)$, with respect to $V$.
For $i=1, 2, \ldots , s$, let
$$H_i=\cup_{j=(i-1)sf(1,2)+1}^{isf(1,2)} P_j^V.$$
For each $H_i$, take the $W$-path decomposition, $P_1^W(H_i),\ldots,
P_{sf(1,2)}^W(H_i)$, and for each $j=1, 2, \ldots , s$, let
$$H_i^j=\cup_{k=(j-1)f(1,2)+1}^{jf(1,2)} P_k^W(H_i).$$
For every
$i, j=1, 2, \ldots , s$,
color $H_i^j$, such that any translate of $V$ (resp. $W$) that intersects it
in at least $f(1,2)$ points,
contains at least one red (resp. blue) point of it. This is possible, since
the
pair $\{ V, W\}$ is ANC.

Consider a translate $V'$ of $V$ that
contains at least $s^2f(1,2)$ points of $S$.
For every $i$, $V'$ intersects $H_i$ in
$sf(1,2)$ points, so there is a $j$ such that
it intersects $H_i^j$ in at least $f(1,2)$
points. Therefore, $V'$ contains at least one red point
of $H_i^j$, so at least $s$ red points of $S$.

Consider now a translate $W'$ of $W$ that
contains at least $s^2f(1,2)$ points of $S$.
There is an $i$ such that
$W'$ intersects $H_i$ in at least $sf(1,2)$
points. Therefore, it intersects
each of $P_1^W(H_i),\ldots,
P_{sf(1,2)}^W(H_i)$, in at least one point, so
for $j=1, 2, \ldots, s$, $W'$ intersects
$H_i^j$ in at least $f(1,2)$ points. Consequently,
it contains at least one blue point of each
$H_i^j$, so at least $s$ blue points of $S$.

\smallskip

Now let $s, t>2$ fixed and suppose that $f(s',t-1)$ exists for every $s'$.
Let $\{W_1, \ldots , W_t\}$ be our set of wedges, such that any pair
of them is NC. Let $s'=f(s, 2)$.
Partition our point set $S$ into
$S'_1, S'_2, \ldots, S'_{t-1}$ such that
for $i=1, 2, \ldots , t-1$, for any translate $W'_i$ of $W_i$,
if $|W'_i\cap S|\ge f(s', t-1)$ then
$|W'_i\cap S'_i|\ge s'=f(s, 2)$.
For each $i=1, 2, \ldots , t-1$, partition
$S'_i$ into two parts, $S''_i$ and $S^t_i$,
such that for any translate $W'_i$ of $W_i$,
if $|W'_i\cap S'_i|\ge f(s, 2)$ then
$|W'_i\cap S''_i|\ge s$, and
for any translate $W'_t$ of $W_t$,
if $|W'_t\cap S'_i|\ge f(s, 2)$ then
$|W'_t\cap S^t_i|\ge s$.
Finally,
for $i=1, 2, \ldots , t-1$, let $S_i=S''_i$ and
let $S_t=\cup_{j=1}^{t-1}S^t_j$.
For $i=1, 2, \ldots , t-1$,
any translate $W'_i$ of $W_i$, if
$|W'_i\cap S|\ge f(s', t-1)$ then
$|W'_i\cap S'_i|\ge s'=f(s, 2)$, so
$|W'_i\cap S_i|\ge s$,

And for any translate $W'_t$ of $W_t$, if
$|W'_i\cap S|\ge f(s', t-1)$,
then for some $i=1, 2, \ldots t-1$,
$|W'_t\cap S'_i|\ge \frac{f(s', t-1)}{t-1}\ge f(s,2)$, therefore,
$|W'_t\cap S^t_i|\ge s$, so
$|W'_t\cap S_t|\ge s$.
This concludes the proof of Lemma \ref{convlem7}. \hfill $\Box$

\begin{remark} The proofs of Lemmas \ref{convlem2}, \ref{convlem4}, \ref{convlem5}, and \ref{convlem6} imply that
$f(1, 2)\le 8$. Combining it with the proof of Lemma \ref{convlem7} we get the bound
$f(s,t)\le (8s)^{2^{t-1}}$.
\end{remark}

As a corollary, we have can now prove Lemma \ref{sokwedge}.
\medskip

\noindent {\bf Lemma \ref{sokwedge}.} {\em A set of wedges ${\cal W}=\{ W_1, W_2, \ldots ,
W_t \}$ is NC if and only if any pair $\{W_i, W_j \}$ is NC.}

\smallskip

\noindent {\bf Proof.}
Clearly, if some pair $\{W_i, W_j \}$ is {\em not} NC, then the whole set
$\cal W$ is not NC either.
Suppose that every pair $\{W_i, W_j \}$ is NC.
Decompose $S$ into $t$ parts $S_1, S_2, \ldots , S_t$
with the property that
for $i=1, 2, \ldots , t$, for any translate $W'_i$ of $W_i$,
if $|W'_i\cap S|\ge f(3, t)$ then
$|W'_i\cap S_i|\ge 3$.
Then, by Lemma \ref{1wedge}, each $S_i$ can be colored with red and blue such
that if
$|W'_i\cap S_i|\ge 3$
then $W'_i$ contains points of both colors.
So this coloring of $S$ has the property that
for $i=1, 2, \ldots , t$, for any translate $W'_i$ of $W_i$,
if $|W'_i\cap S|\ge f(3, t)$ then it contains points of both colors. \hfill $\Box$

\subsection{Summary of Proof of Theorem A}
Although we have already established Theorem A, we find it useful to give another summary of the complete proof.

Suppose that $P$ is an open convex polygon of $n$ vertices and
${\cal P}=\{\ P_i\ |\ i\in I\ \}$ is a collection of translates of
$P$ which forms an
$M$-fold covering of the plane.
We will set the value of $M$ later.
Let $d$ be the minimum distance
between any vertex and non-adjacent side of $P$.
Take a square grid $\cal G$ of basic distance $d/2$.
Obviously, any translate of $P$ intersects at most
$K=4\pi(diam(P)+d)^2/d^2$ basic squares.
For each (closed) basic square $B$, using its compactness,
we can find a {\em finite} subcollection of the translates such that
they still form an $M$-fold covering of $B$.
Take the union of all these subcollections.
We have a {\em locally finite} $M$-fold covering of the plane.
That is, every compact set is intersected by finitely many of the translates.
It is sufficient to decompose this covering.
For simplicity, use the same notation
${\cal P}=\{\ P_i\ |\ i\in I\ \}$ for this subcollection.

We formulate
and solve the problem in its dual
form.
Let $O_i$ be the center of gravity of $P_i$.
Since ${\cal P}$ is an $M$-fold covering of the plane,
every translate of $\bar{P}$, the reflection of $P$ through the
origin, contains at least $M$ points of the locally finite set
${\cal O}=\{\ O_i\ |\ i\in I\ \}$.

The collection ${\cal P}=\{\ P_i\ |\ i\in I\ \}$ can be decomposed into two
coverings
if and only if the set ${\cal O}=\{\ O_i\ |\ i\in I\ \}$ can be colored with
two colors, such that every translate of $\bar{P}$ contains a point of both
colors.

Let ${\cal W}=\{ W_1, W_2, \ldots , W_n \}$ be the set of wedges
that correspond to the vertices of $\bar{P}$.
By the convexity of $\bar{P}$,
no pair $\{W_i, W_j \}$ is of Type 5 (Special), therefore,
by the previous Lemmas,
each pair is NC. Consequently, by Lemma \ref{sokwedge}, ${\cal W}$ is
NC as well. So there is an $m$ with the following property.

\smallskip

\noindent {\bf *}  Any set of points $S$ can be colored with two colors such that
any translate of $W_1, \ldots , W_n$
that contains at least $m$ points of $S$, contains points of both colors.

\smallskip

Choose $M$ such that $M\ge mK$, and color the points of
$\cal O$ in each basic square separately, with property {\bf *}.

Since any translate $P'$ of $\bar{P}$ intersects at most
%$40A(P)/d^2$
$K$ basic squares of the grid $\cal G$, $P'$ contains at least $M/K\ge m$ points of
${\cal O}$ in the same basic square $B'$.
By the choice of the grid $\cal G$,
$B'$ contains at most one vertex of $P'$, hence
$B'\cap P'=B'\cap W$, where $W$ is a translate
of some $W_i\in {\cal W}$. So, by property {\bf *},
$P'$ contains points of ${\cal O}\cap B'$ of both colors.
This concludes the proof of Theorem A. \hfill $\Box$

\subsection{Concluding Remarks}
Throughout this section we made no attempt to optimize the constants.
However, it may be an interesting problem to determine (asymptotically) 
the smallest $m$ in the proof of Theorem A. 

Another interesting question is to decide whether
this constant depends only on the number of vertices of the polygon, or on the
shape as well. In particular, we cannot verify the following.

\begin{conj}\label{conj:quad} There is a constant $m$ such that any
$m$-fold covering of the plane with translates of a convex quadrilateral
can be decomposed into two coverings.
\end{conj}

With a slight modification of our proof of Theorem A, we get the following more general result about decomposition to $k$ coverings.

\medskip

\noindent {\bf Theorem A'.} {\em For any open convex polygon (or concave polygon without Type 5 (special) wedges), $P$, and any 
$k$, there exists a (smallest) number $m_k(P)$, such that 
any $m_k(P)$-fold covering of the plane with translates of
$P$ can be decomposed into $k$ coverings. }

\medskip

Our proof gives $m_k(P)<K_P(8k)^{2^{n-1}}$, where $K_P$ is the constant $K$ from the
proof of Theorem A and $n$ is the number of vertices of $P$. 
The best known lower bound on $m_k(P)$ is $\lfloor 4k/3\rfloor-1$
\cite{PT07}.
Recently Gibson and Varadarajan \cite{GV10} proved Theorem C (iii), which is a linear upper bound for all convex polygons.
However, their proof does not work for any cover-decomposable polygon, because they handle only one case of Type 3 (Contain).
For a summary of their result see the end of Section \ref{sec:introdec}.
We conjecture that a linear upper bound also holds for cover-decomposable concave polygons.

\begin{conj}\label{conj:multipoly} For any cover-decomposable polygon $P$, $m_k(P)=O(k)$.
\end{conj}

For more related conjectures see Section \ref{sec:conj}.
 
\medskip

Our proofs use the assumption that the covering is locally finite, and 
for open polygons we 
could find a locally finite subcollection which is still a $m$-fold covering.
Still, we strongly believe that Theorem A holds for closed convex polygons
as well.

%% file: concave.tex
This section is based on my paper, Indecomposable coverings with concave polygons \cite{P10}.

The main goal of this section is to prove results about {\it non}-cover-decomposable polygons. 
To understand some of the results, we need to recall the notions introduced in Section \ref{sec:total}.

\medskip

\noindent {\bf Definition \ref{deftotal}.} A planar set $P$ is said to be {\em totally-cover-decomposable}
if there exists a (minimal) constant $m^T=m^T(P)$ such that
every $m^T$-fold covering of ANY planar point set with translates of $P$
can be decomposed into two coverings. Similarly, let $m^T_k(P)$ denote the smallest number $m^T$ with the property
that every $m^T$-fold covering of ANY planar point set with translates of $P$
can be decomposed into $k$ coverings.

\medskip

When we want to emphasize the difference from the original definition, we will call the cover-decomposable sets {\em plane-cover-decomposable}. By definition, if a set is totally-cover-decomposable, then it is also plane-cover-decomposable. On the other hand, we cannot rule out the possibility that
there are sets, or even polygons, which are plane-cover-decomposable,
but not totally-cover-decomposable.

The results of Theorem A all remain true if we write totally-cover-decomposable instead of cover-decomposable and similarly, in Theorem C we can replace $m_k(P)$ with $m^T_k(P)$. In fact in the proofs of these theorems, these more general claims are proved. The proof of Theorem B establishes first that concave quadrilaterals are not totally-cover-decomposable and then extends the covering, proving that they are also not plane-cover-decomposable.
The main result of this section is a generalization of Theorem B. We show that {\em almost} all (open or closed) concave polygons are {\it not} totally-cover-decomposable and prove that most of them are also not plane-cover-decomposable. We need the ``almost'' because Theorem A' implies that any concave polygon without Type 5 (Special) wedges is totally-cover-decomposable.

Our main result is the following

\medskip

%previously 1
\noindent {\bf Theorem E.} {\em If a polygon has a pair of Type 5 (Special) wedges, then it is not totally-cover-decomposable.}

\medskip

\begin{figure}[htb]
\begin{center}
\scalebox{0.5}{\includegraphics{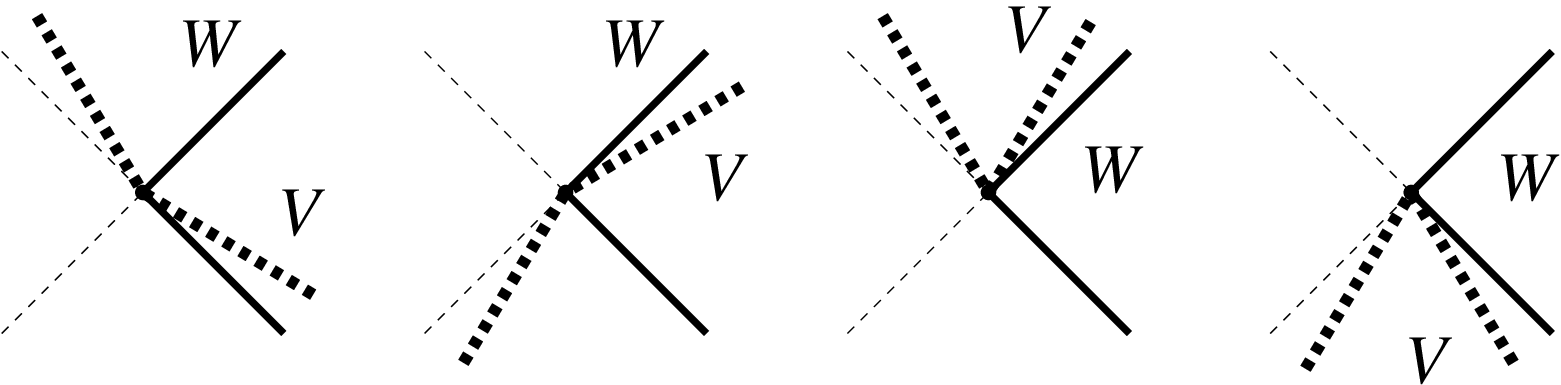}}
\caption{Type 5: Special pair of wedges}
\label{onlyspecial}
\end{center}
\end{figure}

Together with the previous theorem, this gives a complete characterization
of totally-cover-de\-com\-posable open polygons; an open polygon is
totally-cover-decomposable if and only if it does not have a Special pair of
wedges.

We show that every concave polygon with no parallel sides
has a pair of Special wedges, therefore we have

\medskip

%previously 2
\noindent {\bf Theorem E'.} {\em If a concave polygon has no parallel sides, then it is not totally-cover-decomposable.}

\medskip

The proof of these theorems can be found in Section \ref{concsec2}.
The problem of deciding plane-cover-decomposability for concave polygons is still open. However, in Section \ref{concsec3}, we prove that a large class of concave polygons are not plane-cover-decomposable. We also show that any ``interesting'' covering of the plane uses only countably many translates. (However, we do not consider here the problem when we want to decompose into infinitely many coverings; the interested reader is referred to the paper of Elekes, M\'atrai and Soukup \cite{EMS10}.)\\

Finally, in Section \ref{concsec4}, we investigate the problem in three or more dimensions. The notion of totally-cover-decomposability extends naturally and we can also introduce space-cover-decomposability. Previously, the following result was known.

\medskip

\noindent {\bf Theorem F.} {Mani-Levitska, Pach }\cite{MP86} {\em  The unit ball is not space-cover-decomposable.}

\medskip

Using our construction, we establish the first theorem for polytopes which shows that the higher  dimensional case is quite different from the two dimensional one.

\medskip

\noindent {\bf Theorem B'.} {\em Polytopes are not cover-decomposable in the space and in 
 higher dimensions.}

\medskip

\subsection{The Construction - Proof of Theorems E and E'}\label{concsec2}

In this section, for any $k$ and any polygon $C$ that has a Special pair of wedges, we present a (finite) point set and an indecomposable $k$-fold covering of it by (a finite number of) the translates of the polygon. We formulate (and solve) the problem in its dual form, like we did before. Here we recall how the dualization goes. Fix $O$, the center of gravity of $C$ as our origin in the plane. For the planar set $C$ and a point $p$ in the plane we use $C(p)$ to denote the translate of $C$ by the vector $\vec{Op}$. Let $\bar{C}$ be the reflection through $O$ of $C$. For any point $x$, $x\in C(p_i)$ if and only if $p_i\in \bar{C}(x)$.
To see this, apply a reflection through the midpoint of the segment $xp_i$. This switches $C(p_i)$ and $\bar{C}(x)$, and also switches $p_i$ and $x$.

Consider any collection ${\cal C}=\{C(p_i)\ |\ i\in I\}$ of
translates of $C$ and a point set $X$. The collection ${\cal C}$ covers $x$ at least $k$ times if and only if
$\bar{C}(x)$ contains at least $k$ elements of the set
${S}=\{p_i\ |\ i\in I\}$. Therefore a $k$-fold covering of $X$ transforms into a point set such that for every $x \in X$ the set $\bar{C}(x)$ contains at least $k$ points of $S$.
The required decomposition of ${\cal C}$ exists if and only if
the set $S$ can be colored with two colors
such that every translate $\bar{C}(x)$ that contains at least $k$ elements
of $S$ contains at least one element of each color. Thus constructing a finite system of translates of $\bar{C}$ and a point set where this latter property fails is equivalent to constructing an indecomposable covering using the translates of $C$.

If $C$ has a Special pair of wedges, then so does $\bar{C}$. We will use the following theorem to prove Theorem E.

\begin{thm}\label{thm} For any pair of Special wedges, $V$ and $W$, and for every $k,l$, there is a point set of cardinality ${k+l \choose k}-1$, such that for every coloring of $S$ with red and blue, either there is a translate of $V$ containing $k$ red points and no blue points, or there is a translate of $W$ containing $l$ blue points and no red points.
\end{thm}

\begin{proof}

%\noindent
%{\bf Proof of Theorem \ref{thm}.}\\
%\indent
Without loss of generality, suppose that the wedges are contained in the right halfplane.
%Using induction we will exhibit a set of points such that with any coloring either there is a translate of $V$ containing at least $k$ red points and no blue points or there is a translate of $W$ containing at least $l$ blue points and no red points. 

For $k=1$ the statement is trivial, just take $l$ points such that any one is contained alone in a translate of $W$. Similarly $k$ points will do for $l=1$. Let us suppose that we already have a counterexample for all $k'+l'<k+l$ and let us denote these by $S(k',l')$. The construction for $k$ and $l$ is the following.

Place a point $p$ in the plane and a suitable small scaled down copy of $S(k-1,l)$ left from $p$ such that any translate of $V$ with its apex in the neighborhood of $S(k-1,l)$ contains $p$, but none of the translates of $W$ with its apex in the neighborhood of $S(k-1,l)$ does. Similarly place $S(k,l-1)$ such that any translate of $W$ with its apex in the neighborhood of $S(k,l-1)$ contains $p$, but none of the translates of $V$ with its apex in the neighborhood of $S(k,l-1)$ does. (See Figure \ref{concfig2}.)

\begin{figure}[ht]
\begin{center}
\scalebox{0.55}{\includegraphics{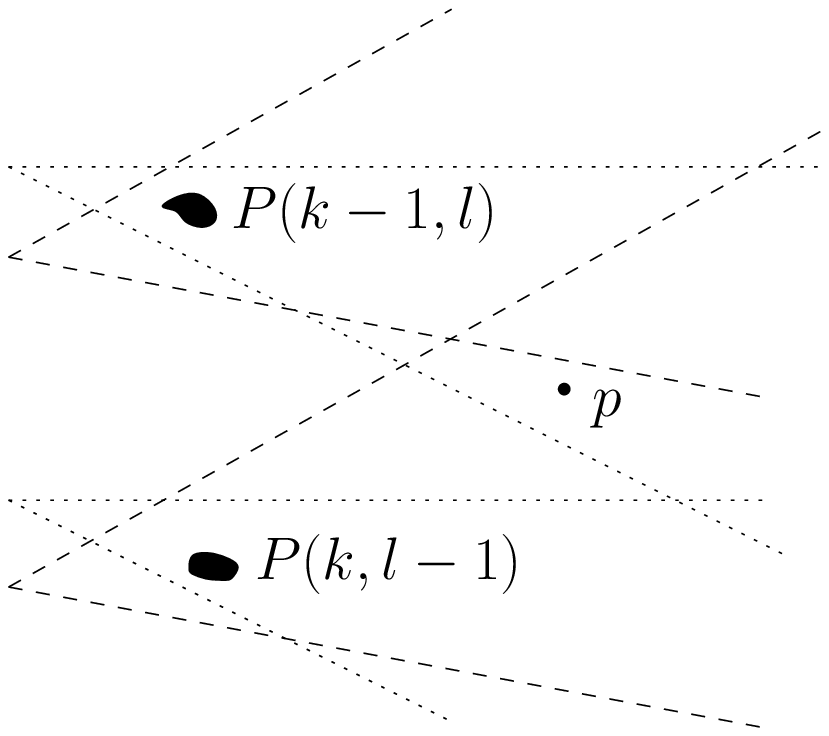}}
\hspace{2cm}
\scalebox{0.65}{\includegraphics{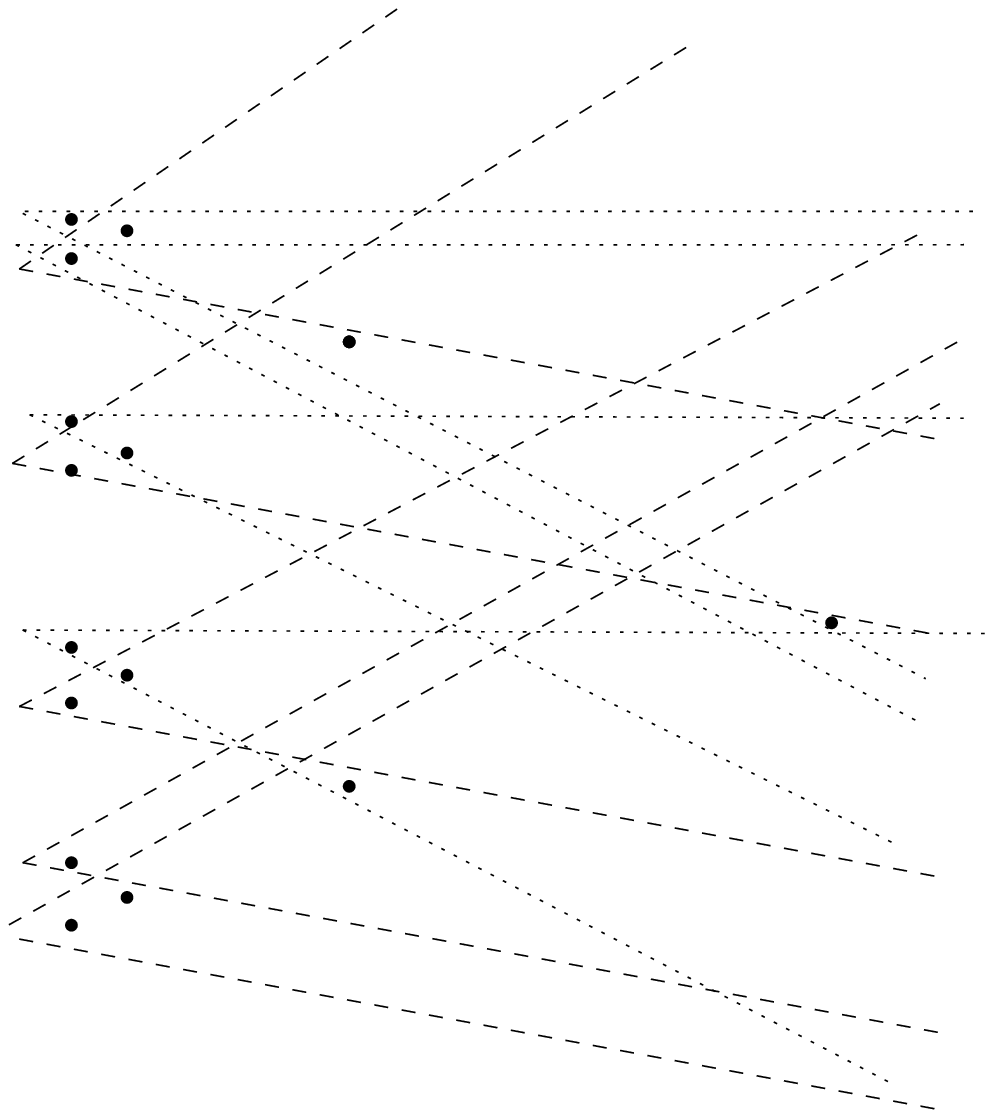}}
\caption{Sketch of one step of the induction and the first few steps.}
\label{concfig2}
\end{center}
\end{figure}

If $p$ is colored red, then\\
\indent
-- either the $S(k-1,l)$ part already contains a translate of $V$ that contains $k-1$ reds and no blues, and it contains $p$ as well, which gives together $k$ red points\\
\indent
-- or the $S(k-1,l)$ part contains a translate of $W$ that contains $l$ blues and no reds and it does not contain $p$.\\
\indent
The same reasoning works for the case when $p$ is colored blue.

Now we can calculate the number of points in $S(k,l)$. For $l=1$ and for $k=1$ we know that $|S(k,1)|=k$ and $|S(1,l)|=l$, while the induction gives $|S(k,l)|=1+|S(k-1,l)|+|S(k,l-1)|$. From this we have $|S(k,l)|={k+l \choose k}-1.$ %This completes the proof of Theorem \ref{thm}.\\
\end{proof}

It is easy to see that if we use this theorem for a pair of Special wedges of $\bar{C}$ and $k=l$, then for every coloring of (a possibly scaled down copy of) the above point set with two colors there is a translate of $\bar{C}$ that contains at least $k$ points, but contains only one of the colors. This is because ``$\bar{C}$ can locally behave like any of its wedges''. Therefore this construction completes the proof of Theorem E. \hfill $\Box$

\begin{remark}\label{remark:oneline} Note that we can even give a finite collection of translates of $V$ and $W$ whose apices all lie on the same line such that one of them will satisfy the conclusion of the theorem. Moreover, this line can be any line that can touch a translate of each wedge in only its apex.
\end{remark}

\begin{remark} We note that for $k=l$ the cardinality of the point set is approximately $4^{k}/\sqrt k$, this significantly improves the previously known construction of Pach, Tardos and Tóth \cite{PTT05} which used approximately $k^k$ points and worked only for quadrilaterals, and in general, for ``even more Special'' pairs of wedges (the ones on the right side of Figure \ref{onlyspecial}). It can be proved that this exponential bound is close to being optimal. Suppose that we have $n$ points and $n<2^{k-2}.$ Since there are two kinds of wedges, there are at most $2n$ essentially different translates that contain $k$ points. There are $2^n$ different colorings of the point set and each translate that contains $k$ points is monochromatic for $2^{n-k+1}$ of the colorings. Therefore, there are at most $2n2^{n-k+1}<2^n$ bad colorings, so there is a coloring with no monochromatic translates.
\end{remark}

Theorem E' follows directly from the next result.

\begin{lem}\label{lem} Every concave polygon that has no parallel sides, has a Special pair of wedges. 
\end{lem}
%\noindent
%{\bf Proof of Lemma \ref{lem}.}\\
%\indent
\begin{proof}
Assume that the statement does not hold for a polygon $C$. There is a touching line $\ell$ to $C$ such that the intersection of $\ell$ and $C$ contains no segments and contains at least two vertices, $v_1$ and $v_2$. (Here we use that $C$ has no parallel sides.)
%of $C$ such that the intersection of the segment connecting them with $C$ is $\{v_1,v_2\}$.
Denote the wedges at $v_i$ by $W_i$. Is the pair $W_1,W_2$ Special? They clearly fulfill the property (i), the only problem that can arise is that the translate of one of the wedges contains the other wedge. This means, without loss of generality, that the angle at $v_1$ contains the angle at $v_2$. Now let us take the two touching lines to $C$ that are parallel to the sides of $W_2$. It is impossible that both of these lines touch $v_2$, because then the touching line $\ell$ would touch only $v_2$ as well. Take a vertex $v_3$ from the touching line (or from one of these two lines) that does not touch $v_2$. (See Figure \ref{concfig3}.) This cannot be $v_1$ because then the polygon would have two parallel sides. Is the pair $W_2,W_3$ Special? They are contained in a halfplane (the one determined by the touching line). This means, again, that the angle at $v_2$ contains the angle at $v_3$. Now we can continue the reasoning with the touching lines to $C$ parallel to the sides of $W_3$, if they would both touch $v_3$, then the touching line $\ell$ would touch only $v_3$. This way we obtain the new vertices $v_4, v_5, \ldots$ what contradicts the fact the $C$ can have only a finite number of vertices. \end{proof}%This completes the proof of Lemma \ref{lem}.\\

\begin{figure}[ht]
\begin{center}
\scalebox{0.8}{\includegraphics{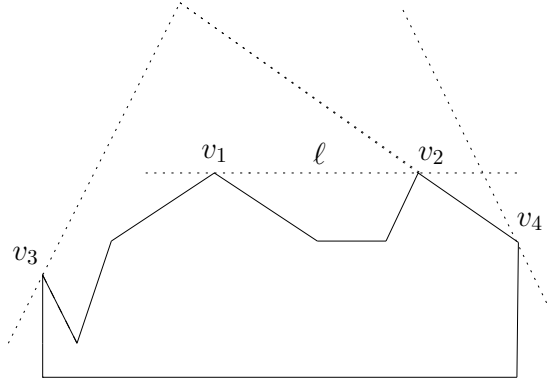}}
\caption{How to find a Special pair of wedges.}
\label{concfig3}
\end{center}
\end{figure}

\subsection{Versions of Cover-decomposability}\label{concsec3}

Here we consider different variants of cover-decomposability and prove relations between them.

\subsubsection{Number of sets: Finite, infinite or more}\label{subsec:total}

We say that a set is {\em finite/countable-cover-decomposable}, if there exists a $k$ such that every $k$-fold covering of any point set by a finite/countable number of its translates is decomposable. 
So by definition we have: totally-cover-decomposable $\Rightarrow$ countable-cover-decomposable $\Rightarrow$ finite-cover-decomposable. But which of these implications can be reversed? We will prove that the first can be for ``nice'' sets.

It is well-known that the plane is hereditary Lindelöf, i.e. if a point set is covered by open sets, then countably many of these sets also cover the point set. It is easy to see that the same holds for $k$-fold coverings as well. This observation implies the following lemma.

\begin{lem} An open set is totally-cover-decomposable if and only if it is countable-cover-decomposable. \hfill $\Box$
\end{lem}

The same holds for ``nice'' closed sets, such as polygons or discs. We say that a closed set $C$ is {\em nice} if there is a $t$ and a set $\mathcal D$ of countably many closed halfdiscs such that if $t$ different translates of $C$ cover a point $p$, then their union covers a halfdisc from $\mathcal D$ centered at $p$ (meaning that $p$ is halving the straight side of the halfdisc) and the union of their interiors covers the interior of the halfdisc. For a polygon, $t$ can be the number of its vertices plus one, $\mathcal D$ can be the set of halfdiscs whose side is parallel to a side of the polygon and has rational length. For a disc, $t$ can be $2$ and $\mathcal D$ can be the set of halfdiscs whose side has a rational slope and a rational length. In fact every convex set is nice.

\begin{claim} Every closed convex set is nice.
\end{claim}
\begin{proof} Some parts of the boundary of the convex set $C$ might be segments, we call these sides. Trivially, every convex set can have only countably many sides. Choose $t=5$ and let the set of halfdiscs $\mathcal D$ be the ones whose side is either parallel to a side of $C$ or its slope is rational and has rational length. Assume that $5$ different translates of $C$ cover a point $p$. Shifting these translates back to $C$, denote the points that covered $p$ by $p_1, \ldots, p_5$. If any of these points is not on the boundary of $C$, we are done. The $p_1p_2p_3p_4p_5$ pentagon has two neighboring angles the sum of whose degrees is strictly bigger than $2\pi$, without loss of generality, $p_1$ and $p_2$. If $p_1p_2$ is also the side of $C$, then the $5$ translates cover a halfdisc whose side is parallel to $p_1p_2$, else they cover one whose side has a rational slope.
\end{proof}

Taking a rectangle verifies that $t=5$ is optimal in the previous proof.

\begin{lem} A nice set is totally-cover-decomposable if and only if it is countable-cover-decomposable.
\end{lem}
\begin{proof} We have to show that if we have an infinite covering of some point set $S$ by the translates of our nice, countable-cover-decomposable set $C$, then we can suitably color the points of $S$. Denote by $S^*$ the points that are covered by $2$ copies of the same translate of the nice set $C$. Color one of these red, the other blue. Now we only have to deal with $S'=S\setminus S^*$ and we can suppose that there is only one copy of each translate. Now instead of coloring these translates, we rather show that we can choose countably many of them such that they still cover every point of $S'$ many times. Using after this that the set is countable-cover-decomposable finishes the proof. So now we show that if there is a set of translates of $C$ that cover every point of $S'$ at least $kt$ times, then we can choose countably many of these translates that cover every point of $S'$ at least $k$ times. It is easy to see that it is enough if we show this for $k=1$ (since we can repeat this procedure $k$ times).

Denote the points that are contained in the interior of a translate by $S_0$. Because of the hereditary Lindelöf property, countably many translates cover $S_0$.
If a point $p\in S'$ is covered $t$ times, then because of the nice property of $C$, a halfdisc from $\mathcal D$ centered at $p$ is covered by these translates. We say that this (one of these) halfdisc(s) belongs to $p$. Take a partition of $S'\setminus S_0$ into countably many sets $S_1\cup S_2\cup\ldots$ such that the $i^{th}$ halfdisc belongs to the points of $S_i$. Now it is enough to show that $S_i$ can be covered by countably many translates. Denote the halfdisc belonging to the points of $S_i$ by $D_i$. Using the hereditary Lindelöf property for $S_i$ and open discs (not halfdiscs!) with the radius of $D_i$ centered at the points of $S_i$, we obtain a countable covering of $S_i$. Now replacing the open discs with closed halfdiscs still gives a covering of $S_i$ because otherwise we would have $p,q\in S_i$ such that $p$ is in the interior of $q+D_i$, but interior of $q+D_i$ is covered by the interiors of translates of $C$, which would imply $p\in S_0$, contradiction. Finally we can replace each of the halfdiscs belonging to the points of $S_i$ by $t$ translates of $C$, we are done.
\end{proof}

Unfortunately, we did not manage to establish any connection among the finite- and the  countable-cover-decomposability. We conjecture that they are equivalent for nice sets (with a possible slight modification of the definition of nice).
If one manages to find such a statement, then it would imply that considering cover-decomposability, it does not matter whether the investigated geometric set is open or closed, as long as it is nice. For example, it is unknown whether closed triangles are cover-decomposable or not. We strongly believe that they are.

\subsubsection{Covering the whole plane}\label{subsec:plane}

%We say that a set is {\em plane-cover-decomposable}, if there exists a $k$ such that every $k$-fold covering of the plane by its translates is decomposable.
Remember that by definition if a set is totally-cover-decomposable, then it is also plane-cover-decomposable. However, the other direction is not always true. For example take the lower halfplane and ``attach'' to its top a pair of Special wedges (see Figure \ref{concfig4}). Then the counterexample using the Special wedges works for a special point set, thus this set is not totally-cover-decomposable, but it is easy to see that a covering of the whole plane can always be decomposed.

For a given polygon $C$, our construction gives a set of points $S$ and a non-decomposable $k$-fold covering of $S$ by translates of $C$. It is not clear when we can extend this covering to a $k$-fold covering of the whole plane such that none of the new translates contain any point of $S$. This would be necessary to ensure that the covering remains non-decomposable.

%The problem is that the dual of our construction in Section 2 does not give a covering of the whole plane and it is not clear when it can be extended to one.

We show that in certain cases it can be extended, but it remains an open problem to decide whether plane- and totally-cover-decomposability are equivalent or not for open polygons/bounded sets. %Earlier Theorem D was known about quadrilaterals. We can generalize this theorem in the following form.

\begin{figure}[ht]
\begin{center}
\scalebox{0.77}{\includegraphics{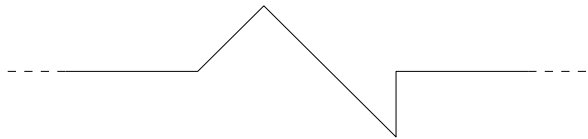}}
\caption{The lower halfplane with a Special pair of wedges at its top.}
\label{concfig4}
\end{center}
\end{figure}

%\begin{thm}[Pach, Tardos, Tóth '05]\cite{PTT} No concave quadrilateral is plane-cover-decomposable.
%\end{thm}

\begin{thm}\label{plane} If a concave polygon $C$ has two Special wedges that have a common locally touching line, and one of the two touching lines parallel to this line is touching $C$ in only a finite number of points (i.e. does not contain a side), then it is not plane-cover-decomposable.
\end{thm}

\begin{proof} Assume, without loss of generality, that this locally touching line is vertical. We just have to extend our construction with the Special wedges into a covering of the whole plane. Or, in the dual, we have to add more points to our construction such that every translate of $C$ will contain at least $k$ points. Of course, to preserve that the construction works, we cannot add more points into those translates that we used in the construction. Otherwise, our argument that the construction is correct, does not work. Because of Remark \ref{remark:oneline} we can suppose that the apices of the wedges all lie on the same vertical line. Therefore, the translates can all be obtained from each other via a vertical shift, because we had a vertical locally touching line to both wedges. Now we can simply add all points that are not contained in any of these original translates. Proving that every translate of $C$ contains at least $k$ points is equivalent to showing that the original translates do not cover any other translate of $C$. It is clear that they could only cover a translate that can be obtained from them via a vertical shift. On the touching vertical line each of the translates has only finitely many points. In the construction we have the freedom to perturbate the wedges a bit vertically, this way we can ensure that the intersection of each other translate (obtainable via a vertical shift) with this vertical line is not contained in the union of the original translates.
\end{proof}

\begin{figure}[ht]
\begin{center}
\scalebox{0.88}{\includegraphics{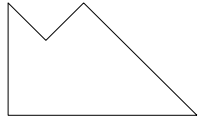}}
\caption{A pentagon that is cover-decomposable but is not the union of a finite number of translates of the same convex polygon.}
\label{concfig5}
\end{center}
\end{figure}

\begin{cor} A pentagon is totally-cover-decomposable if and only if it is plane-cover-decomposable.
\end{cor}
\begin{proof} All totally-cover-decomposable sets are also plane-cover-decomposable. (See an example on Figure \ref{concfig5}.) If our pentagon is not totally-cover-decomposable, then it has a Special pair of wedges and it must also have a touching line that touches it in these Special wedges, thus we can use the previous theorem.
\end{proof}

The same argument does not work for hexagons, for example we do not know whether the hexagon depicted in Figure \ref{pcd}/c is plane-cover-decomposable or not.% (but it is totally-cover-decomposable, the wedges determined by the angles at the vertices $A$ and $E$ give a special pair).
% because of the following example, hexagon $C$, that is depicted in Figure 5, because it has the following property.\\

\begin{figure}[ht]
\begin{center}
\scalebox{0.5}{\includegraphics{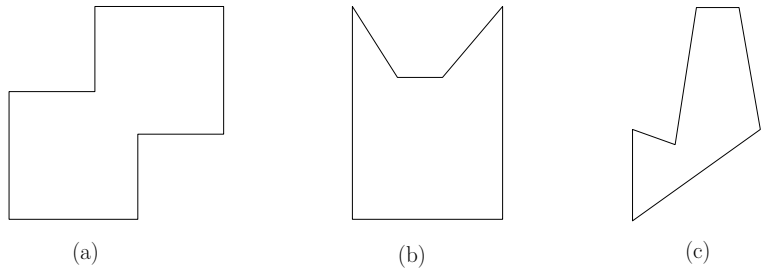}}
%\hskip 2cm
%\scalebox{0.44}{\includegraphics{pcdk.eps}}
\caption{Three different polygons. (a): totally-cover-decomposable (hence, also plane-cover-decomposable),
(b): not plane-cover-decomposable, (hence neither totally-cover-decomposable),
(c): not totally-cover-decomposable, but not known if plane-cover-decomposable.}  
\label{pcd}
\end{center}
\end{figure}

\subsection{Higher Dimensions - Proof of Theorem B'}\label{concsec4}
The situation is different for the space. For any polytope and any $k$, one can construct a $k$-fold covering of the space that is not decomposable. First note that it is enough to prove this result for the three dimensional space, since for higher dimensions we can simply intersect our polytope with a three dimensional space, use our construction for this three dimensional polytope and then extend it naturally. To prove the theorem for three dimensional polytopes, first we need some observations about polygons. Given two polygons and one side of each of them that are parallel to each other, we say that these sides are {\em directedly parallel} if the polygons are on the ``same side'' of the sides (i.e. the halfplane which contains the first polygon and whose boundary contains this side of the first polygon can be shifted to contain the second polygon such that its boundary contains that side of the second polygon). We will slightly abuse this definition and say that a side is directedly parallel if it is directedly parallel to a side of the other polygon. We can similarly define {\em directedly parallel faces} for a single polytope. We say that a face is directedly parallel to another, if they are parallel to each other and the polytope is on the ``same side'' of the faces (e.g. every face is directedly parallel to itself and if the polytope is convex, then to no other face).

\begin{lem}\label{bernadett} Given two convex polygons, both of which have at most two sides that are directedly parallel, there is always a Special pair among their wedges.
\end{lem}
\begin{proof} Take the smallest wedge of the two polygons, excluding the ones that contain a wedge whose sides are both directedly parallel, if it exists. Without loss of generality, we can suppose that the right side of this minimal wedge is not directedly parallel and it is going to the right (i.e. its direction is $(1,0)$), while the left side goes upwards. Take a wedge of the other polygon both of whose sides go upwards (there always must be one since the right side of the first wedge was not directedly parallel). If this second wedge is not contained in the first, we found a Special pair. If the second wedge is contained in the first wedge, then because of the minimality of the first wedge, we get a contradiction.
\end{proof}

\medskip

\noindent {\bf Theorem B'.} {\em Polytopes are not cover-decomposable in the space and in 
 higher dimensions.}

\medskip

\begin{proof} We will, as usual, work in the dual case. This means that to prove that our polytope $C$ is not totally-cover-decomposable, we will exhibit a point set for any $k$ such that we cannot color it with two colors such that any translate of $C$ that contains at least $k$ points contains both colors. These points will be all in one plane, and the important translates of $C$ will intersect this plane either in a concave polygon or in one of two convex polygons. It is enough to show that this concave polygon is not cover-decomposable or that among the wedges of these convex polygons there is a Special pair.\\Take a plane $\pi$ that is not parallel to any of the segments determined by the vertices of $C$. The touching planes of $C$ parallel to $\pi$ are touching $C$ in one vertex each, $A$ and $B$. Denote the planes parallel to $\pi$ that are very close to $A$ and $B$ and intersect $C$, by $\pi_A$ and $\pi_B$. Denote $C\cap \pi_A$ by $C_A$ and $C\cap \pi_B$ by $C_B$. Now we will have two cases.\\
\indent
{\bf Case 1.} $C_A$ or $C_B$ is concave.\\
Without loss of generality, assume $C_A$ is concave. Then since no two faces of $C$ incident to $A$ can be parallel to each other, with a perturbation of $\pi_A$ we can achieve that the sides of $C_A$ are not parallel. After this, using Theorem E', we are done.\\\indent
{\bf Case 2.} Both $C_A$ and $C_B$ are convex.\\
Now by perturbing $\pi$, we cannot necessarily achieve that $C_A$ and $C_B$ have no parallel sides, but we can achieve that they have at most two directedly parallel sides. This is true because there can be at most two pairs of faces that are directedly parallel to each other and one of them is incident to $A$, the other to $B$, %(the definition is analogous to the definition of parallel sides), 
since $A$ is touched from above, $B$ from below by the plane parallel to $\pi$. Therefore $C_A$ and $C_B$ satisfy the conditions of Lemma \ref{bernadett}, this finishes the proof of totally-cover-decomposability.

To prove non-space-cover-decomposability, just as in  the proof of Theorem \ref{plane}, we have to add more points to the constructions, such that every translate will contain at least $k$ points, but we do not add any points to the original translates of our construction. This is the same as showing that these original translates do not cover any other translate. Note that there are two types of original translates (depending on which wedge of it we use) and translates of the same type can be obtained from each other via a shift that is parallel to the side of the halfplane in $\pi$ that contains our Special wedges. This means that the centers of all the original translates lie in one plane. With a little perturbation of the construction, we can achieve that this plane is in general position with respect to the polytope. But in this case it is clear that the translates used in our construction cannot cover any other translate, this proves space-cover-decomposability.
\end{proof}

\subsection{Concluding Remarks}

A lot of questions remain open. In three dimensions, neither polytopes, nor unit balls are cover-decomposable. Is there any nice (e.g. open and bounded) set in three dimensions that is cover-de\-com\-posable? Maybe such nice sets exist only in the plane.

\begin{conj}\label{conj:threed} Three-dimensional convex sets are not cover-decomposable.
\end{conj}

In Section \ref{subsec:total} we have seen that interesting covers with translates of nice sets only use countably many translates. We could not prove, but conjecture, that every cover can be somehow reduced to a locally finite cover. Is it true that if a nice set is finite-cover-decomposable, then it is also countable-cover-decomposable? This would have implications about the cover-decomposability of closed sets.

\begin{conj}\label{conj:closed} Closed, convex polygons are cover-decomposable.
\end{conj}

In Section \ref{subsec:plane} we have seen that our construction is not naturally extendable to give an indecomposable covering of the whole plane. Maybe the reason for this is that it is impossible to find such a covering. 

\begin{quest}\label{quest:pcd} Are there polygons that are not totally-cover-decomposable but plane-cover-decomposable?
\end{quest}

%% file: introslope.tex
A planar layout of a graph $G$ is called a {\em drawing} if the vertices of $G$ are represented  by distinct points in the plane and every edge is represented by a continuous arc connecting the corresponding pair of points and not passing through any other point representing a vertex \cite{DETT99}. If it leads to no confusion, in notation and terminology we make no distinction between a vertex and the corresponding point, and between an edge and the corresponding arc. If the edges are represented by line segments, the drawing is called a {\em straight-line drawing}. The {\em slope} of an edge in a straight-line drawing is the slope of the corresponding segment.

Wade and Chu \cite{WC94} introduced the following graph parameter:
The {\em slope number} of a graph $G$ is the smallest number $s$
with the property that $G$ has a straight-line drawing with edges
of at most $s$ distinct slopes. Let us compare this with two other well studied graph parameters. The {\em thickness} of a graph $G$
is defined as the smallest number of planar subgraphs it can be decomposed into \cite{MuOS}. It is one of the
several widely known graph parameters that measures how
far $G$ is from being planar.
The {\em geometric thickness} of $G$, defined as the
smallest number of {\em crossing-free} subgraphs of a straight-line
drawing of $G$ whose union is $G$, is another similar notion. \cite{Ka}.
It follows directly from
the definitions that the thickness of any graph is at most as
large as its geometric thickness, which, in turn, cannot exceed
its slope number.
For many interesting results about these
parameters, consult \cite{DiEH, DEK04, DSW04, DuW, E04, HuSW}.

Obviously, if
$G$ has a vertex of degree $d$, then its slope number is at least
$\lceil d/2\rceil$, because, according to the above definitions,
in a proper drawing two edges are not allowed to partially
overlap. The question arises whether the slope number can be
bounded from above by any function of the maximum degree $d$ (see
\cite{DSW07}). Bar\'at, Matou\v sek, and Wood \cite{BMW06} and,
independently, in our paper with Pach \cite{PP06}, we proved using a counting argument that the answer is no for $d\ge 5$. We present this latter proof, which gives a better bound, in Section \ref{sec:slopedeg5}. We show that for any $d\ge 5$ and $n$, there
exist a graph with $n$ vertices of maximum degree
$d$, whose slope number is at least $n^{\frac{1}{2}-\frac{1}{d-2}-o(1)}$.
Since then this bound was improved for $d\ge 9$ by Dujmovi\'c, Suderman and Wood
 \cite{DSW07},  they showed that the slope number is at least $n^{1-\frac{8+\epsilon}{d+4}}$. 
  Note that for smaller $d$'s our bound is still the best. 

Trivially, every graph of maximum
degree two has slope number at most three. The case $d= 3$ was solved in our paper \cite{KPPT08}, in which we prove that every cubic graph\footnote{A graph is {\it cubic} if its maximum degree is at most 3.} can be drawn with 5 slopes. This proof is presented in Section \ref{sec:slopenum}. Later, Mukkamala and Szegedy \cite{MSz07} showed that 4 slopes suffice if the graph is connected. 
However, for disconnected graphs, still five slopes is the best bound. This is because it cannot be guaranteed that different components are drawn with the same four slopes. It would be interesting to decide whether four fixed directions always suffice.

The case $d=4$ remains an interesting open problem.

\begin{conj}\label{slope4} The slope number of graphs with maximum degree 4 is unbounded.
\end{conj}

In Section \ref{sec:slopepar} we investigate a similar notion, called {\em slope parameter}, first defined as follows by Ambrus,
Bar\'at, and P. Hajnal \cite{AmBH}. Given a set $P$ of points in
the plane and a set $\Sigma$
of slopes, define $G(P,\Sigma)$ as
the graph on the vertex set $P$, in which two vertices $p,q\in P$
are connected by an edge if and only if the slope of the line $pq$
belongs to $\Sigma$. The {\em slope parameter} $s(G)$ of $G$ is
the size of the smallest set $\Sigma$
of slopes such that $G$ is
isomorphic to $G(P,\Sigma)$ for a suitable set of points $P$ in
the plane. This definition was motivated by the fact that all
connections (edges) in an electrical circuit (graph) $G$ can be
easily realized by the overlay of $s(G)$ finely striped
electrically conductive layers.

The slope parameter, $s(G)$, is closely related to the slope number.
For instance, for triangle-free
graphs, $s(G)$ is at least as large as the slope number of $G$, thus also bigger than the thickness and the geometric thickness. Indeed, in the drawing
realizing the slope parameter, there are no three points on a line,
so this drawing proves that the slope number is smaller or equal to the
slope parameter, the only difference being that in case of the slope number it is not obligatory to connect two vertices if their slope is in $\Sigma$.
The slope parameter of a triangle-free graph is also at least its {\em edge chromatic number}, $\chi'(G)$, as there can be at most one edge with the same slope from any vertex.

On the other hand, the slope parameter sharply
differs from other parameters in the sense that the slope parameter of a
complete graph on $n$ vertices is {\em one}, while the thickness, the geometric thickness, and the slope number of $K_n$ tend to infinity as $n\rightarrow\infty$. Jamison
\cite{Ja} proved that the slope number of $K_n$ is $n$.

Our main result in Section \ref{sec:slopepar} is that the slope parameter of every cubic graph is also bounded. In our drawing no three vertices will be collinear, therefore as a corollary we obtain another proof for the fact that the slope number of cubic graphs is bounded, with a worse constant.

Finally, in Section \ref{sec:slopeplan} we investigate the {\em planar slope number} of bounded degree planar graphs. The planar slope number of a planar graph $G$ is the smallest number $s$ with the property that $G$ has a straight-line drawing with non-crossing edges
of at most $s$ distinct slopes. We prove that any bounded degree planar graph has a bounded planar slope number. Then we investigate planar drawings where we allow one or two bends on each edge, in which cases we prove better bounds. For the exact statement of our results, see the beginning of Section \ref{sec:slopeplan}.

%% file: slopedeg5.tex
This section is based on our paper with J\'anos Pach, Bounded-degree graphs can have arbitrarily large slope numbers \cite{PP06}.

If it creates no confusion, the vertex (edge) of $G$ and the
point (segment) representing it will be denoted by the
same symbol.  Dujmovi\'c et al. \cite{DSW04} asked whether the
slope parameter of bounded-degree graphs can be
arbitrarily large. The following short argument shows that
the answer is yes for graphs of degree at most {\em
five}. 

\medskip

Define a ``frame'' graph $F$ on the vertex set
$\{1,\ldots,n\}$ by connecting vertex 1 to 2 by an edge
and connecting every $i>2$ to $i-1$ and $i-2$. Adding a 
perfect matching $M$ between these $n$ points, we obtain 
a graph $G_M:=F\cup M$. The number of different matchings 
is at least $(n/3)^{n/2}$. Let $G$ denote the huge graph 
obtained by taking the union of disjoint copies of all 
$G_M$. Clearly, the maximum degree of the vertices of $G$ 
is {\em five}. Suppose that $G$ can be drawn using at 
most $S$ slopes, and fix such a drawing. 

For every edge $ij\in M$, label the points in $G_M$
corresponding to $i$ and $j$ by the slope of $ij$ in
the drawing. Furthermore, label  each frame edge 
$ij\; (|i-j|\le 2)$ by its slope. Notice that no two 
components of $G$ receive the same labeling. Indeed, 
up to translation and scaling, the labeling of the 
edges uniquely determines the positions of the points 
representing the vertices of $G_M$. Then the labeling
of the vertices uniquely determines the edges belonging
to $M$. Therefore, the number of different possible 
labelings, which is $S^{|F|+n}<S^{3n}$, is an upper bound for
the number of components of $G$. On the other hand, we
have seen that the number of components (matchings) is
at least $(n/3)^{n/2}$. Thus, for any $S$ we obtain a
contradiction, provided that $n$ is sufficiently
large.  \hfill $\Box$

\medskip

With some extra care one can refine this argument to obtain 

\begin{thm} For any $d\ge 5$ and $n$, there
exist a graph with $n$ vertices of maximum degree
$d$, whose slope number is at least $n^{\frac{1}{2}-\frac{1}{d-2}-o(1)}$.
\end{thm}

\noindent
{\bf Proof.} Now instead of a matching, we add to the
frame $F$ in every possible way a $(d-4)$-regular graph
$R$ on the vertex set $\{1,\ldots,n\}$. Thus, we obtain
at least $(cn/d)^{(d-4)n/2}$ different graphs
$G_R:=F\cup R$, each having maximum degree at most $d$
(here $c>0$ is a constant; see e.g. \cite{BC78}).
Suppose that each $G_R$ can be drawn using $S$ slopes
$\sigma_1<\ldots<\sigma_S$. Now we cannot insist that
these slopes are the same for all $G_R$, therefore, these
numbers will be regarded as variables.

Fix a graph $G_R=F\cup R$ and one of its drawings with the above
properties, in which vertex 1 is mapped into the origin and vertex
2 is mapped into a point whose $x$-coordinate is 1. Label every
edge belonging to $F$ by the symbol $\sigma_k$ representing its
slope. Furthermore, label each vertex $j$ with a $(d-4)$-tuple of
the $\sigma_k$s: with the symbols corresponding to the slopes of
the $d-4$ edges incident to $j$ in $R$ (with possible repetition).
Clearly, the total number of possible labelings of the frame edges
and vertices is at most $S^{|F|+(d-4)n} < S^{(d-2)n}$. Now the
labeling itself does not necessarily identify the graph $G_R$,
because we do not know the {\em actual} values of the slopes
$\sigma_k$.

However, we can show that the number of different $G_R$s that
receive the same labeling cannot be too large. To prove this,
first notice that for a fixed labeling of the edges of the frame,
the coordinates of every vertex $i$ can be expressed as the ratio
of two polynomials of degree at most $n$ in the variables
$\sigma_1,\ldots, \sigma_S$. Indeed, let $\sigma(ij)$ denote the
label of $ij\in F$, and let $x(i)$ and $y(i)$ denote the
coordinates of vertex $i$. Since, by assumption, we have
$x(1)=y(1)=0$ and $x(2)=1$, we can conclude that
$y(2)=\sigma(12)$. We have the following equations for the
coordinates of $3$:
$$y(3)-y(1)=\sigma(13)(x(3)-x(1)), \; \; \; \;
y(3)-y(2)=\sigma(23)(x(3)-x(2)).$$ Solving them, we
obtain
$$x(3)=\frac{\sigma(12)-\sigma(23)}{\sigma(13)-\sigma(23)},
\; \; \; \;
y(3)=\frac{\sigma(13)(\sigma(12)-\sigma(23))}{\sigma(13)-\sigma(23)},$$
and so on. In particular,
$x(i)=\frac{Q_i(\sigma_1,\ldots,\sigma_S)}{Q'_i(\sigma_1,\ldots,\sigma_S)},$
for suitable polynomials $Q_i$ and $Q'_i$ of degree at most $i-1$.
Moreover, $Q'_j$ is a multiple of $Q'_i$ for all $j>i$.

Since $$x(i)-x(j)=\frac{Q_i{\frac{Q'_j}{Q'_i}}-Q_j}{Q'_j},$$ we
can decide whether the image of $i$ is to the left of the image of
$j>i$, to the right of it, or they have the same $x$-coordinate,
provided that we know the ``sign pattern'' of the polynomials
$P'_{ij}:={Q_i{\frac{Q'_j}{Q'_i}}-Q_j}$ and $Q'_j$, i.e., we know
which of them are positive, negative, or zero.

Now if we also know that $\sigma_k$ is one of the labels
associated with vertex $i$, the condition that the line connecting
$i$ and $j$ has slope $\sigma_k$ can be rewritten as
$$\frac{y(i)-y(j)}{x(i)-x(j)}-\sigma_k
=\frac{\sigma(1i)Q_iQ_j'-\sigma(1j)Q_i'Q_j}{Q_iQ_j'-Q_i'Q_j}-\sigma_k=0,$$
that is, as a polynomial equation
$P_{ijk}(\sigma_1,\ldots,\sigma_S)=0$ of degree at most $2n$. For
a fixed labeling of the frame edges and vertices, there are $d-4$
labels $k$ associated with a vertex $i$, so that the number of
these polynomials $P_{ijk}$ is at most $(d-4)n(n-1)$. Thus,
together with the ${n\choose 2}+n$ polynomials $P'_{ij}$ and
$Q'_j$, we have fewer than $dn^2$ polynomials, each of degree at
most $2n$.

It is easy to verify that, for any fixed labeling, the sign
pattern of these polynomials uniquely determines the graph $G_R$.
(Observe that if the label of a vertex $i$ is a $(d-4)$-tuple
containing the symbol $\sigma_k$, then from the sign pattern of
the above polynomials we can reconstruct the sequence of all
vertices that belong to the line of slope $\sigma_k$ passing
through $i$, from left to right. From this sequence, we can select
all elements whose label contains $\sigma_k$, and determine all
edges of $R$ along this line.)

To conclude the proof, we need the Thom-Milnor theorem \cite{BPR03}:
Given $N$ polynomials in $S\le N$ variables, each of degree at most
$2n$, the number of sign patterns determined by them is at most
$\left(CNn/S\right)^S$, for a suitable constant $C>0$.

In our case, the number of graphs $G_R$ is at most the number of
labelings ($< S^{(d-2)n}$) multiplied by the maximum number of
sign patterns of the above $<dn^2$ polynomials of degree at most
$2n$. By the Thom-Milnor theorem, this latter quantity is smaller
than $\left(Cdn^3\right)^S$. Thus, the number of $G_R$s is at most
$S^{(d-2)n}(Cdn^3)^S$. Comparing this to the lower bound
$(cn/d)^{(d-4)n/2}$ stated in the first paragraph of the proof, we
obtain that $S\ge n^{\frac{1}{2}-\frac{1}{d-2}-o(1)},$ as
required.Now instead of a matching, we add to the
frame $F$ in every possible way a $(d-4)$-regular graph
$R$ on the vertex set $\{1,\ldots,n\}$. Thus, we obtain
at least $(cn/d)^{(d-4)n/2}$ different graphs
$G_R:=F\cup R$, each having maximum degree at most $d$
(here $c>0$ is a constant; see e.g. \cite{BC78}).
Suppose that each $G_R$ can be drawn using $S$ slopes
$\sigma_1<\ldots<\sigma_S$. Now we cannot insist that
these slopes are the same for all $G_R$, therefore, these
numbers will be regarded as variables.

Fix a graph $G_R=F\cup R$ and one of its drawings with the above
properties, in which vertex 1 is mapped into the origin and vertex
2 is mapped into a point whose $x$-coordinate is 1. Label every
edge belonging to $F$ by the symbol $\sigma_k$ representing its
slope. Furthermore, label each vertex $j$ with a $(d-4)$-tuple of
the $\sigma_k$s: with the symbols corresponding to the slopes of
the $d-4$ edges incident to $j$ in $R$ (with possible repetition).
Clearly, the total number of possible labelings of the frame edges
and vertices is at most $S^{|F|+(d-4)n} < S^{(d-2)n}$. Now the
labeling itself does not necessarily identify the graph $G_R$,
because we do not know the {\em actual} values of the slopes
$\sigma_k$.

However, we can show that the number of different $G_R$s that
receive the same labeling cannot be too large. To prove this,
first notice that for a fixed labeling of the edges of the frame,
the coordinates of every vertex $i$ can be expressed as the ratio
of two polynomials of degree at most $n$ in the variables
$\sigma_1,\ldots, \sigma_S$. Indeed, let $\sigma(ij)$ denote the
label of $ij\in F$, and let $x(i)$ and $y(i)$ denote the
coordinates of vertex $i$. Since, by assumption, we have
$x(1)=y(1)=0$ and $x(2)=1$, we can conclude that
$y(2)=\sigma(12)$. We have the following equations for the
coordinates of $3$:
$$y(3)-y(1)=\sigma(13)(x(3)-x(1)), \; \; \; \;
y(3)-y(2)=\sigma(23)(x(3)-x(2)).$$ Solving them, we
obtain
$$x(3)=\frac{\sigma(12)-\sigma(23)}{\sigma(13)-\sigma(23)},
\; \; \; \;
y(3)=\frac{\sigma(13)(\sigma(12)-\sigma(23))}{\sigma(13)-\sigma(23)},$$
and so on. In particular,
$x(i)=\frac{Q_i(\sigma_1,\ldots,\sigma_S)}{Q'_i(\sigma_1,\ldots,\sigma_S)},$
for suitable polynomials $Q_i$ and $Q'_i$ of degree at most $i-1$.
Moreover, $Q'_j$ is a multiple of $Q'_i$ for all $j>i$.

Since $$x(i)-x(j)=\frac{Q_i{\frac{Q'_j}{Q'_i}}-Q_j}{Q'_j},$$ we
can decide whether the image of $i$ is to the left of the image of
$j>i$, to the right of it, or they have the same $x$-coordinate,
provided that we know the ``sign pattern'' of the polynomials
$P'_{ij}:={Q_i{\frac{Q'_j}{Q'_i}}-Q_j}$ and $Q'_j$, i.e., we know
which of them are positive, negative, or zero.

Now if we also know that $\sigma_k$ is one of the labels
associated with vertex $i$, the condition that the line connecting
$i$ and $j$ has slope $\sigma_k$ can be rewritten as
$$\frac{y(i)-y(j)}{x(i)-x(j)}-\sigma_k
=\frac{\sigma(1i)Q_iQ_j'-\sigma(1j)Q_i'Q_j}{Q_iQ_j'-Q_i'Q_j}-\sigma_k=0,$$
that is, as a polynomial equation
$P_{ijk}(\sigma_1,\ldots,\sigma_S)=0$ of degree at most $2n$. For
a fixed labeling of the frame edges and vertices, there are $d-4$
labels $k$ associated with a vertex $i$, so that the number of
these polynomials $P_{ijk}$ is at most $(d-4)n(n-1)$. Thus,
together with the ${n\choose 2}+n$ polynomials $P'_{ij}$ and
$Q'_j$, we have fewer than $dn^2$ polynomials, each of degree at
most $2n$.

It is easy to verify that, for any fixed labeling, the sign
pattern of these polynomials uniquely determines the graph $G_R$.
(Observe that if the label of a vertex $i$ is a $(d-4)$-tuple
containing the symbol $\sigma_k$, then from the sign pattern of
the above polynomials we can reconstruct the sequence of all
vertices that belong to the line of slope $\sigma_k$ passing
through $i$, from left to right. From this sequence, we can select
all elements whose label contains $\sigma_k$, and determine all
edges of $R$ along this line.)

To conclude the proof, we need the Thom-Milnor theorem \cite{BPR03}:
Given $N$ polynomials in $S\le N$ variables, each of degree at most
$2n$, the number of sign patterns determined by them is at most
$\left(CNn/S\right)^S$, for a suitable constant $C>0$.

In our case, the number of graphs $G_R$ is at most the number of
labelings ($< S^{(d-2)n}$) multiplied by the maximum number of
sign patterns of the above $<dn^2$ polynomials of degree at most
$2n$. By the Thom-Milnor theorem, this latter quantity is smaller
than $\left(Cdn^3\right)^S$. Thus, the number of $G_R$s is at most
$S^{(d-2)n}(Cdn^3)^S$. Comparing this to the lower bound
$(cn/d)^{(d-4)n/2}$ stated in the first paragraph of the proof, we
obtain that $S\ge n^{\frac{1}{2}-\frac{1}{d-2}-o(1)},$ as
required.
\hfill $\Box$

%% file: slopenum.tex
This section is based on our paper with Bal\'azs Keszegh, J\'anos Pach and G\'eza T\'oth, Drawing cubic graphs with at most five slopes \cite{KPPT08}.

Our main result is

\begin{thm}\label{slopenum1} Every graph of maximum degree at
most three has slope number at most five.
\end{thm}

Our terminology is somewhat unorthodox: by the {\em slope} of a
line $\ell$, we mean the angle $\alpha$ modulo $\pi$ such that a
counterclockwise rotation through $\alpha$ takes the $x$-axis to a
position parallel to $\ell$. The slope of an edge (segment) is the
slope of the line containing it. In particular, the slopes of the
lines $y=x$ and $y=-x$ are $\pi/4$ and $-\pi/4$, and they are
called {\em Northeast} (or Southwest) and {\em Northwest} (or
Southeast) lines, respectively.

For any two points $p_1=(x_1,y_1), p_2=(x_2,y_2)\in {\bf R}^2$, we
say that $p_2$ is {\em to the North} (or {\em to the South} of $p_1$ if
$x_2=x_1$ and $y_2>y_1$ (or $y_2<y_1$). Analogously, we say that
$p_2$ is to the Northeast (to the Northwest) of $p_1$ if $y_2>y_1$
and $p_1p_2$ is a Northeast (Northwest) line. Directions are often
abbreviated by their first letters: N, NE, E, SE, etc. These four
directions are referred to as {\em basic}. That is, a line $\ell$
is said to be of one of the four basic directions if $\ell$ is
parallel to one of the axes or to one of the NE and NW lines $y=x$
and $y=-x$.

The main tool of our proof is the following result of independent
interest.

\begin{thm}\label{slopenum2} Let $G$ be a connected graph that
is not a cycle and whose every vertex has degree at most three.
Suppose that $G$ has at least one vertex of degree less than
three, and denote by $v_1, ..., v_m$ the vertices of degree at most two $\;(m\ge 1)$.

Then, for any sequence $x_1, x_2, \ldots , x_m$ of real numbers,
linearly independent over the rationals, $G$ has a straight-line
drawing with the following properties:

\noindent(1) {\em Vertex $v_i$ is mapped into a point with
$x$-coordinate $x(v_i)=x_i\; (1\le i\le m)$;}

\noindent(2) {\em The slope of every edge is $0, \pi/2, \pi/4,$ or
$-\pi/4.$}

\noindent(3) {\em No vertex is to the North of any vertex of
degree {\it two}.}

\noindent(4) {\em No vertex is to the North or to the Northwest of
any vertex of degree {\it one}.}
\end{thm}

Before this theorem only the following special cases were known.

It was shown by Dujmovi\'c at al. \cite{DESW07} that every planar
graph with maximum degree three has a drawing with non-crossing
straight-line edges of at most three different slopes, except that
three edges of the outer-face may have a bend.

Max Engelstein \cite{En05}, a student from Stuyvesant High School,
New York has shown that every graph of maximum degree {\em three}
that has a Hamiltonian cycle can be drawn with edges of at most
{\em five} different slopes.

\subsection{Embedding Cycles}

Let $C$ be a straight-line drawing of a cycle in the plane. A
vertex $v$ of $C$ is said to be a {\em turning point} if the
slopes of the two edges meeting at $v$ are not the same.

We start with two simple auxiliary statements.

\begin{lem}\label{slopenumlem21} Let $C$ be a straight-line drawing
of a cycle such that the slope of every edge is $0$, $\pi/4$, or
$-\pi/4$. Then the $x$-coordinates of the vertices of $C$ are {\it
not} independent over the rational numbers.

Moreover, there is a vanishing linear combination of the
$x$-coordinates of the vertices, with %at least
as many nonzero
(rational) coefficients as many turning points $C$ has.
\end{lem}

\noindent{\bf Proof.} Let $v_1, v_2,\ldots, v_n$ denote the
vertices of $C$ in cyclic order ($v_{n+1}=v_1$). Let $x(v_i)$ and
$y(v_i)$ be the coordinates of $v_i$. For any $i\; (1\le i\le n)$,
we have
$y(v_{i+1})-y(v_i)=\lambda_i\left(x(v_{i+1})-x(v_i)\right),$ where
$\lambda_i=0,1,$ or $-1$, depending on the slope of the edge
$v_iv_{i+1}$. Adding up these equations for all $i$, the left-hand
sides add up to zero, while the sum of the right-hand sides is a
linear combination of the numbers $x(v_1), x(v_2),\ldots, x(v_n)$
with integer coefficients of absolute value at most {\em two}.

Thus, we are done with the first statement of the lemma, unless
all of these coefficients are zero. Obviously, this could happen
if and only if $\lambda_1=\lambda_2=\ldots=\lambda_n$, which is
impossible, because then all points of $C$ would be collinear,
contradicting our assumption that in a proper {\em straight-line
drawing} no edge is allowed to pass through any vertex other than
its endpoints.

To prove the second statement, it is sufficient to notice that the
coefficient of $x(v_i)$ vanishes if and only if $v_i$ is not a
turning point. \hfill $\Box$
\medskip

Lemma \ref{slopenumlem21} shows that Theorem \ref{slopenum2} does not hold if $G$ is a cycle.
Nevertheless, according to the next claim, cycles satisfy a very
similar condition. Observe, that the main difference is that here 
we have an exceptional vertex, denoted by $v_0$.

\begin{lem}\label{slopenumlem22} Let $C$ be a cycle with vertices
$v_0, v_1, \ldots , v_m$, in this cyclic order.

Then, for any real numbers $x_1, x_2, \ldots , x_m$, linearly
independent over the rationals, $C$ has a straight-line drawing
with the following properties:

\noindent(1) {\em Vertex $v_i$ is mapped into a point with
$x$-coordinate $x(v_i)=x_i\; (1\le i\le m)$;}

\noindent(2) {\em The slope of every edge is $0, \pi/4,$ or
$-\pi/4.$}

\noindent(3) {\em No vertex is to the North of any other vertex.}

\noindent(4) {\em No vertex has a larger $y$-coordinate than
$y(v_0)$.}
\end{lem}

\noindent {\bf Proof.}
We can assume without loss of generality that $x_2>x_1$.
Place $v_1$ at any point $(x_1,0)$ of the
$x$-axis. Assume that for some $i<m$, we have already determined
the positions of $v_1, v_2, \ldots v_{i}$, satisfying conditions
(1)--(3). If $x_{i+1}>x_i$, then place $v_{i+1}$ at the (unique)
point Southeast of $v_i$, whose $x$-coordinate is $x_{i+1}$. If
$x_{i+1}<x_i$, then put $v_{i+1}$ at the point West of $x_i$,
whose $x$-coordinate is $x_{i+1}$. Clearly, this placement of
$v_{i+1}$ satisfies (1)--(3), and the segment $v_iv_{i+1}$ does
not pass through any point $v_j$ with $j<i$.

After $m$ steps, we obtain a noncrossing straight-line drawing of
the path $v_1v_2\ldots v_{m}$, satisfying conditions (1)--(3). We
still have to find a right location for $v_0$. Let $R_W$ and
$R_{SE}$ denote the rays (half-lines) starting at $v_1$ and
pointing to the West and to the Southeast. Further, let $R$ be the
ray starting at $v_m$ and pointing to the Northeast. It follows
from the construction that all points $v_2, \ldots, v_{m}$ lie in
the convex cone below the $x$-axis, enclosed by the rays $R_W$ and
$R_{SE}$.

Place $v_0$ at the intersection point of $R$ and the
$x$-axis. Obviously, the segment $v_mv_0$ does not pass through
any other vertex $v_j\;(0<j<m)$. Otherwise, we could find a
drawing of the cycle $v_jv_{j+1}\ldots v_m$ with slopes $0,
\pi/4,$ and $-\pi/4$. By Lemma \ref{slopenumlem21}, this would imply that the
numbers $x_j, x_{j+1}, \ldots, x_m$ are {\em not} independent over
the rationals, contradicting our assumption. It is also clear that
the horizontal segment $v_0v_1$ does not pass through any vertex
different from its endpoints because all other vertices are
below the horizontal line determined by $v_0v_1$.
Hence, we obtain a proper
straight-line drawing of $C$ satisfying conditions (1),(2), and
(4).

It remains to verify (3). The only thing we have to check is that
$x(v_0)$ does not coincide with any other $x(v_i)$. Suppose it
does, that is, $x(v_0)=x(v_i)=x_i$ for some $i>0$. By the second
statement of Lemma \ref{slopenumlem21}, there is a vanishing linear combination
$$\lambda_0x(v_0)+\lambda_1x_1+\lambda_2x_2+\ldots+\lambda_mx_m=0$$
with rational coefficients $\lambda_i$, where the number of
nonzero coefficients is at least the number of turning points,
which cannot be smaller than {\em three}. Therefore, if in this
linear combination we replace $x(v_0)$ by $x_i$, we still obtain a
nontrivial rational combination of the numbers $x_1, x_2,\ldots,
x_m$. This contradicts our assumption that these numbers are
independent over the rationals. \hfill $\Box$
\medskip

\subsection{Subcubic Graphs - Proof of Theorem \ref{slopenum2}}

First we settle Theorem \ref{slopenum2} in a special case.
 
\begin{lem}\label{slopenumlem31} Let $m,k\ge 2$ and let $G$ be a graph consisting of two
disjoint cycles, $C=\{v_0, v_1, \ldots , v_m\}$ and $C'=\{v_0', v_1', \ldots , v_m'\}$,
connected by a single edge $v_0v'_0$. 

Then, for any sequence $x_1, x_2, \ldots , x_m, x'_1, x'_2, \ldots , x'_k$  
of real numbers, linearly independent over the rationals, $G$ has a 
straight-line drawing satisfying the following conditions:

\noindent(1) {\em The vertices $v_i$ and $v'_j$ are mapped into points with
$x$-coordinates $x(v_i)=x_i\; (1\le i\le m)$ and $x(v_j)=x'_j\; (1\le j\le k)$.}

\noindent(2) {\em The slope of every edge is $0, \pi/2, \pi/4,$ or $-\pi/4.$}

\noindent(3) {\em No vertex is to the North of any vertex of degree {\it two}.}

\end{lem}

\noindent {\bf Proof.}
Apply Lemma \ref{slopenumlem22} to cycle $C$ with vertices $v_0, v_1, \ldots , v_m$ and with 
assigned $x$-coordinates $x_1, x_2, \ldots , x_m$, and analogously,
to the cycle $C'$, with vertices $v'_0, v'_1, \ldots , v'_k$ and with assigned 
$x$-coordinates $x'_1, x'_2, \ldots , x'_k$. For simplicity, the resulting 
drawings are also denoted by $C$ and $C'$.

Let $x_0$ and $x'_0$ denote the $x$-coordinates of $v_0\in C$ and $v'_0\in C'.$
It follows from Lemma \ref{slopenumlem21} that $x_0$
is a linear combination of $x_1, x_2, \ldots , x_m$, and $x_0'$
is a linear combination of $x'_1, x'_2, \ldots , x'_k$) with rational coefficients.
Therefore, if $x_0=x'_0$, then there is a nontrivial 
linear combination of $x_1, x_2, \ldots , x_m, x'_1, x'_2, \ldots , x'_k$ 
that gives $0$, contradicting the assumption that these numbers are 
independent over the rationals.  Thus, we can conclude that $x_0\ne x'_0$. 
Assume without loss of generality that $x_0<x'_0$.  
Reflect $C'$ about the $x$-axis, and shift it in the vertical 
direction so that $v'_0$ ends up to the Northeast from $v_0$. Clearly, we can 
add the missing edge $v_0v'_0$. Let $D$ denote the resulting drawing of $G$.  
We claim that $D$ meets all the requirements of the Theorem. Conditions (1), 
(2), and (3) are obviously satisfied, we only have to check that no vertex lies
in the interior of an edge. It follows from Lemma \ref{slopenumlem22} that the $y$-coordinates 
of $v_1, \ldots , v_m$ are all smaller than or equal to the $y$-coordinate of 
$v_0$ and the $y$-coordinates of $v'_1, \ldots , v'_k$ are all greater than or 
equal to the $y$-coordinate of $v'_0$.  We also have $y(v_0)<y(v'_0)$.
Therefore, there is no vertex in the interior of $v_0v'_0$.
Moreover, no edge of $C$ (resp. $C'$) can contain any vertex 
of $v'_0, v'_1, \ldots , v'_k$ (resp. $v_0, v_1, \ldots , v_m$)
in its interior. \hfill $\Box$

\medskip

The rest of the proof is by induction on the number of vertices of $G$. The
statement is trivial if the number of vertices is at most {\em
two}. Suppose that we have already established Theorem \ref{slopenum2} for all
graphs with fewer than $n$ vertices.

Suppose that $G$ has $n$ vertices, it is not a cycle and not
the union of two cycles
connected by one edge.
Let $v_1, v_2, \ldots , v_m$ be the vertices of $G$ with degree
less than {\em three}, and let the $x$-coordinates assigned to
them be $x_1, x_2, \ldots , x_m$.

We distinguish several cases.
\medskip

\noindent{\bf Case 1:} {\em $G$ has a vertex of degree {\it one}}.
\smallskip

Assume, without loss of generality, that $v_1$ is such a vertex.
If $G$ has no vertex of degree {\em three}, then it consists of a
simple path $P=v_1v_2\ldots v_m$, say. Place $v_m$ at the point
$(x_m,0)$. In general, assuming that $v_{i+1}$ has already been
embedded for some $i<m$, and $x_i<x_{i+1}$, place $v_{i}$ at the
point West of $v_{i+1}$, whose $x$-coordinate is $x_{i}$. If
$x_i>x_{i+1}$, then put $v_{i}$ at the point Northeast of
$v_{i+1}$, whose $x$-coordinate is $x_{i}$. The resulting drawing
of $G=P$ meets all the requirements of the theorem. To see this,
it is sufficient to notice that if $v_j$ would be Northwest of
$v_m$ for some $j<m$, then we could apply Lemma \ref{slopenumlem21} to the cycle
$v_jv_{j+1}\ldots v_m$, and conclude that the numbers $x_j,
x_{j+1},\ldots, x_m$ are dependent over the rationals. This
contradicts our assumption.

Assume next that $v_1$ is of degree {\em one}, and that $G$ has at
least one vertex of degree {\em three}. Suppose without loss of
generality that $v_1v_2\ldots v_kw$ is a path in $G$, whose
internal vertices are of degree {\em two}, but the degree of $w$
is {\em three}. Let $G'$ denote the graph obtained from $G$ by
removing the vertices $v_1, v_2, \ldots , v_{k}$. Obviously, $G'$
is a connected graph, in which the degree of $w$ is {\em two}.

If $G'$ is a {\em cycle}, then apply Lemma \ref{slopenumlem22} to $C=G'$ with $w$
playing the role of the vertex $v_0$ which has no preassigned
$x$-coordinate. We obtain an embedding of $G'$ with edges of
slopes $0, \pi/4,$ and $-\pi/4$ such that $x(v_i)=x_i$ for all
$i>k$ and there is no vertex to the North, to the Northeast, or to
the Northwest of $w$. By Lemma \ref{slopenumlem21}, the numbers $x(w),
x_{k+1},\ldots, x_m$ are not independent over the rationals.
Therefore, $x(w)\neq x_k$, so we can place $v_k$ at the point to
the Northwest or to the Northeast of $w$, whose $x$-coordinate is
$x_k$, depending on whether $x(w)>x_k$ or $x(w)<x_k$. After this,
embed $v_{k-1}, \ldots , v_1$, in this order, so that $v_i$ is
either to the Northeast or to the West of $v_{i+1}$ and
$x(v_i)=x_i$. According to property (4) in Lemma \ref{slopenumlem21}, the path
$v_1v_2\ldots v_k$ lies entirely above $G'$, so that no point of
$G$ can lie to the North or to the Northwest of $v_1$.

If $G'$ is {\em not a cycle}, then use the induction hypothesis to
find an embedding of $G'$ that satisfies all conditions of Theorem
\ref{slopenum2}, with $x(w)=x_k$ and $x(v_i)=x_i$ for every $i>k$. Now place
$v_k$ very far from $w$, to the North of it, and draw $v_{k-1},
\ldots , v_1$, in this order, in precisely the same way as in the
previous case. Now if $v_k$ is far enough, then none of the points
$v_k, v_{k-1},\ldots, v_1$ is to the Northwest or to the Northeast
of any vertex of $G'$. It remains to check that condition (4) is true
for $v_1$, but this follows from the fact that there is no point
of $G$ whose $y$-coordinate is larger than that of $v_1$.

\smallskip

{}  From now on, we can and will assume that $G$ has {\em no vertex of
degree one}.

A graph with {\em four} vertices and {\em five} edges between them
is said to be a {\em $\Theta$-graph}.

\medskip
\noindent{\bf Case 2:} {\em $G$ contains a $\Theta$-subgraph.}
\smallskip

Suppose that $G$ has a $\Theta$-subgraph with vertices $a,b,c,d,$
and edges $ab$, $bc$, $ac$, $ad$, $bd$. 
If neither $c$ nor $d$ has a third
neighbor, then $G$ is identical to this graph, which can easily be
drawn in the plane with all conditions of the theorem satisfied.

If $c$ and $d$ are connected by an edge, then all four points of
the $\Theta$-subgraph have degree {\em three}, so that $G$ has no
other vertices. So $G$ is a complete graph of four vertices, and it
has a drawing that meets the
requirements.

Suppose that $c$ and $d$ have a common neighbor $e\neq a,b$. If
$e$ has no further neighbor, then $a,b,c,d,e$ are the only
vertices of $G$, and again we can easily find a proper drawing.
Thus, we can assume that $e$ has a third neighbor $f$. By the
induction hypothesis, $G'=G\setminus \{a,b,c,d,e\}$ has a drawing
satisfying the conditions of Theorem \ref{slopenum2}. In particular, no vertex
of $G'$ is to the North of $f$ (and to the Northwest of $f$,
provided that the degree of $f$ in $G'$ is {\em one}). Further,
consider a drawing $H$ of the subgraph of $G$ induced by the
vertices $a,b,c,d,e$, which satisfies the requirements. We
distinguish two subcases.

If the degree of $f$ in $G'$ is {\em one}, then take a very small
{\em homothetic} copy of $H$ (i.e., similar copy in parallel
position), and rotate it about $e$ in the clockwise direction
through $3\pi/4$. There is no point of this drawing, denoted by
$H'$, to the Southeast of $e$, so that we can translate it into a
position in which $e$ is to the Northwest of $f\in V(G')$ and very
close to it. Connecting now $e$ to $f$, we obtain a drawing of $G$
satisfying the conditions. Note that it was important to make $H'$
very small and to place it very close to $f$, to make sure that
none of its vertices is to the North of any vertex of $G'$ whose
degree is at most {\em two}, or to the Northwest of any vertex of
degree {\em one} (other than $f$).

If the degree of $f$ in $G'$ is {\em two}, then we follow the same
procedure, except that now $H'$ is a small copy of $H$, rotated by
$\pi$. We translate $H'$ into a position in which $e$ is to the
North of $f$, and connect $e$ to $f$ by a vertical segment. It is
again clear that the resulting drawing of $G$ meets the
requirements in Theorem \ref{slopenum2}. Thus, we are done if $c$ and $d$ have a
common neighbor $e$.

Suppose now that only one of $c$ and $d$ has a third neighbor,
different from $a$ and $b$. Suppose, without loss of generality,
that this vertex is $c$, so that the degree of $d$ is {\em two}.
Then in $G'=G\setminus \{a,b,d\}$, the degree of $c$ is {\em one}.
Apply the induction hypothesis to $G'$ so that the $x$-coordinate
originally assigned to $d$ is now assigned to $c$ (which had no
preassigned $x$-coordinate in $G$). In the resulting drawing, we
can easily reinsert the remaining vertices, $a, b, d$, by adding a
very small square whose lowest vertex is at $c$ and whose
diagonals are parallel to the coordinate axes. The highest vertex
of this square will represent $d$, and the other two vertices will
represent $a$ and $b$.

We are left with the case when both $c$ and $d$ have a third
neighbor, other than $a$ and $b$, but these neighbors are
different. Denote them by $c'$ and $d'$, respectively. Create a
new graph $G'$ from $G$, by removing $a, b, c, d$ and adding a new
vertex $v$, which is connected to $c'$ and $d'$. Draw $G'$ using
the induction hypothesis, and reinsert $a,b,c,d$ in a small
neighborhood of $v$ so that they form the vertex set of a very
small square with diagonal $ab$. (See Figure \ref{slopenumfig1}.) As before, we
have to choose this square sufficiently small to make sure that
$a, b, c, d$ are not to the North of any vertex $w\neq c',d',v$ of
$G'$, whose degree is at most {\em two}, or to the Northwest of
any vertex of degree {\em one}. Thus, we are done if $G$ has a
$\Theta$-subgraph.

\smallskip

So, from now on we assume that $G$ has {\em no
$\Theta$-subgraph}.

\begin{figure}[htb]
\epsfxsize=10truecm
\begin{center}
\epsffile{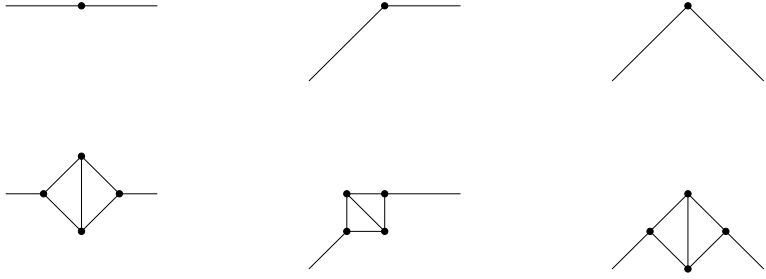}
\caption{Replacing $v$ by $\Theta$.}
\label{slopenumfig1}
\end{center}
\end{figure}

\medskip
\noindent{\bf Case 3:} {\em $G$ has no cycle that passes through a
vertex of degree {\em two}.}
\smallskip

Since $G$ is not three-regular, it contains at least one vertex of
degree {\em two}. Consider a decomposition of $G$ into
two-connected blocks and edges. If a block contains a vertex of
degree {\em two}, then it consists of a single edge. The block
decomposition has a treelike structure, so that there is a vertex
$w$ of degree {\em two}, such that $G$ can be obtained as the
union of two graphs, $G_1$ and $G_2$, having only the vertex $w$
in common, and there is no vertex of degree {\em two} in $G_1$.

By the induction hypothesis, for any assignment of rationally
independent $x$-coordinates to all vertices of degree less than
{\em three}, $G_1$ and $G_2$ have proper straight-line embeddings
(drawings) satisfying conditions (1)--(4) of the theorem. The only
vertex of $G_1$ with a preassigned $x$-coordinate is $w$. Applying
a vertical translation, if necessary, we can achieve that in both
drawings $w$ is mapped into the same point. Using the induction
hypothesis, we obtain that in the union of these two drawings,
there is no vertex in $G_1$ or $G_2$ to the North or to the
Northwest of $w$, because the degree of $w$ in $G_1$ and $G_2$ is
{\em one} (property (4)). This is stronger than what we need:
indeed, in $G$ the degree of $w$ is {\em two}, so that we require
only that there is no point of $G$ to the North of $w$ (property
(3)).

The superposition of the drawings of $G_1$ and $G_2$ satisfies all
conditions of the theorem. Only two problems may occur:
\begin{enumerate}
 \item A vertex of $G_1$ may end up at a point to the North
of a vertex of $G_2$ with degree {\em two}.
 \item The (unique) edges in $G_1$ and $G_2$, incident to $w$,
may partially overlap.
\end{enumerate}
Notice that both of these events can be avoided by enlarging the
drawing of $G_1$, if necessary, from the point $w$, and rotating
it about $w$ by $\pi/4$ in the clockwise direction. The latter
operation is needed only if problem 2 occurs. This completes the
induction step in the case when $G$ has no cycle passing through a
vertex of degree {\em two}.
\smallskip

It remains to analyze the last case.

\medskip

\noindent{\bf Case 4:} {\em $G$ has a cycle passing through a
vertex of degree {\em two}}.
\smallskip

By assumption, $G$ itself is not a cycle. Therefore, we can also
find a {\em shortest} cycle $C$ whose vertices are denoted by $v,
u_1, \ldots , u_{k}$, in this order, where the degree of $v$ is
{\em two} and the degree of $u_1$ is {\em three}. The length of
$C$ is $k+1$.

It follows from the minimality of $C$ that $u_i$ and $u_j$ are not
connected by an edge of $G$, for any $|i-j|>1$. Moreover, if
$|i-j|>2$, then $u_i$ and $u_j$ do not even have a common neighbor
$(1\le i\neq j\le k)$. This implies that any vertex $v\in
V(G\setminus C)$ has at most {\em three} neighbors on $C$, and
these neighbors must be consecutive on $C$. However, {\em three}
consecutive vertices of $C$, together with their common neighbor,
would form a $\Theta$-subgraph in $G$ (see Case 2). Hence, we can
assume that every vertex belonging to $G\setminus C$ is joined to
at most {\em two} vertices on $C$.

Let $B_i$ denote the set of all vertices of $G\setminus C$ that
have precisely $i$ neighbors on $C\; (i=0,1,2)$. Thus, we have
$V(G\setminus C)=B_0\cup B_1\cup B_2$. Further, $B_1=B_1^2\cup
B_1^3$, where an element of $B_1$ belongs to $B_1^2$ or $B_1^3$,
according to whether its degree in $G$ is {\em two} or {\em
three}.

Consider the list $v_1, v_2,\ldots, v_m$ of all vertices of $G$
with degree {\em two}. (Recall that we have already settled the
case when $G$ has a vertex of degree {\em one}.) Assume without
loss of generality that $v_1=v$ and that $v_i$ belongs to $C$ if
and only if $1\le i\le j$ for some $j\le m$.

\medskip
Let ${\bf x}$ denote the {\em assignment} of $x$-coordinates to
the vertices of $G$ with degree {\em two}, that is, ${\bf
x}=(x(v_1), x(v_2), \ldots,$$x(v_m))$$=(x_1, x_2, \ldots, x_m)$.
Given $G$, $C$, ${\bf x}$, and a real parameter $L$, we define the
following so-called {\sc Embedding Procedure$(G, C, {\bf x}, L)$}
to construct a drawing of $G$ that meets all requirements of the
theorem, and satisfies the additional condition that the
$y$-coordinate of every vertex of $C$ is at least $L$ higher than
the $y$-coordinates of all other vertices of $G$.

\smallskip

\noindent{\sc Step} 1: If $G':=G\setminus C$ is {\em not} a cycle, then 
construct recursively a drawing of
$G':=G\setminus C$ satisfying the conditions of Theorem \ref{slopenum2} with the
assignment ${\bf x}'$ of $x$-coordinates $x(v_i)=x_i$ for $j<i\le
m$, and $x(u_1')=x_1$, where $u_1'$ is the unique vertex in
$G\setminus C$, connected by an edge to $u_1\in V(C)$.
\smallskip

If $G'=G\setminus C$ is a cycle, then, by assumption, there are at least two
edges between $C$ and $G'$. One of them connects $u_1$ to $u_1'$.
Let $u_{\alpha}u'_{\alpha}$ be another such edge, where 
$u_{\alpha}\in C$ and $u'_{\alpha}\in G'$. Since the maximum degree is three, 
$u'_1\ne u'_{\alpha}$. 
Now construct recursively a drawing of
$G':=G\setminus C$ satisfying the conditions of Lemma \ref{slopenumlem22}, with the
assignment ${\bf x}'$ of $x$-coordinates $x(v_i)=x_i$ for $j<i\le
m$, $x(u_1')=x_1$, and with exceptional vertex $u'_{\alpha}$.

\smallskip

\noindent{\sc Step} 2: For each element of $B_1^2\cup B_2$, take
two rays starting at this vertex, pointing to the Northwest and to
the North. Further, take a vertical ray pointing to the North from
each element of $B_1^3$ and each element of the set $B_{\bf
x}:=\{(x_2, 0), (x_3, 0), \ldots , (x_j, 0)\}$. Let ${\cal R}$
denote the set of all of these rays. Choose the $x$-axis above all
points of $G'$ and all intersection points between the rays in
$\cal R$.

For any $u_h\; (1\le h\le k)$ whose degree in $G$ is {\em three},
define $N(u_h)$ as the unique neighbor of $u_h$ in $G\setminus C$.
If $u_h$ has degree {\em two} in $G$, then  $u_h=v_i$ for some
$1\le i\le j$, and let $N(u_h)$ be the point $(x_i, 0)$.
\smallskip

\noindent{\sc Step} 3: Recursively place $u_1, u_2, \ldots u_{k}$
on the rays belonging to ${\cal R}$, as follows. Place $u_1$ on
the vertical ray starting at $N(u_1)=u_1'$ such that $y(u_1)=L$.
Suppose that for some $i<k$ we have already placed $u_1, u_2,
\ldots u_{i}$, so that $L\le y(u_1)\le y(u_2)\le\ldots\le y(u_i)$
and there is no vertex to the West of $u_i$. Next we determine the
place of $u_{i+1}$.

If $N(u_{i+1})\in B_1^2$, then let $r\in{\cal R}$ be the ray
starting at $N(u_{i+1})$ and pointing to the Northwest. If
$N(u_{i+1})\in B_1^3\cup B_{\bf x}$, let $r\in{\cal R}$ be the ray
starting at $N(u_{i+1})$ and pointing to the North. In both cases,
place $u_{i+1}$ on $r$: if $u_i$ lies on the left-hand side of
$r$, then put $u_{i+1}$ to the Northeast of $u_i$; otherwise, put
$u_{i+1}$ to the West of $u_i$.

If $N(u_{i+1})\in B_2$, then let $r\in{\cal R}$ be the ray
starting at $N(u_{i+1})$ and pointing to the North, or, if we have
already placed a point on this ray, let $r$ be the other ray from
$N(u_{i+1})$, pointing to the Northwest, and proceed as before.
\smallskip

\begin{figure}[htb]
\epsfxsize=6.5truecm
\begin{center}
\epsffile{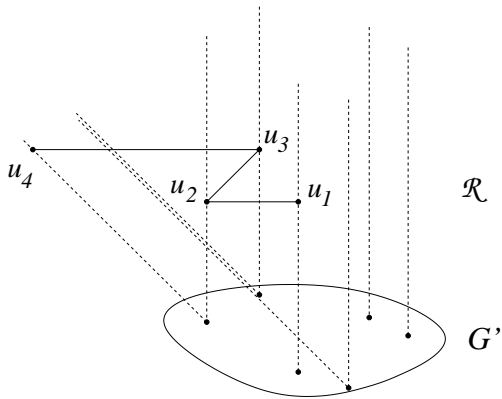}
\caption{Recursively place $u_1, u_2, \ldots u_{k}$
on the rays belonging to ${\cal R}$.}
\label{slopenumfig2}
\end{center}
\end{figure}

\noindent{\sc Step} 4: Suppose we have already placed $u_{k}$. It
remains to find the right position for $u_0:=v$, which has only
two neighbors, $u_1$ and $u_{k}$. Let $r$ be the ray at $u_{1}$,
pointing to the North. If $u_{k}$ lies on the left-hand side of
$r$, then put $u_0$ on $r$ to the Northeast of $u_k$; otherwise,
put $u_0$ on $r$, to the West of $u_{k}$.

During the whole procedure, we have never placed a vertex on any
edge, and all other conditions of Theorem \ref{slopenum2} are satisfied \hfill $\Box$.
\medskip

Remark that the $y$-coordinates of the vertices $u_0=v, u_1,
\ldots , u_{k}$ are at least $L$ higher than the $y$-coordinates
of all vertices in $G\setminus C$. If we fix $G, C,$ and ${\bf
x}$, and let $L$ tend to infinity, the coordinates of the vertices
given by the above {\sc Embedding Procedure$(G, C, {\bf x}, L)$}
change continuously.

\begin{figure}[htb]
\epsfxsize=7truecm
\begin{center}
\epsffile{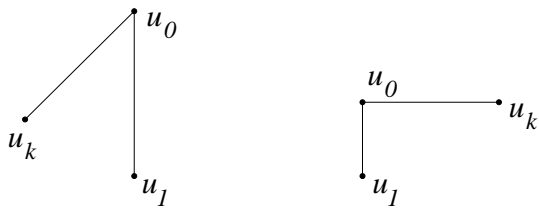}
\caption{Find the right position for $u_0$.}
\label{slopenumfig3}
\end{center}
\end{figure}

\subsection{Cubic Graphs - Proof of Theorem \ref{slopenum1}}

We are going to show that any graph $G$ with maximum degree three permits a straight-line
drawing using only the four basic directions (of slopes $0, \pi/2,
\pi/4,$ and $-\pi/4$), and perhaps one further direction, which is
almost vertical and is used for at most one edge in each connected
component of $G$.

Denote the connected components of $G$ by $G_1, G_2, \ldots, G_t$.
If a component $G_s$ is not three-regular, or if it is a complete
graph with {\em four} vertices, then, by Theorem \ref{slopenum2}, it can be
drawn using only the four basic directions. If $G_s$ has a
$\Theta$-subgraph, one can argue in the same way as in Case 2 of
the proof of Theorem \ref{slopenum2}: Embed recursively the rest of the graph,
and attach to it a small copy of this subgraph such that all edges
of the $\Theta$-subgraph, as well as the edges used for the
attachment, are parallel to one of the four basic directions.
Actually, in this case, $G_s$ itself can be drawn using the four
basic directions, so the fifth direction is not needed.

Thus, in the rest of the proof we can assume that $G_s$ is
three-regular, it has more than {\em four} vertices, and it
contains no $\Theta$-subgraph. For simplicity, we drop the
subscript and we write $G$ instead of $G_s$. Choose a shortest
cycle $C=u_0u_1 \ldots u_k$ in $G$. Each vertex of $C$ has
precisely one neighbor in $G\setminus C$. On the other hand, as in
the proof of the last case of Theorem \ref{slopenum2}, all vertices in
$G\setminus C$ have at most two neighbors in $C$.

We distinguish two cases.

\smallskip

\noindent {\bf Case 1.} $G\setminus C$ is a cycle.
Since $G$ is three-regular, $C$ and $G\setminus C$
are of the same size and the remaining edges of $G$ 
form a matching between the vertices of $C$
and the vertices of $G\setminus C$. 
For any $i$, $0\le i\le k$, let $u'_i$ 
denote the vertex of $G\setminus C$ which is connected to 
$u_i$.
Denote the vertices of $G\setminus C$ by $v_0, v_1, \ldots , v_k$,
in cyclic order, so that $v_1=u'_1$. Then we have
$v_i=u'_0$, for some $i>1$.
Apply Lemma \ref{slopenumlem22} to $G\setminus C$ with a rationally independent 
assignment ${\bf x}$ of $x$-coordinates to the vertices
$v_1, \ldots , v_k$, such that $x(v_1)=1$, $x(v_i)=\sqrt{2}$, and
the $x$-coordinates of the other vertices
are all greater than $\sqrt{2}$. (Recall that $v_0$ is an exceptional vertex 
with no assigned $x$-coordinate.)
It is not hard to see that if we follow the construction described
in the proof of Lemma \ref{slopenumlem22}, we also have $x(v_0)>\sqrt{2}$.
\smallskip

\noindent {\bf Case 2.} $G\setminus C$ is {\em not} a cycle.
Let $u_0'$ denote the neighbor of $u_0$ in $G\setminus C$. Since
$G$ has no $\Theta$-subgraph, $u_0'$ cannot be joined to both
$u_1$ and $u_k$. Assume without loss of generality that $u_0'$ is
not connected to $u_1$. Let $u_1'$ denote the neighbor of $u_1$ in
$G\setminus C$. 

Fix a rationally independent
assignment ${\bf x}$ of $x$-coordinates to the vertices of degree
at most {\em two} in $G\setminus C$, such that $x(u_0')=\sqrt{2}$,
$x(u_1')=1$, and the $x$-coordinates of the other vertices
are all greater than $\sqrt{2}$. 
Consider a drawing of $G\setminus C$,
meeting the requirements of Theorem \ref{slopenum2}. 

\medskip

Now in both cases,  let $G'$ denote the graph obtained from $G$ after
the removal of the edge $u_0u_0'$. Clearly $G\setminus C=G'\setminus C$,
and for any $L$, {\sc Embedding Procedure}$(G',C,{\bf x},L)$ gives a drawing
of
$G'$.
It follows from the construction,
that $x(u_0)=x(u_1)=x(u_1')=1$, $x(u'_0)=\sqrt{2}$.
Therefore, for any sufficiently small $\varepsilon>0$
there is an $L>0$ such that 
{\sc Embedding Procedure}$(G',C,{\bf x},L)$ gives a drawing of
$G'$, in which the slope of the line connecting $u_0$ and $u_0'$
is $\frac{\pi}{2}+\varepsilon$. 

We want to add the segment $u_0u_0'$ to this drawing. 
Since there is no
vertex with $x$-coordinate between $1$ and $\sqrt{2}$,
the 
segment $u_0u_0'$
cannot pass through any vertex of $G$.

Summarizing: if $\varepsilon$ is sufficiently small (that is, if
$L$ is sufficiently large), then each component of the graph has a
proper drawing in which all edges are of one of the four basic
directions, with the exception of at most one edge whose slope is
$\frac{\pi}{2}+\varepsilon$. If we choose an $\varepsilon>0$ that
works for all components, then the whole graph can be drawn using
only at most {\em five} directions. This concludes the proof of
Theorem \ref{slopenum1}. \hfill $\Box$

%---------------------------------------------------------------------------------

\subsection{Algorithm}

Based on the proof, it is not hard to design an algorithm to find a proper
drawing, in quadratic time.

First, if our graph is a circle, we have no problem drawing it in $O(n)$ steps.
If our graph has a vertex of degree \emph{one}
then the procedure of Case 1 of the proof of Theorem \ref{slopenum2} requires at most $O(m)$
time when we reinsert $v_1, \ldots, v_m$.

We can check if our graph has any $\Theta$-subgraph in $O(n)$
time. If we find one, we can proceed by induction as in Case 2 of
the proof of Theorem \ref{slopenum2}. We can reinsert the $\Theta$-subgraph as
described in Case 2 in $O(1)$ time.

Now assume that we have a vertex $v$ of degree \emph{two}. Execute
a breadth first search from any vertex, and take a {\em minimal}
vertex of degree two, that is, a vertex $v$ of degree two, all of
whose descendants are of degree three. If there is an edge in the
graph connecting a descendant of $v$ with a non-descendant, then
there is a cycle through $v$; we can find a minimal one with a
breadth first search from it and proceed as in Case 4. Otherwise,
$v$ can play the role of $w$ in Case 3, and we can proceed
recursively.

Finally, if the graph is 3-regular, then we draw each component
separately, except the last step, when we have to pick an
$\epsilon$ small enough simultaneously for all components, this
takes $O(n)$ steps. We only have to find the greatest slope and
pick an $\varepsilon$ such that $\frac{\pi}{2}+\varepsilon$ is
even steeper.

\medskip

We believe that this algorithm is far from being optimal. It may
perform a breadth first search for each induction step, which is
probably not necessary. One may be able to replace this step by
repeatedly updating the results of the first search. We cannot
even rule out that the problem can be solved in linear time.

%% file: slopepar.tex
This section is based on our paper with Bal\'azs Keszegh, J\'anos Pach and G\'eza T\'oth, Cubic graphs have bounded slope parameter \cite{KPPT10}.

Let us recall the definition of the {\em slope parameter}.
Given a set $P$ of points in
the plane and a set $\Sigma$
of slopes, define $G(P,\Sigma)$ as
the graph on the vertex set $P$, in which two vertices $p,q\in P$
are connected by an edge if and only if the slope of the line $pq$
belongs to $\Sigma$. The {\em slope parameter} $s(G)$ of $G$ is
the size of the smallest set $\Sigma$
of slopes such that $G$ is
isomorphic to $G(P,\Sigma)$ for a suitable set of points $P$ in
the plane. 

Any graph $G$ of maximum degree {\em two} splits into
vertex-disjoint cycles, paths, and possibly isolated vertices. Hence, for
such graphs we have $s(G)\le 3$. In contrast, as was shown by
Bar\'at {\it et al.} \cite{BMW06}, for any $d\ge 5$, there exist
graphs of maximum degree $d$, whose slope parameters are arbitrarily
large.

Remember that a graph is said to be {\em cubic} if the degree of each of its
vertices is at most {\em three}. A cubic graph is {\em subcubic}
if each of its connected components has a vertex of degree smaller
than {\em three}.

The main result of this section is

\begin{thm}\label{slopepar1} Every cubic graph has slope
parameter at most {\em seven}.
\end{thm}

This theorem is not likely to be tight. The best lower bound we are aware of is {\em four}. This bound is attained, for example, for the 8-vertex subcubic graph that can be obtained from the graph formed by the edges of a 3-dimensional cube by deleting one of its edges.

We will refer to the angles $i\pi/5, \; 0\le i\le 4,$ as the {\em
five basic slopes}. We start by proving the following
statement, which constitutes the first step of the proof of
Theorem \ref{slopepar1}.

\begin{thm}\label{slopepar2} Every subcubic graph has slope
parameter at most {\em five}. Moreover, this can be realized by a
straight-line drawing such that no {\em three} vertices are on a
line and each edge has one of the {\em five} basic slopes.
\end{thm}

Using the fact that in the drawing guaranteed by Theorem \ref{slopepar2} no {\em
three} vertices are collinear, we can also conclude that the slope
{\em number} of every subcubic graph is at most {\em five}. In
the Section \ref{sec:slopenum}, however, it was shown that this number
is at most {\em four} and for cubic graphs it is at most {\em five}. 

\subsection{Subcubic Graphs - Proof of Theorem \ref{slopepar2}}

The proof is by induction on the number of vertices of the graph.
Clearly, the statement holds for graphs with fewer than {\em
three} vertices. Let $n$ be fixed and suppose that we have already
established the statement for graphs with fewer than $n$ vertices.
Let $G$ be a subcubic graph of $n$ vertices. We can assume that
$G$ is connected, otherwise we can draw each of its connected
components separately and translate the resulting drawings through
suitable vectors so that no two points in distinct components
determine a line of basic slope.

To obtain a straight-line
drawing of $G$, we have to find proper locations for
its vertices. At each inductive step,
we start with a drawing of a subgraph
of $G$ satisfying the conditions of Theorem \ref{slopepar2}
and extend it by adding a vertex.
At a given stage of the procedure, for any vertex $v$ that has
already been added, consider the (basic) slopes of all edges
adjacent to $v$ that have already been drawn, and let $\s(v)$
denote the set of integers $0\le i<5$
%\mod 5$
for which $i\pi/5$ is
such a slope. That is, at the beginning $\s(v)$ is undefined, then
it gets defined, and later it may change (expand). Analogously,
for any edge $uv$ of $G$, denote by $\s(uv)$ the integer $0\le i<5$
%\mod 5$
for which the slope of $uv$ is $i\pi/5$.

\smallskip

\noindent{\bf Case 1:} {\em $G$ has a vertex of degree {\em one}}.

Assume without loss of generality, that $v$ is a vertex of degree
{\em one}, and let $w$ denote its only neighbor. Deleting $v$ from
$G$, the degree of $w$ in the resulting graph $G'$ is at most {\em
two}. Therefore, by the induction hypothesis, $G'$ has a drawing
meeting the requirements. As $w$ has degree at most {\em two},
there is a basic slope $\sigma$ such that no other vertex of $G'$
lies on the line $\ell$ of slope $\sigma$ that passes through $w$.
Draw all {\em five} lines of basic slopes through each vertex of
$G'$. These lines intersect $\ell$ in finitely many points. We can
place $v$ at any other point of $\ell$, to obtain a proper drawing
of $G$.

\smallskip

From now on, assume that $G$ has no vertex of degree one.

\smallskip

\noindent{\bf Case 2:} {\em $G$ has no cycle that passes through a
vertex of degree {\em two}.}

Since $G$ is subcubic, it contains a
vertex $w$ of degree {\em two} such that $G$ is the union of two
graphs, $G_1$ and $G_2$, having only vertex $w$ in common. Both
$G_1$ and $G_2$ are subcubic and have fewer than $n$ vertices, so
by the induction hypothesis both of them have a drawing satisfying
the conditions. Translate the drawing of $G_2$ so that the points
representing $w$ in the two drawings coincide. Since $w$ has
degree {\em one} in both $G_1$ and $G_2$, by a possible rotation
of $G_2$ about $w$ through an angle that is a multiple of $\pi/5$,
we can achieve that the
two edges adjacent to $w$ are not parallel. By scaling $G_2$ from
$w$, if necessary, we can also achieve that the slope of no
segment between a vertex of $G_1\setminus w$ and a vertex of
$G_2\setminus w$ is a basic slope. Thus, the resulting drawing of
$G$ meets the requirements.

\smallskip

\noindent{\bf Case 3:} {\em $G$ has a cycle passing through a
vertex of degree {\em two}}.

If $G$ itself is a cycle, we can
easily draw it. If it is not the case, let $C$ be a {\em shortest}
cycle which contains a vertex of degree two. Let $u_0, u_1, \ldots
, u_{k}$ denote the vertices of $C$, in this cyclic order, such that
$u_0$ has degree {\em two} and $u_1$ has degree {\em three}. The
indices are understood mod $k+1$, that is, for instance,
$u_{k+1}=u_0$. It follows from the minimality of $C$ that $u_i$
and $u_j$ are not connected by an edge of $G$ whenever $|i-j|>1$.

Since $G\setminus C$ is subcubic, by assumption, it admits a
straight-line drawing satisfying the conditions. Each $u_i$ has at
most {\em one} neighbor in $G\setminus C$. Denote this
neighbor by $t_i$, if it exists. For every $i$ for which $t_i$
exists, we place $u_i$ on a line passing through $t_i$. We place
the $u_i$'s one by one, ``very far" from $G\setminus C$, starting
with $u_1$. Finally, we arrive at $u_0$, which has no neighbor in
$G\setminus C$, so that it can be placed at the intersection of
two lines of basic slope, through $u_1$ and $u_{k}$, respectively.
We have to argue that our method does not create ``unnecessary''
edges, that is, we never place two independent vertices in such a
way that the slope of the segment connecting them is a basic
slope. In what follows, we make this argument precise.

%\vskip 0.2cm

\begin{figure}[ht]
\begin{center}
\scalebox{\figsize}{\includegraphics{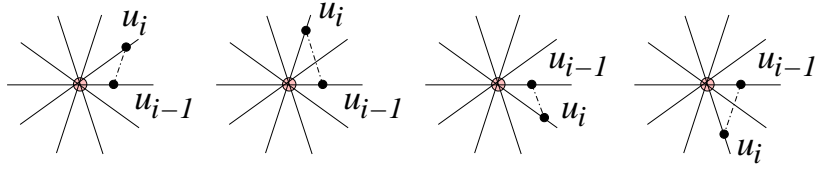}}
\caption{The four possible locations of $u_{i}$.}
\label{slopeparfig1}
\end{center}
\end{figure}

The locations of the vertices $u_0, u_1, \ldots, u_k$ are determined by using the following algorithm, {\sc Procedure}$(G, C, u_0, u_1, x)$, where $G$ is the input
subcubic graph, $C$ is a shortest cycle passing through a vertex
of degree {\em two}, $u_0$, that has a degree {\em three} neighbor, $u_1$, 
and $x$ is a real
parameter. Note that {\sc Procedure}$(G, C, u_0, u_1, x)$ is a
{\em nondeterministic} algorithm, as we have more than one choice
at certain steps. (However, it is very easy to make it deterministic.)\\

\bigskip

\noindent {\sc Procedure}$(G, C, u_0, u_1, x)$

\smallskip

\begin{itemize}
\item{\sc Step $0$.} Since $G\setminus C$ is subcubic, it has
    a representation with the {\em five} basic slopes. Take
    such a representation, scaled and translated in such a way
    that $t_1$ (which exists since the degree of $u_1$ is three)
    is at the origin, and all other vertices are
    within unit distance from it.

For any $i$, $2\le i\le k$, for which $u_{i}$ does not have a
neighbor in $G\setminus C$, let $t_{i}$ be any unoccupied
point closer to the origin than $1$, such that the slope of
none of the lines connecting $t_i$ to $t_1, t_2, \ldots
t_{i-1}$ or to any other already embedded point of $G\setminus
C$ is a basic slope.
\end{itemize}

For any point $p$ and for any $0\le i\le 4$, let $\ell_i(p)$ denote
the line with $i$th basic slope, $i\pi/5$, passing through $p$.
Let $\ell_i$ stand for $\ell_i(O)$, where $O$ denotes the origin.

We will place $u_1,\ldots,u_{k}$ recursively, so that $u_j$ is
placed on $\ell_i(t_j)$, for a suitable $i$. Once 
$u_j$ has been placed on some $\ell_i(t_j)$, define
$ind(u_j)$, the {\em index} of $u_j$, to be $i$. (The
indices are taken mod $5$. Thus, for example, $|i-i'|\ge 2$ is
equivalent to saying that $i\neq i'$ and $i\neq i'\pm 1\mod 5$.)
Start with $u_1$. The degree of $t_1$ in $G\setminus C$ is at most
{\em two}, so that at the beginning the set $\s(t_1)$ (defined in
the first paragraph of this section) has at most {\em two}
elements. Let $l\notin\s(t_1)$. Direct the line $\ell_l(t_1)$
arbitrarily, and place $u_1$ on it at distance $x$ from $t_1$ in
the positive direction. (According to this rule, if $x<0$, then
$u_1$ is placed on $\ell_l(t_1)$ at distance $|x|$ from $t_1$ in
the {\em negative} direction.)

Suppose that $u_1$, $u_2, \ldots$, $u_{i-1}$ have been already
placed and that $u_{i-1}$ lies on the line $\ell_l(t_{i-1})$, that
is, we have $ind(u_{i-1})=l$.

\smallskip
\begin{itemize}
\item{\sc Step $i$.} We place $u_i$ at one of the following
    four locations (see Figure \ref{slopeparfig1}):

(1) the intersection of $\ell_{l+1}(t_i)$ and
$\ell_{l+2}(u_{i-1})$;\\
(2) the intersection of $\ell_{l+2}(t_i)$ and
$\ell_{l+3}(u_{i-1})$;\\
(3) the intersection of $\ell_{l-1}(t_i)$ and
$\ell_{l-2}(u_{i-1})$;\\
(4) the intersection of $\ell_{l-2}(t_i)$ and
$\ell_{l-3}(u_{i-1})$.

Choose from the above four possibilities so that the edge
$u_it_i$ is not parallel to any other edge already drawn and
adjacent to $t_i$, i.e., before adding the edge $u_it_i$ to
the drawing, $\s(t_i)$ did not include $\s(u_it_i)$.
\end{itemize}

It follows directly from (1)--(4) that $\s(u_{i-1})\subset \{ l, l-1, l+1
\mod 5\}$, while $\s(u_iu_{i-1})\subset \{ l-2, l+2
\mod 5\}$, that is,
before adding the edge $u_iu_{i-1}$ to the
drawing, we had $\s(u_iu_{i-1})\notin \s(u_{i-1})$.
Avoiding for $u_it_i$ the slopes of the edges already incident to $t_i$,
leaves available two of the choices (1), (2), (3), (4).

Let $u_{i-1}'$ be the translation of $u_{i-1}$ by the vector $\overrightarrow{t_{i-1}O}$,
and similarly, let $u_{i}'$ be the translation of $u_{i}$ by the vector $\vec{t_{i}O}$.
That is, $Ou'_{i-1}u_{i-1}t_{i-1}$ and
$Ou'_{i}u_{i}t_{i}$ are
parallelograms.
We have
$$\overline{Ou_{i-1}}-1<\overline{Ou_{i-1}'}<\overline{Ou_{i-1}}+1,$$
$$\overline{Ou_{i}}-1<\overline{Ou_{i}'}<\overline{Ou_{i}}+1,$$
and
$$2\cos\left({\pi\over 5}\right)\overline{Ou_{i-1}'}=\overline{Ou_{i}'}.$$
Therefore,
for any possible
location of $u_i$, we have
$$
1.6\overline{Ou_{i-1}}-4<
2\cos\left({\pi\over 5}\right)\overline{Ou_{i-1}}-4<
\overline{Ou_i}<
2\cos\left({\pi\over 5}\right)\overline{Ou_{i-1}}+4<
1.7\overline{Ou_{i-1}}+4.
$$
Suppose that $|x|\ge 50$.
Clearly, $|x|-1<\overline{Ou_1}$, and by the previous calculations
it is easy to show by induction that
$|x|-1<\overline{Ou_{i}}$ for all $i\le k$.
Therefore,
$1.5\overline{Ou_{i-1}}<1.6\overline{Ou_{i-1}}-4$ so we obtain
\begin{equation}\label{egyenlotlenseg}
1.5\overline{Ou_{i-1}}<\overline{Ou_i}.
\end{equation}
We have to verify that the above procedure does not produce
``unnecessary'' edges, that is, the following statement is true.

\begin{claim}\label{slopeparc1} Suppose that $|x|\ge 50$.

(i) {\em The slope of $u_iu_j$ is not a basic slope, for any
$j<i-1$.}

(ii) {\em The slope of $u_iv$ is not a basic slope, for any $v\in
V(G\setminus C)$, $v\ne t_i$.}

\end{claim}

\noindent {\bf Proof.} (i) Suppose that the slope of $u_iu_j$ is a
basic slope for some $j<i-1$. By repeated application of
inequality~(\ref{egyenlotlenseg}), we obtain that
$\overline{Ou_i}>1.5^{i-j}\overline{Ou_j}>2\overline{Ou_j}$. On
the other hand, if $u_iu_j$ has a basic slope, then easy geometric
calculations show that $\overline{Ou_i}< 2\cos\left({\pi\over
5}\right)\overline{Ou_{j}}+4<2\overline{Ou_j}$, a contradiction.

(ii) Suppose for simplicity that $t_iu_i$ has slope $0$, i.e., it
is horizontal. By the construction, no vertex $v$ of $G\setminus
C$ determines a horizontal segment with $t_i$, but all of them are
within distance $2$ from $t_i$. As $\overline{Ou_i}> x-1$, segment
$vu_i$ is almost, but not exactly horizontal. That is, we have
$0<|\angle t_iu_iv|< \pi/5$, contradiction. \hfill $\Box$

\smallskip

Suppose that {\sc Step $0$, Step $1$, \ldots, Step $k$} have
already been completed. It remains to determine the position of
$u_0$. We need some preparation.
The notation $|x|\ge 2 \mod 5$ means that $x=2$ or $x=3 \mod 5$.

\begin{claim}\label{slopeparc2} There exist two integers 
$0\le \alpha, \beta <5$ with $|\alpha -\beta |\ge 2$ $\mod 5$ such that starting the {\sc
Procedure} with $ind(u_1)=\alpha$ and with $ind(u_1)=\beta$, we can
continue so that $ind(u_2)$ is the same in both cases.
\end{claim}

\noindent {\bf Proof.} Suppose that the degrees of $t_1$ and $t_2$
in $G\setminus C$ are {\em two}, that is, there are two forbidden
lines for both $u_1$ and $u_2$. In the other cases, when the
degree of $t_1$ or the degree of $t_2$ is less than {\em two}, or
when $t_1=t_2$, the proof is similar, but simpler. We can place
$u_1$ on $\ell_l(t_1)$ for any $l\notin\s(t_1)$. Therefore, we
have three choices, two of which, $\ell_{\alpha}(t_1)$ and
$\ell_{\beta}(t_1)$, are not consecutive, so that
$|\alpha-\beta|\ge 2 \mod 5$.

The vertex $u_2$ cannot be placed on $\ell_m(t_2)$ for any
$m\in\s(t_2)$, so there are three possible lines for $u_2$:
$\ell_{x}(t_2)$, $\ell_{y}(t_2)$, $\ell_{z}(t_2)$, say. For any
fixed location of $u_1$, we can place $u_2$ on four possible lines,
so on at least two of the
lines $\ell_{x}(t_2)$, $\ell_{y}(t_2)$, and $\ell_{z}(t_2)$.
Therefore, at least one of them, say $\ell_{x}(t_2)$, can be used
for both locations of $u_1$. \hfill $\Box$

\begin{claim}\label{slopeparc3} We can place the vertices $u_1, u_2,
\ldots , u_k$ using the {\sc Procedure} so that for all $k$ we have
$|ind(u_1)-ind(u_k)|\ge 2\  \mod 5.$
\end{claim}

\noindent {\bf Proof.} By Claim \ref{slopeparc2}, there are two placements of the
vertices of $C\setminus \{u_0, u_k\}$, denoted by $u_1, u_2,
\ldots , u_{k-1}$ and by $u'_1, u'_2, \ldots , u'_{k-1}$ such that
$|ind(u_1)-ind(u'_1)|\ge 2 \mod 5$, and $ind(u_i)=ind(u'_i)$ for
all $i\ge 2$. That is, we can start placing the vertices on two
non-neighboring lines so that from the second step of the {\sc
Procedure} we use the same lines. We show that we can place $u_k$
so that $u_1$ and $u_k$, or $u'_1$ and $u_k$ are on
non-neighboring lines. Having placed $u_{k-1}$ (or $u'_{k-1}$), we
have four choices for $ind(u_k)$. Two of them can be ruled out by
the condition $ind(u_k)\notin\s(t_k)$. We still have two choices.
Since $u_1$ and $u'_1$ are on non-neighboring lines, there is only
one line which is a neighbor of both of them. Therefore, we still
have at least one choice for $ind(u_k)$ such that
$|ind(u_1)-ind(u_k)|\ge 2$ or $|ind(u'_1)-ind(u_k)|\ge 2$. \hfill $\Box$

\begin{figure}[ht]
\begin{center}
\scalebox{\figsize}{\includegraphics{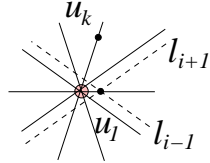}}
\caption{$\ell_{i+1}(u_1)$, does not separate the
    vertices of $G\setminus C$ from $u_k$, $\ell_{i-1}(u_1)$ does.}
\label{slopeparfig2}   
\end{center}
\end{figure}

\begin{itemize}
\item{\sc Step $k+1$.} Let $i=ind(u_1)$, $j=ind(u_k)$, and
    assume, by Claim \ref{slopeparc3}, that $|i-j|\ge 2$ $\mod 5$. Consider
    the lines $\ell_{i-1}(u_1)$ and $\ell_{i+1}(u_1)$. One of
    them, $\ell_{i+1}(u_1)$, say, does not separate the
    vertices of $G\setminus C$ from $u_k$, the other one does. See Fig. \ref{slopeparfig2}.
    
Place $u_0$ at the intersection of $\ell_{i+1}(u_1)$ and $\ell_{i}(u_k)$.

%\smallskip

\end{itemize}

%\vskip 0.2cm

\begin{figure}[ht]
\begin{center}
\scalebox{\figsize}{\includegraphics{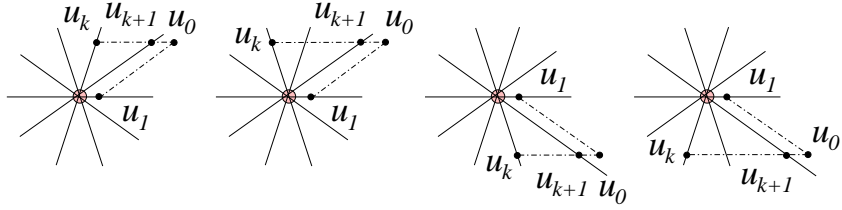}}
\caption{The four possible locations of $u_{0}$.}
\label{slopeparfig3}
\end{center}
\end{figure}

%\vskip 0.2cm

\begin{claim}\label{slopeparc4} Suppose that $|x|\ge 50$.

(i) {\em The slope of $u_0u_j$ is not a basic slope, for any
$1<j<k$.}

(ii) {\em The slope of $u_0v$ is not a basic slope, for any $v\in
V(G\setminus C)$.}

\end{claim}

\noindent {\bf Proof.} (i) Denote by  $u_{k+1}$ the intersection
of $\ell_{i+1}(O)$ and $\ell_{i}(u_k)$. Suppose that the slope of
$u_0u_j$ is a basic slope for some $1<j<k$. As in the proof of Claim
\ref{slopeparc1}, by repeated application of inequality~(\ref{egyenlotlenseg}), we
obtain that $\overline{Ou_{k+1}}>
1.5^{k+1-j}\overline{Ou_j}>2\overline{Ou_j}$. On the other hand,
by an easy geometric argument, if the slope of $u_0u_j$ is a basic
slope, then $\overline{Ou_{k+1}}< 2\cos\left({\pi\over
5}\right)\overline{Ou_{j}}+4<2\overline{Ou_j}$, a contradiction,
provided that $|x|\ge 50$.

(ii) For any vertex $v\in G\setminus C$, the slope of the segment
$u_0v$ is strictly between $i\pi/5$ and ${(i+1)\pi/5}$, therefore,
it is not a basic slope. See Figure \ref{slopeparfig3}. This concludes the proof
of the claim and hence Theorem \ref{slopepar2}. \hfill $\Box$

\subsection{Cubic Graphs - Proof of Theorem \ref{slopepar1}}

First we note that if $G$ is connected, then Theorem \ref{slopepar1} is an easy
corollary to Theorem \ref{slopepar2}. Indeed, delete any vertex, and then put it
back using two extra directions. If $G$ is not connected, the only
problem that may arise is that these extra directions can differ
for different components. We will define a family of drawings for
each component $G^i$ of $G$, depending on parameters $\varepsilon_i$, and
then choose the values of these parameters in such a way that the
extra directions will coincide.

Suppose that $G$ is a cubic graph. If a connected component is not
3-regular then, by Theorem \ref{slopepar2}, it can be drawn using the {\em five}
basic slopes. If a connected component is a complete graph $K_4$
on {\em four} vertices, then it can also be drawn using the basic
slopes. For the sake of simplicity, suppose that we do not have such
components, i. e. each connected component $G^1, \ldots, G^m$
of $G$ is 3-regular and none of them is isomorphic to $K_4$.

First we concentrate on $G^1$. Let $C$ be a shortest cycle in
$G^1$. We distinguish two cases.

\smallskip

\noindent{\bf Case 1:} {\em $C$ is not a triangle.} Denote by
$u_0, \ldots , u_k$ the vertices of $C$, and let $t_0$ be the
neighbor of $u_0$ not belonging to $C$. Delete the edge $u_0t_0$,
and let $\bar{G}$ be the resulting graph.\\

\noindent{\bf Case 2:} {\em $C$ is a triangle.} Every vertex
of $C$ has precisely {\em one} neighbor that does not belong to
$C$. If all these neighbors coincide, then $G^1$ is a complete
graph on {\em four} vertices, contradicting our assumption. So one
vertex of $C$, $u_0$, say, has a neighbor $t_0$ which does not
belong to $C$ and which is not adjacent to the other two vertices,
$u_1$ and $u_2$, of $C$. Delete the edge $u_0t_0$, and let
$\bar{G}$ be the resulting graph.

\smallskip

Observe that in both cases, $u_k$ and $t_0$ are not connected in
$G^1$. Indeed, suppose for a contradiction that they are
connected. In the first case, $G^1$ would contain the triangle
$u_0u_kt_0$, contradicting the minimality of $C$. In the second
case, the choice of $u_0$ would be violated.

There will be exactly two edges with extra directions, $u_0u_k$ and $u_0t_0$.
The slope of $u_0u_k$ will be very close to a basic slope and
the slope of $u_0t_0$ will be decided at the end, but we will show that almost any choice will do.

For any nonnegative $\varepsilon$ and real $x$, {\sc
Modified\-Procedure}$(\bar{G}, C, u_0, u_1, x, \varepsilon)$ is defined as
follows. {\sc Steps} $0, 1, \ldots, k$ are identical to {\sc Steps} $0, 1, \ldots, k$ of {\sc Procedure}$(\bar{G}, C, u_0,
u_1, x)$.

\smallskip

\begin{itemize}
\item{\sc Step $k+1$.} If there is a
    segment, determined by the vertices of $G\setminus C$, of
    slope $i\pi/5+\varepsilon$ or $i\pi/5-\varepsilon$, for
    any $0\le i< 5$, then {\sc Stop}. In this case, we
    say that $\varepsilon$ is {\em 1-bad} for $\bar{G}$.

Otherwise, when $\varepsilon$ is {\em 1-good}, let
$i=ind(u_1)$ and $j=ind(u_k)$. We can assume by Claim \ref{slopeparc3}
that $|i-j|\ge 2
\mod 5$. Consider the lines $\ell_{i-1}(u_1)$ and
$\ell_{i+1}(u_1)$. One of them does not separate the vertices
of $G\setminus C$ from $u_k$, the other one does.

If $\ell_{i-1}(u_1)$ separates $G\setminus C$ from $u_k$, then
place $u_0$ at the intersection of $\ell_{i+1}(u_1)$ and the
line through $u_k$ with slope ${i\pi/5+\varepsilon}$. If
$\ell_{i+1}(u_1)$ separates $G\setminus C$ from $u_k$,
then place $u_0$ at the intersection of $\ell_{i-1}(u_1)$ and
the line through $u_k$ with slope ${i\pi/5-\varepsilon}$.

%\smallskip

\end{itemize}

Since {\sc Steps} $0, 1, \ldots, k$ are identical in {\sc Procedure}\-$(\bar{G}, C, u_0, u_1, x)$ and in {\sc Modified\-Procedure}$(\bar{G}, C, u_0, u_1, x, \varepsilon)$, Claims \ref{slopeparc1}, \ref{slopeparc2}, and
\ref{slopeparc3} are also true for the {\sc Modified\-Procedure}.

Moreover, it is easy to see that an analogue of Claim \ref{slopeparc4} also
holds with an identical proof, provided that $\varepsilon$ is
sufficiently small: $0< \varepsilon< 1/100$.

\smallskip

\noindent {\bf Claim \ref{slopeparc4}'.} {\em Suppose that $|x|\ge 50$ and $0< \varepsilon< 1/100$.}

(i) {\em The slope of $u_0u_j$ is not a basic slope, for any
$1<j<k$.}

(ii) {\em The slope of $u_0v$ is not a basic slope, for any $v\in
V(\bar{G}\setminus C)$.}  \hfill $\Box$

\smallskip

Perform {\sc Modified\-Procedure}$(\bar{G}, C, u_0, u_1, x,
\varepsilon)$ for a fixed $\varepsilon$, and observe how the
drawing changes as $x$ varies. For any vertex $u_i$ of $C$, let
$u_i(x)$ denote the position of $u_i$, as a function of $x$. For
every $i$, the function $u_i(x)$ is linear, that is, $u_i$ moves
along a line as $x$ varies.

\begin{claim}\label{slopeparc5} With finitely many exceptions, for
every value of $x$, {\sc Modified\-Procedure}\-$(\bar{G}, C, u_0,
u_1, x, \varepsilon)$ produces a proper drawing of $\bar{G}$,
provided that $\varepsilon$ is 1-good.
\end{claim}

\noindent {\bf Proof.} Claims \ref{slopeparc1}, \ref{slopeparc2}, \ref{slopeparc3}, and \ref{slopeparc4}' imply Claim \ref{slopeparc5} for
$|x|\ge 50$. Let $u$ and $v$ be two vertices of $\bar{G}$. Since
$u(x)$ and $v(x)$ are linear functions, their difference,
$\vec{uv}(x)$, is also linear.

If $uv$ is an edge of $\bar{G}$, then the direction of
$\vec{uv}(x)$ is the same for all $|x|\ge 50$. Therefore, it is
the same for all values of $x$, with the possible exception of one
value, for which $\vec{uv}(x)=0$ holds.

If $uv$ is not an edge of $\bar{G}$, then the slope of
$\vec{uv}(x)$ is not a basic slope for any $|x|\ge 50$. Therefore,
with the exception of at most {\em five} values of $x$, the slope
of $\vec{uv}(x)$ is never a basic slope, nor does $\vec{uv}(x)=0$
hold. \hfill $\Box$

\smallskip

Take a closer look at the relative position of the endpoints of
the missing edge, $u_0(x)$ and $t_0(x)$. Since $t_0\in
\bar{G}\setminus C$, $t_0=t_0(x)$ is the same for all values of
$x$. The position of $u_0=u_0(x)$ is a linear function of $x$. Let
$\ell$ be the line determined by the function $u_0(x)$. If $\ell$
passes through $t_0$, then we say that $\varepsilon$ is {\em
2-bad} for $\bar{G}$. If $\varepsilon$ is 1-good and it is not
2-bad for $\bar{G}$, then we say that it is {\em 2-good} for
$\bar{G}$. If $\varepsilon$ is 2-good, then by varying $x$ we can
achieve almost any slope for the edge $t_0u_0$. This will turn out
to be crucially important, because we want to attain that these
slopes coincide in all components.

\begin{claim}\label{slopeparc6} Suppose that the values $\varepsilon\neq\delta$,
$0<\varepsilon, \delta<1/100$, are 1-good for $\bar{G}$. Then at
least one of them is 2-good for $\bar{G}$.
\end{claim}

\noindent {\bf Proof.} Suppose, for simplicity, that $ind(u_1)=0$,
$ind(u_k)=2$, and that $u_1$ and $u_k$ are in the right
half-plane (of the vertical line through the origin). The other cases can be settled analogously. To
distinguish between {\sc Modified\-Procedure}$(\bar{G}, C, u_0,
u_1, x, \varepsilon)$ and {\sc Modified\-Procedure}$(\bar{G}, C,
u_0, u_1, x, \delta)$, let $u^{\varepsilon}_0(x)$ denote the
position of $u_0$ obtained by the first procedure and
$u^{\delta}_0(x)$ its position obtained by the second. Let
$\ell^{\varepsilon}$ and $\ell^{\delta}$ denote the lines
determined by the functions $u^{\varepsilon}_0(x)$ and
$u^{\delta}_0(x)$. Suppose that $x$ is very large. Since, by inequality~(\ref{egyenlotlenseg}), we have
$\overline{u_k(x)O}>1.5\overline{u_1(x)O}$, both
$u^{\varepsilon}_0(x)$ and $u^{\delta}_0(x)$ are on the line
$\ell_1(u_1(x))$, very far in the positive direction. Therefore, both of them
are
above the line
$\ell_{\pi/10}$. On the other hand, if $x<0$ is very small (i.e.,
if $|x|$ is very big), both $u^{\varepsilon}_0(x)$ and
$u^{\delta}_0(x)$ lie below the line $\ell_{\pi/10}$. It follows
that the slopes of $\ell^{\varepsilon}$ and $\ell^{\delta}$ are
larger than $\pi/10$, but smaller than $\pi/5$.

Suppose that neither $\varepsilon$ nor $\delta$ is 2-good. Then
both $\ell^{\varepsilon}$ and $\ell^{\delta}$ pass through $t_0$.
That is, for a suitable value of $x$, we have
$u^{\varepsilon}_0(x)=t_0$. We distinguish two cases.

\noindent{\bf Case 1:} $u^{\varepsilon}_0(x)=t_0=u_k(x)$. Then, as
$x$ varies, the line determined by $u_k(x)$ coincides with
$\ell_2(t_0)$. Consequently, $t_0$ and $u_k$ are connected in
$G^1$, a contradiction.

\noindent{\bf Case 2:} $u^{\varepsilon}_0(x)=t_0\neq u_k(x)$. In
order to get a contradiction, we try to determine the position of
$u^{\delta}_0(x)$. If we consider {\sc Step} $k+1$ in {\sc
Modified\-Procedure}$(\bar{G}, C, u_0, u_1, x, \varepsilon)$ and
in
{\sc Modified\-Procedure}\-$(\bar{G}, C, u_0, u_1, x, \delta)$, we
can conclude that $u_1(x)$ lies on $\ell_1(u^{\varepsilon}_0)=\ell_1(t_0)$,
$u^{\delta}_0(x)$ lies on $\ell_1(u_1(x))$, therefore,
$u^{\delta}_0(x)$ lies on $\ell_1(t_0)$. On the other hand,
$u^{\delta}_0(x)$ lies on $\ell^{\delta}$, and, by assumption,
$\ell^{\delta}$ passes through $t_0$. However, we have shown that
$\ell^{\delta}$ and $\ell_1(t_0)$ have different slopes,
therefore, $u^{\delta}_0(x)$ must be at their intersection point,
so we have $u^{\delta}_0(x)=u^{\varepsilon}_0(x)=t_0$.

Considering again {\sc Step} $k+1$ in {\sc
Modi\-fied\-Procedure}$(\bar{G}, C, u_0, u_1, x, \varepsilon)$ and
in {\sc Modi\-fied\-Procedure}$(\bar{G}, C, u_0, u_1, x, \delta)$,
we can conclude that the point
$u^{\delta}_0(x)=t_0=u^{\varepsilon}_0(x)$ belongs to both
$\ell_{\varepsilon}(u_k(x))$ and $\ell_{\delta}(u_k(x))$. This
contradicts our assumption that $u_k(x)$ is different from
$u^{\delta}_0(x)=t_0=u^{\varepsilon}_0(x)$. \hfill $\Box$

\smallskip

By Claim \ref{slopeparc5}, for every $\varepsilon< 1/100$ and with finitely many
exceptions for every value of $x$, {\sc
Modified\-Procedure}\-$(\bar{G}, C, u_0, u_1, x, \varepsilon)$
produces a proper drawing of $\bar{G}$. When we want to add the
edge $u_0t_0$, the slope of $u_0(x)t_0$ may coincide with the
slope of $u(x)u'(x)$, for some $u, u'\in \bar G$. The following
statement guarantees that this does not happen ``too often''. We
use $\alpha(\vec{u})$ to denote the {\em slope} of a vector
$\vec{u}$.

\begin{claim} \label{slopeparc7} Let $\vec{u}(x)$ and $\vec{v}(x)$:
${\mbox R}\rightarrow {\mbox R}^2$ be two linear functions, and
let $\ell(u)$ and $\ell(v)$ denote the lines determined by
$\vec{u}(x)$ and $\vec{v}(x)$. Suppose that for some
$x_1<x_2<x_3$, the vectors $\vec{u}, \vec{v}$ do not vanish and
that their slopes coincide, that is,
$\alpha(\vec{u}(x_1))=\alpha(\vec{v}(x_1))$,
$\alpha(\vec{u}(x_2))=\alpha(\vec{v}(x_2))$, and
$\alpha(\vec{u}(x_3))=\alpha(\vec{v}(x_3))$. Then $\ell(u)$ and
$\ell(v)$ must be parallel.
\end{claim}

\begin{figure}[ht]
\begin{center}
\scalebox{0.42}{\includegraphics{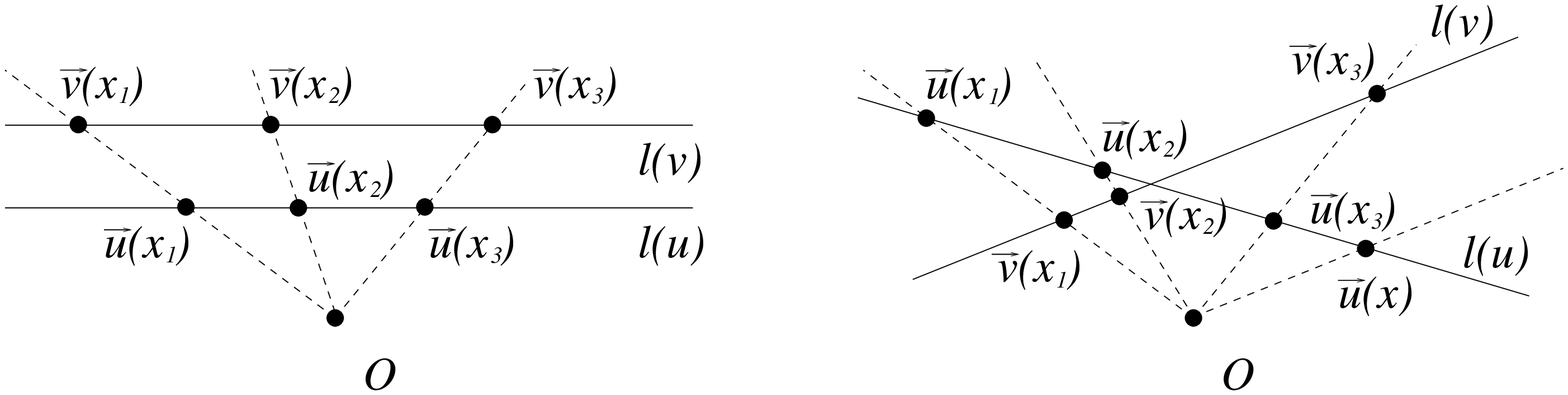}}
\caption{$\ell(u)$ and
$\ell(v)$ must be parallel.}
\end{center}
\end{figure}

\noindent {\bf Proof.} If $\ell(u)$ passes through the origin,
then for every value of $x$, $\vec{u}(x)$ has the same slope. In
particular,
$\alpha(\vec{v}(x_1))=\alpha(\vec{v}(x_2))=\alpha(\vec{v}(x_3))$.
Therefore, $\ell(v)$ also passes through the origin and is
parallel to $\ell(u)$. (In fact, we have $\ell(u)=\ell(v)$.) We
can argue analogously if $\ell(u)$ passes through the origin.
Thus, in what follows, we can assume that neither $\ell(u)$ nor
$\ell(v)$ passes through the origin.

Suppose that $\alpha(\vec{u}(x_1))=\alpha(\vec{v}(x_1))$,
$\alpha(\vec{u}(x_2))=\alpha(\vec{v}(x_2))$, and
$\alpha(\vec{u}(x_3))=\alpha(\vec{v}(x_3))$. For any $x$, define
$\vec{w}(x)$ as the intersection point of $\ell(v)$ and the line
connecting the origin to $\vec{u}(x)$, provided that they
intersect. Clearly, $\vec{v}(x)=\vec{w}(x)$ for $x=x_1, x_2, x_3$,
and $\vec{u}(x)$ and $\vec{w}(x)$ have the same slope for every
$x$. The transformation $\vec{u}(x)\rightarrow \vec{w}(x)$ is a
projective transformation from $\ell(u)$ to $\ell(v)$, therefore,
it preserves the cross ratio of any four points. That is, for any
$x$, we have
$$\left( \vec{u}(x_1), \vec{u}(x_2);
\vec{u}(x_3), \vec{u}(x)\right)= \left( \vec{w}(x_1),
\vec{w}(x_2); \vec{w}(x_3), \vec{w}(x)\right).$$ Since both
$\vec{u}(x)$ and $\vec{v}(x)$ are linear functions, we also have
$$\left( \vec{u}(x_1), \vec{u}(x_2);
\vec{u}(x_3), \vec{u}(x)\right)= \left( \vec{v}(x_1),
\vec{v}(x_2); \vec{v}(x_3), \vec{v}(x)\right).$$ Hence, we can
conclude that $\vec{v}(x)=\vec{w}(x)$ for all $x$. However, this
is impossible, unless $\ell(u)$ and $\ell(v)$ are parallel.
Indeed, suppose that $\ell(u)$ and $\ell(v)$ are not parallel, and
set $x$ in such a way that $\vec{u}(x)$ is parallel to $\ell(v)$.
Then $\vec{w}(x)$ cannot have the same slope as $\vec{u}(x)$, a
contradiction. \hfill $\Box$

\smallskip

Suppose that $\varepsilon$ is 2-good and let us fix it. As above,
let $u^{\varepsilon}_0(x)$ be the position of $u_0$ obtained by
{\sc Modified\-Procedure}$(\bar{G}, C, u_0, u_1, x, \varepsilon)$,
and let $\ell^{\varepsilon}$ be the line determined by
$u^{\varepsilon}_0(x)$.

Suppose also that there exist two independent vertices of $\bar
G$, $u, u'\neq u_0$, such that the line determined by
$\vec{uu'}(x)$ is parallel to $\ell^{\varepsilon}$. Then we say
that $\varepsilon$ is {\em 3-bad} for $\bar{G}$. If $\varepsilon$
is 2-good and it is not 3-bad for $\bar{G}$, then we say that it
is {\em 3-good} for $\bar{G}$.

It is easy to see that, for any $0<\varepsilon, \delta< 1/100$,
$\ell^{\varepsilon}$ and $\ell^{\delta}$ are not parallel,
therefore, for any fixed $u, u'$, there is at most one value of
$\varepsilon$ for which the line determined by $\vec{uu'}(x)$ is
parallel to $\ell^{\varepsilon}$. Thus, with finitely many
exceptions, all values $0<\varepsilon< 1/100$ are 3-good.

Summarizing, we have obtained the following.

\begin{claim}\label{slopeparc8} Suppose that $\varepsilon$ is 3-good
for $\bar{G}$. With finitely many exceptions, for every value of
$x$, {\sc Modified\-Procedure}\-$(\bar{G}, C, u_0, u_1, x,
\varepsilon)$ gives a proper drawing of ${G^1}$.\hfill $\Box$
\end{claim} 

Now we are in a position to complete the proof of Theorem \ref{slopepar1}.
Proceed with each of the components as described above for $G^1$.
For any fixed $i$, let $u^i_0v^i_0$ be the edge deleted from
$G^i$, and denote the resulting graphs by $\bar{G}^1, \ldots,
\bar{G}^m$. Let $0<\varepsilon<1/100$ be fixed in such a way that
$\varepsilon$ is 3-good for all graphs $\bar{G}^1, \ldots,
\bar{G}^m$. This can be achieved, in view of the fact that there
are only finitely many values of $\varepsilon$ which are not
3-good. Perform {\sc Modified\-Procedure}$(\bar{G}^i, C^i, u^i_0,
u^i_1, x^i, \varepsilon)$. Now the line $\ell^i$ determined by all
possible locations of $u^i_0$ does not pass through $t^i_0$.

Notice that when {\sc Modified\-Procedure}$(\bar{G}^i, C^i, u^i_0,
u^i_1, x^i, \varepsilon)$ is executed, apart from edges with basic
slopes, we use an edge with slope $r\pi/5\pm\varepsilon$, for some
integer $r \mod 5$. By using rotations through $\pi/5$ and a
reflection, if necessary, we can achieve that each component $\bar{G}^i$ is
drawn using the basic slopes and one edge of slope $\varepsilon$.

It remains to set the values of $x_i$ and draw the missing edges
$u^i_0v^i_0$. Since the line $\ell^i$ determined by the possible
locations of $u^i_0$ does not pass through $t^i_0$, by varying the
value of $x^i$, we can attain any slope for the missing edge
$t^i_0u^i_0$, except for the slope of $\ell^i$. By Claim \ref{slopeparc8}, with
finitely many exceptions, all values of $x^i$ produce a proper
drawing of $G^i$. Therefore, we can choose $x^1, \ldots,
x^m$ so that all segments $t^i_0u^i_0$ have the same slope and
every component $G^i$ is properly drawn using the same seven slopes.
Translating the resulting drawings
through suitable vectors gives a proper drawing of $G$, this completes the proof
of Theorem \ref{slopepar1}.

\subsection{Concluding Remarks}

In the proof of Theorem \ref{slopepar1}, the slopes we use depend on the graph
$G$. However, the proof shows that one can simultaneously embed
all cubic graphs using only {\em seven} fixed slopes.

It is unnecessary to use $|x|\ge 50$, in every step, we could pick
any $x$, with finitely many exceptions.

It seems to be only a technical problem that we needed {\em two}
extra directions in the proof of Theorem \ref{slopepar1}. We believe that {\em
one} extra direction would suffice.

The most interesting problem that remains open is to decide
whether the number of slopes needed for graphs of maximum degree
{\em four} is bounded.

Another not much investigated question is to estimate the complexity of computing the slope parameter of a graph. A related problem is to decide under what conditions a graph can be drawn on a polynomial sized grid using a fixed number of slopes. 

\begin{quest}\label{polygrid} Is it possible to draw all cubic graphs with a bounded number of slopes on a polynomial sized grid?
\end{quest}

%% file: slopeplan.tex
This section is based on our paper with Bal\'azs Keszegh and J\'anos Pach, Drawing planar graphs of bounded degree with few slopes \cite{KPP10}.

In this section, we will be concerned with drawings of planar graphs. Unless it is stated otherwise, all {\em drawings} will be {\em non-crossing}, that is, no two arcs that represent different edges have an interior point in common.

Every planar graph admits a straight-line drawing \cite{F48}. From the practical and aesthetical point of view, it makes sense to minimize the number of slopes we use \cite{WC94}. Remember that the {\em planar slope number} of a planar graph $G$ is the smallest number $s$
with the property that $G$ has a straight-line drawing with edges
of at most $s$ distinct slopes. If $G$ has a vertex of degree $d$, then its planar slope number is at least $\lceil d/2\rceil$, because in a straight-line drawing no two edges are allowed to overlap.

Dujmo\-vi\'c, Eppstein, Suderman, and Wood \cite{DESW07} raised the question whether there exists a function $f$ with the property that the planar slope number of every planar graph with maximum degree $d$ can be bounded from above by $f(d)$. Jelinek et al.~\cite{JJ10} have shown that the answer is yes for {\em outerplanar} graphs, that is, for planar graphs that can be drawn so that all of their vertices lie on the outer face. In Section \ref{sec:slopeplan}.2, we answer this question in full generality. We prove the following.

\begin{thm}\label{one} Every planar graph with maximum degree $d$ admits a straight-line drawing, using segments of $O(d^2(3+2\sqrt 3)^{12d})\le K^d$ distinct slopes.
\end{thm}

The proof is based on a paper of Malitz and Papakostas \cite{MP94}, who used Koebe's theorem \cite{K36} on disk representations of planar graphs to prove the existence of drawings with relatively large angular resolution. As the proof of these theorems, our argument is nonconstructive; it only yields a nondeterministic algorithm with running time $O(dn)$. However, if one combines our result with a polynomial time algorithm that computes the $\epsilon$-approximation of the disk representation (see e.g. Mohar \cite{M93}), then one can obtain a deterministic algorithm running in time exponential in $d$ but polynomial in $n$.

For $d=3$, much stronger results are known than the one given by our theorem. Dujmovi\'c at al. \cite{DESW07} showed that every planar graph with maximum degree 3 admits a straight-line drawing using at most 3 different slopes, except for at most 3 edges of the outer face, which may require 3 additional slopes. This complements Ungar's old theorem \cite{U53}, according to which 3-regular, 4-edge-connected planar graphs require only 2 slopes and 4 extra edges.

The exponential upper bound in Theorem~\ref{one} is probably far from being optimal. However, we were unable to give any superlinear lower bound for the largest planar slope number of a planar graph with maximum degree $d$. The best constructions we are aware of are presented in Section \ref{sec:slopeplan}.5.

We also show that significantly fewer slopes are sufficient if we are allowed to represent the edges by short noncrossing poly\-gonal paths. If such a path consists of $k+1$ segments, we say that the edge is drawn by $k$ {\em bends}. In Section \ref{sec:slopeplan}.3, we show if we allow one bend per edge, then every planar graph can be drawn using segments with $O(d)$ slopes.

\begin{thm}\label{onebend}
Every planar graph $G$ with maximum degree $d$ can be drawn with at most $1$ bend per edge, using at most $2d$ slopes.
\end{thm}

Allowing {\em two} bends per edge yields an optimal result: almost all planar graphs with maximum degree $d$ can be drawn with $\left\lceil d/2 \right\rceil$ slopes. In Section \ref{sec:slopeplan}.4, we establish

\begin{thm}\label{twobendsimproved}\label{twobends}
Every planar graph $G$ with maximum degree $d\ge 3$ can be drawn with at most 2 bends per edge, using segments of at most $\lceil d/2\rceil$ distinct slopes. The only exception is the graph formed by the edges of an octahedron, which is 4-regular, but requires 3 slopes. These bounds are best possible.
\end{thm}

It follows from the proof of Theorem~\ref{twobends} that in the cyclic order of directions, the slopes of the edges incident to any given vertex form a contiguous interval. Moreover, the $\lceil d/2\rceil$ directions we use can be chosen to be equally spaced in $[0,2\pi)$. We were unable to guarantee such a nice property in Theorem~\ref{onebend}: even for a fixed $d$, as the number of vertices increases, the smallest difference between the $2d-2$ slopes we used tends to zero. We suspect that this property is only an unpleasant artifact of our proof technique.

\subsection{Straight-line Drawings - Proof of Theorem~\ref{one}}

Note that it is sufficient to prove the theorem for triangulated planar graphs,  because any planar graph can be triangulated by adding vertices and edges so that the degree of each vertex increases only by a factor of at most three \cite{PT06}, so at the end we will lose this factor.

We need the following result from \cite{MP94}, which is not displayed as a theorem there, but is stated right above Theorem 2.2.

\begin{lem}\label{mp} {\rm (Malitz-Papakostas)} The vertices of any triangulated planar graph $G$ with maximum degree $d$ can be represented by nonoverlapping disks in the plane so that two disks are tangent to each other if and only if the corresponding vertices are adjacent, and the ratio of the radii of any two disks that are tangent to each other is at least $\alpha^{d-2},$ where $\alpha=\frac 1{3+2\sqrt 3}\approx 0.15$.
\end{lem}

Lemma~\ref{mp} can be established by taking any representation of the vertices of $G$ by tangent disks, as guaranteed by Koebe's theorem, and applying a conformal mapping to the plane that takes the disks corresponding to the three vertices of the outer face to disks of the same radii. The lemma now follows by the observation that any internal disk is surrounded by a ring of at most $d$ mutually touching disks, and the radius of none of them can be much smaller than that of the central disk.

The idea of the proof of Theorem \ref{one} is as follows. Let $G$ be a triangulated planar graph with maximum degree $d$, and denote its vertices by $v_1, v_2, \ldots$. Consider a disk representation of $G$ meeting the requirements of Lemma~\ref{mp}. Let $D_i$ denote the disk that represents $v_i$, and let $O_i$ be the center of $D_i$. By properly scaling the picture if necessary, we can assume without loss of generality that the radius of the smallest disk $D_i$ is sufficiently large. Place an integer grid on the plane, and replace each center $O_i$ by the nearest grid point. Connecting the corresponding pairs of grid points by segments, we obtain a straight-line drawing of $G$. The advantage of using a grid is that in this way we have control of the slopes of the edges. The trouble is that the size of the grid, and thus the number of slopes used, is very large. Therefore, in the neighborhood of each disk $D_i$, we use a portion of a grid whose side length is proportional to the radius of the disk. These grids will nicely fit together, and each edge will connect two nearby points belonging to grids of comparable sizes. Hence, the number of slopes used will be bounded. See Figure \ref{fig:discs}.

\begin{figure}[ht]
\centering
		\includegraphics[scale=0.65]{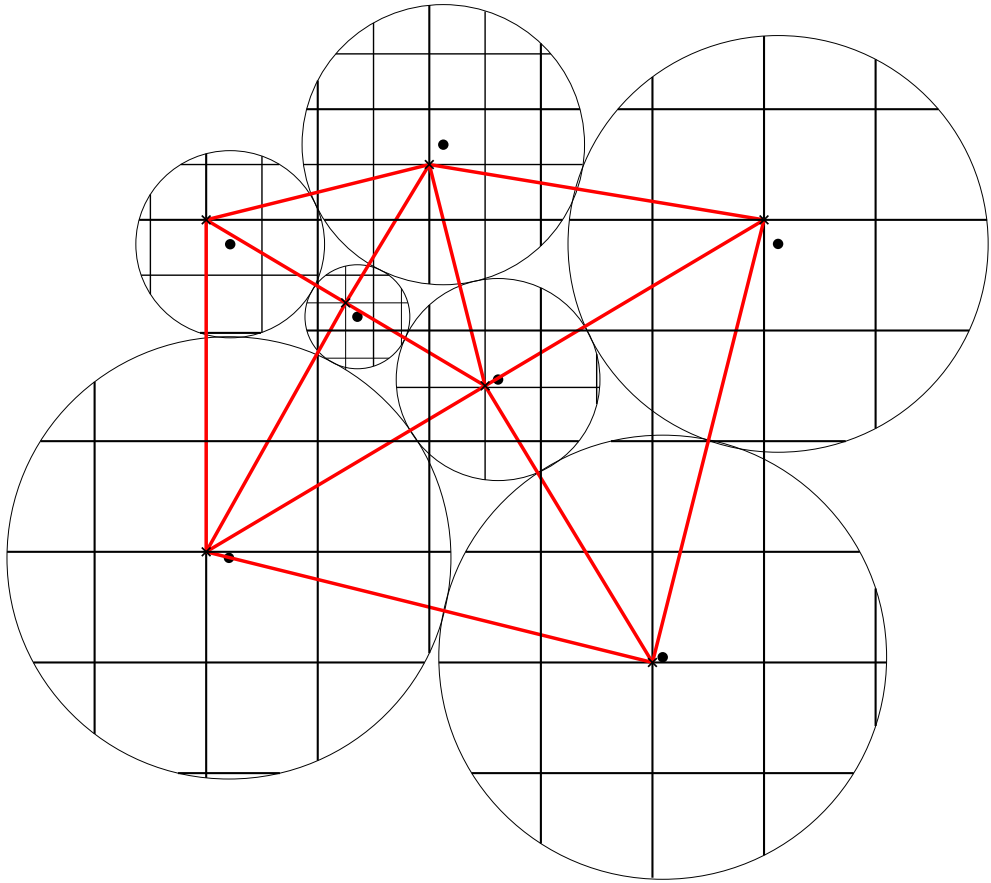}
               \hspace{5mm}
		 \caption{Straight-line graph from disk representation}
		\label{fig:discs}
\end{figure}

Now we work out the details. Let $r_i$ denote the radius of $D_i\; (i=1,2\ldots)$, and suppose without loss of generality that $r^*$, the radius of the smallest disk is
$$r^*=min_i r_i=\sqrt{2}/\alpha^{d-2}>1,$$
where $\alpha$ denotes the same constant as in Lemma~\ref{mp}.

Let $s_i=\lfloor\log_d (r_i/r^*)\rfloor\ge 0$, and represent each vertex $v_i$ by the integer point nearest to $O_i$ such that both of its coordinates are divisible by $d^{s_i}$. (Taking a coordinate system in general position, we can make sure that this point is unique.) For simplicity, the point representing $v_i$ will also be denoted by $v_i$. Obviously, we have that the distance between $O_i$ and $v_i$ satisfies
$$\overline{O_iv_i}< \frac{d^{s_i}}{\sqrt{2}}.$$

Since the centers $O_i$ of the disks induce a (crossing-free) straight-line drawing of $G$, in order to prove that moving the vertices to $v_i$ does not create a crossing, it is sufficient to verify the following statement.

\begin{lem} For any three mutually adjacent vertices, $v_i, v_j, v_k$ in $G$, the orientation of the triangles $O_iO_jO_k$ and $v_iv_jv_k$ are the same.
\end{lem}

\begin{proof} By Lemma~\ref{mp}, the ratio between the radii of any two adjacent disks is at least $\alpha^{d-2}$. Suppose without loss of generality that $r_i\ge r_j\ge r_k\ge \alpha^{d-2}r_i$.
For the orientation to change, at least one of $\overline{O_iv_i}$, $\overline{O_jv_j}$, or $\overline{O_kv_k}$ must be at least half of the smallest altitude of the triangle $O_iO_jO_k$, which is at least $\frac{r_k}{2}$.

On the other hand, as we have seen before, each of these numbers is smaller than
$$\frac{d^{s_i}}{\sqrt{2}}\le\frac{r_i/r^*}{\sqrt{2}}
=\frac{{\alpha}^{d-2}r_i}{2}\le \frac{r_k}{2}$$
which completes the proof.
\end{proof}

Now we are ready to complete the proof of Theorem~\ref{one}. Take an edge $v_iv_j$ of $G$, with $r_i\ge r_j\ge \alpha^{d-2}r_i$.
The length of this edge can be bounded from above by
$$\overline{v_iv_j}\le  \overline{O_iO_j}+\overline{O_iv_i}+\overline{O_jv_j}
\le r_i+r_j+\frac{d^{s_i}}{\sqrt{2}}+\frac{d^{s_j}}{\sqrt{2}}
\le 2r_i+{\sqrt{2}}{d^{s_i}}
\le 2r_i+{\sqrt{2}}{r_i/r^*}$$
$$ \le r_i/r^*(2r^*+\sqrt{2})
\le \frac{r_j/r^*}{\alpha^{d-2}}(2r^*+\sqrt{2})
< \frac{d^{s_j+1}}{\alpha^{d-2}}(\frac{2\sqrt{2}}{\alpha^{d-2}}+\sqrt{2}).$$

According to our construction, the coordinates of $v_j$ are integers divisible by $d^{s_j}$, and the coordinates of $v_i$ are integers divisible by $d^{s_i}\ge d^{s_j}$, thus also by $d^{s_j}$.

Thus, shrinking the edge $v_iv_j$ by a factor of $d^{s_j}$, we obtain a segment whose endpoints are integer points at a distance at most
$\frac{d}{\alpha^{d-2}}(\frac{2\sqrt{2}}{\alpha^{d-2}}+\sqrt{2}).$
Denoting this number by $R(d)$, we obtain that the number of possible slopes for $v_iv_j$, and hence for any other edge in the embedding, cannot exceed the number of integer points in a disk of radius $R(d)$ around the origin. Thus, the planar slope number of any triangulated planar graph of maximum degree $d$ is at most roughly $R^2(d)\pi=O(d^2/\alpha^{4d})$, which completes the proof. \hfill $\square$

\medskip

Our proof is based on the result of Malitz and Papakostas that does not have an algorithmic version. However, with some reverse engineering, we can obtain a nondeterministic algorithm for drawing a triangulated planar graph of bounded degree with a bounded number of slopes. Because of the enormous constants in our expressions, this algorithm is only of theoretical interest. Here is a brief sketch.
\medskip

{\em Nondeterministic algorithm.} First, we guess the three vertices of the outer face and their coordinates in the grid scaled according to their radii. Then embed the remaining vertices one by one. For each vertex, we guess the radius of the corresponding disk as well as its coordinates in the proportionally scaled grid. This algorithm runs in nondeterministic $O(dn)$ time.

\subsection{One Bend per Edge - Proof of Theorem~\ref{onebend}}

In this section, we represent edges by noncrossing polygonal paths, each consisting of at most two segments. Our goal is to establish Theorem~\ref{onebend}, which states that the total number of directions assumed by these segments grows at most linearly in $d$.

The proof of Theorem~\ref{onebend} is based on a result of Fraysseix et al. \cite{FMR94}, according to which every planar graph can be represented as a contact graph of $T$-shapes. A {\em $T$-shape} consists of a vertical and a horizontal segment such that the upper endpoint of the vertical segment lies in the interior of the horizontal segment. The vertical and horizontal segments of $T$ are called its {\em leg} and {\em hat}, while their point of intersection is the {\em center} of the $T$-shape. The two endpoints of the hat and the bottom endpoint of the leg are called {\em ends} of the $T$-shape.

Two $T$-shapes are {\em noncrossing} if the interiors of their segments are disjoint. Two $T$-shapes are {\em tangent} to each other if they are noncrossing but they have a point in common.

\begin{lem}\label{Tshapes} {\rm (Fraysseix et al.)} The vertices of any planar graph with $n$ vertices can be represented by noncrossing $T$-shapes such that
\begin{enumerate}
\item two $T$-shapes are tangent to each other if and only if the corresponding vertices are adjacent;
\item the centers and the ends of the $T$-shapes belong to an $n\times n$ grid.
\end{enumerate}
Moreover, such a representation can be computed in linear time.
\end{lem}

The proof of the lemma is based on the canonical ordering of the vertices of a planar graph, introduced in \cite{FPP89}.

\begin{proof}[Proof of Theorem \ref{onebend}]
Consider a representation of $G$ by $T$-shapes satisfying the conditions in the lemma. See Figure \ref{fig:tshapesleft}. For any $v\in V(G)$, let $T_v$ denote the corresponding $T$-shape. We define a drawing of $G$, in which the vertex $v$ is mapped to the center of $T_v$. To simplify the presentation, the center of $T_v$ is also denoted by $v$. For any $uv\in E(G)$, let $p_{uv}$ denote the point of tangency of $T_u$ and $T_v$. The polygonal path $up_{uv}v$ consists of a horizontal and a vertical segment, and these paths together almost form a drawing of $G$ with one bend per edge, using segments of two different slopes. The only problem is that these paths partially overlap in the neighborhoods of their endpoints. Therefore, we modify them by replacing their horizontal and vertical pieces by almost horizontal and almost vertical ones, as follows.

For any $1\le i\le d$, let $\alpha_i$ denote the slope of the (almost horizontal) line connecting the origin $(0,0)$ to the point $(2in,-1)$. Analogously, let $\beta_i$ denote the slope of the (almost vertical) line passing through $(0,0)$ and $(1,2in)$.

Fix a $T$-shape $T_v$ in the representation of $G$. It is tangent to at most $d$ other $T$-shapes. Starting at its center $v$, let us pass around $T_v$ in the counterclockwise direction, so that we first visit the upper left side of its hat, then its lower left side, then the left side and right side of its leg, etc. We number the points of tangencies along $T_v$ in this order. (Note that there are no points of tangencies on the lower side of the hat.)

Suppose now that the hat of a $T$-shape $T_u$ is tangent to the leg of $T_v$, and let $p_{uv}$ be their point of tangency. Assume that $p_{uv}$ was the number $i$ point of tangency along $T_u$ and the number $j$ point of tangency along $T_v$. Let $p'_{uv}$ denote the unique point of intersection of the (almost horizontal) line through $u$ with slope $\alpha_i$ and the (almost vertical) line through $v$ with slope $\beta_j$. In our drawing of $G$, the edge $uv$ will be represented by the polygonal path $up'_{uv}v$. See Figure \ref{fig:tshapesright} for the resulting drawing and Figure \ref{fig:tshapesmiddle} for a version distorted for the human eye to show the underlying structure.

\begin{figure}[ht]
\centering
  \subfigure[]
{		\includegraphics[scale=0.4]{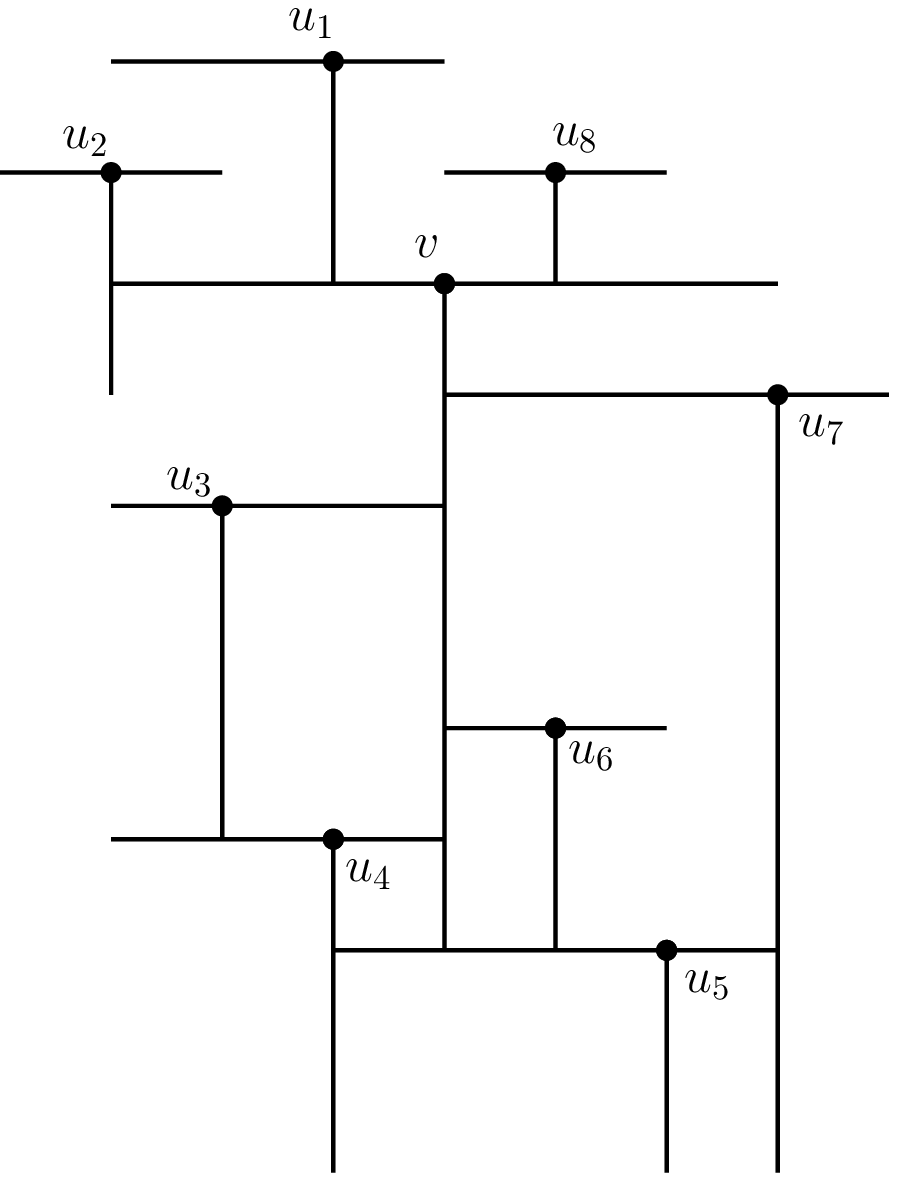}
		\hspace{1cm}
				\label{fig:tshapesleft}
	}
	\subfigure[]
{
		\includegraphics[scale=0.5]{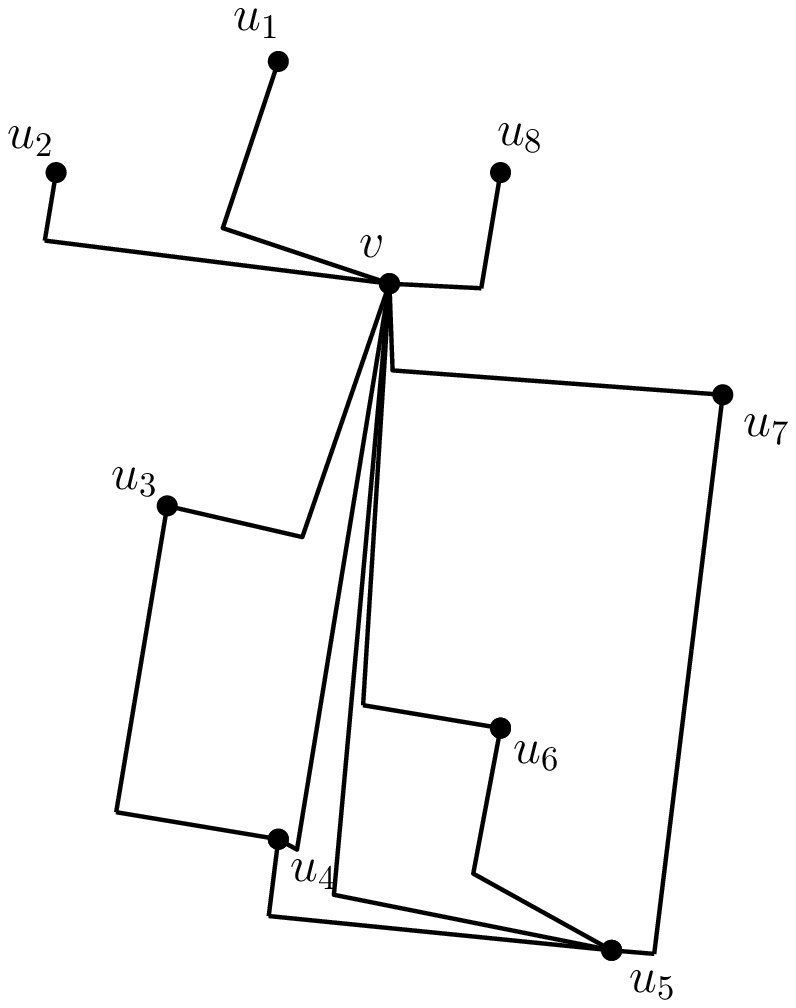}
		\hspace{1cm}
				\label{fig:tshapesmiddle}
		}	
  \subfigure[]
{\includegraphics[scale=0.5]{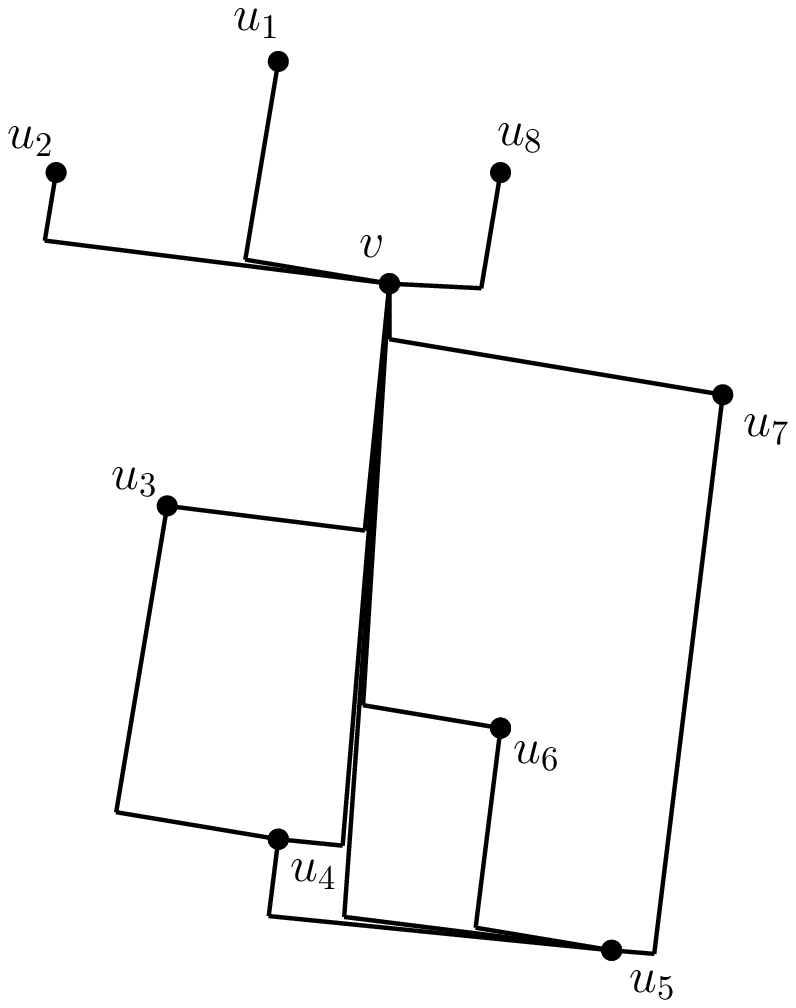}
               \hspace{5mm}
		 		\label{fig:tshapesright}
		 }
		 \caption{Representation with $T$-shapes and the drawing with one bend per edge}
		\label{fig:tshapes}
\end{figure}

Since the segments we used are almost horizontal or vertical, the modified edges $up'_{uv}v$ are very close (within distance $1/2$) of the original polygonal paths $up_{uv}v$. Thus, no two nonadjacent edges can cross each other. On the other hand, the order in which we picked the slopes around each $v$ guarantees that no two edges incident to $v$ will cross or overlap. This completes the proof.
\end{proof}

\subsection{Two Bends per Edge - Proof of Theorem~\ref{twobendsimproved}}

In this section, we draw the edges of a planar graph by polygonal paths with at most {\em two} bends. Our aim is to establish Theorem~\ref{twobendsimproved}.

Note that the statement is trivially true for $d=1$ and is false for $d=2$. It is sufficient to prove Theorem~\ref{twobendsimproved} for even values of $d$. For $d=4$, the assertion was first proved by Liu et al. \cite{LMS91} and later, independently, by Biedl and Kant \cite{BK98} (also that the only exception is the octahedral graph). The latter approach is based on the notion of $st$-ordering of biconnected ($2$-connected) graphs from Lempel et al. \cite{LEC67}. We will show that this method generalizes to higher values of $d\ge 5$. As it is sufficient to prove the statement for even values of $d$, from now on we suppose that $d\ge 6$ even. %in particular the statement for $d=6$ implies it also for $d=5$.
We will argue that it is enough to consider {\em biconnected} graphs. Then we review some crucial claims from \cite{BK98} that will enable us to complete the proof. We start with some notation.

Take $d\ge 5$ lines that can be obtained from a vertical line by clockwise rotation by $0, \pi/d, 2\pi/d,$ $\ldots, (d-1)\pi/d$ degrees. Their slopes are called the $d$ {\em regular slopes}. We will use these slopes to draw $G$. Since these slopes depend only on $d$ and not on $G$, it is enough to prove the theorem for connected graphs. If a graph is not connected, its components can be drawn separately. 

In this section we always use the term ``slope'' to mean a regular slope. The {\em  directed slope} of a directed line or segment is defined as the angle (mod $2\pi$) of a clockwise rotation that takes it to a position parallel to the upward directed $y$-axis. Thus, if the directed slopes of two segments differ by $\pi$, then they have the same slope. We say that the slopes of the segments incident to a point $p$ form a {\em contiguous interval} if the set $\vec S\subset \{0, \pi/d, 2\pi/d,\ldots, (2d-1)\pi/d\}$ of directed slopes of the segments directed away from $p$, has the property that for all but at most one $\alpha\in \vec S$, we have that $\alpha +\pi/d \mod 2\pi \in \vec S$ (see Figure \ref{fig:twobendsboth}).

Finally, we say that $G$ admits a {\em good drawing} if $G$ has a planar drawing such that every edge has at most $2$ bends, every segment of every edge has one of the $\lceil d/2 \rceil$ regular slopes, and the slopes of the segments incident to any vertex form a contiguous interval.
If $t$ is a vertex whose degree is at least two but less than $d$, then we can define the two {\em extremal segments} at $t$ as the segments corresponding to the slopes at the two ends of the contiguous interval formed by the slopes of all the segments incident to $t$.
Also define the {\em $t$-wedge} as the infinite cone bounded by the extension of the two extremal segments, which contains all segments incident to $t$ and none of the ``missing'' segments. See Figure \ref{fig:tcone}. For a degree one vertex $t$ we define the {\em $t$-wedge} as the infinite cone bounded by the extension of the rotations of the segment incident to $t$ around $t$ by $\pm\pi/2d$.

\begin{figure}[ht]
\centering
  \subfigure[]
{		\includegraphics[scale=0.4]{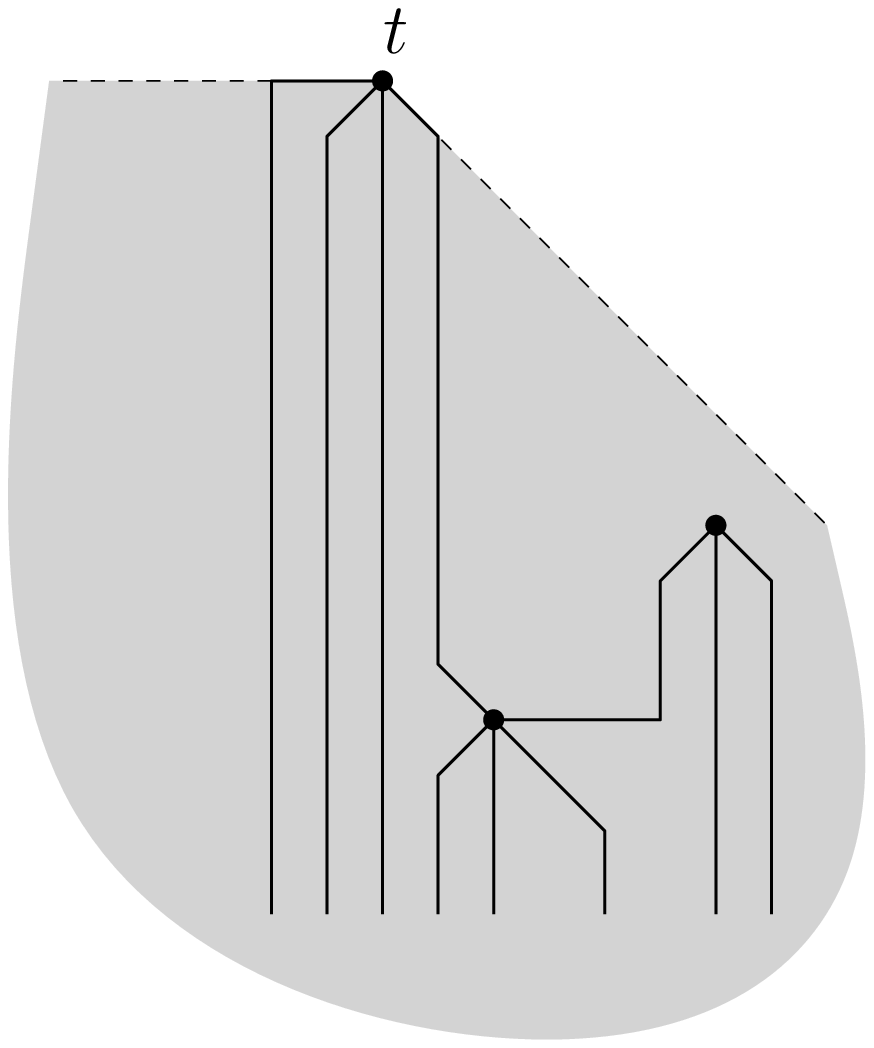}
		\hspace{2cm}
				\label{fig:tcone2}
	}
	\subfigure[]
{
		\includegraphics[scale=0.4]{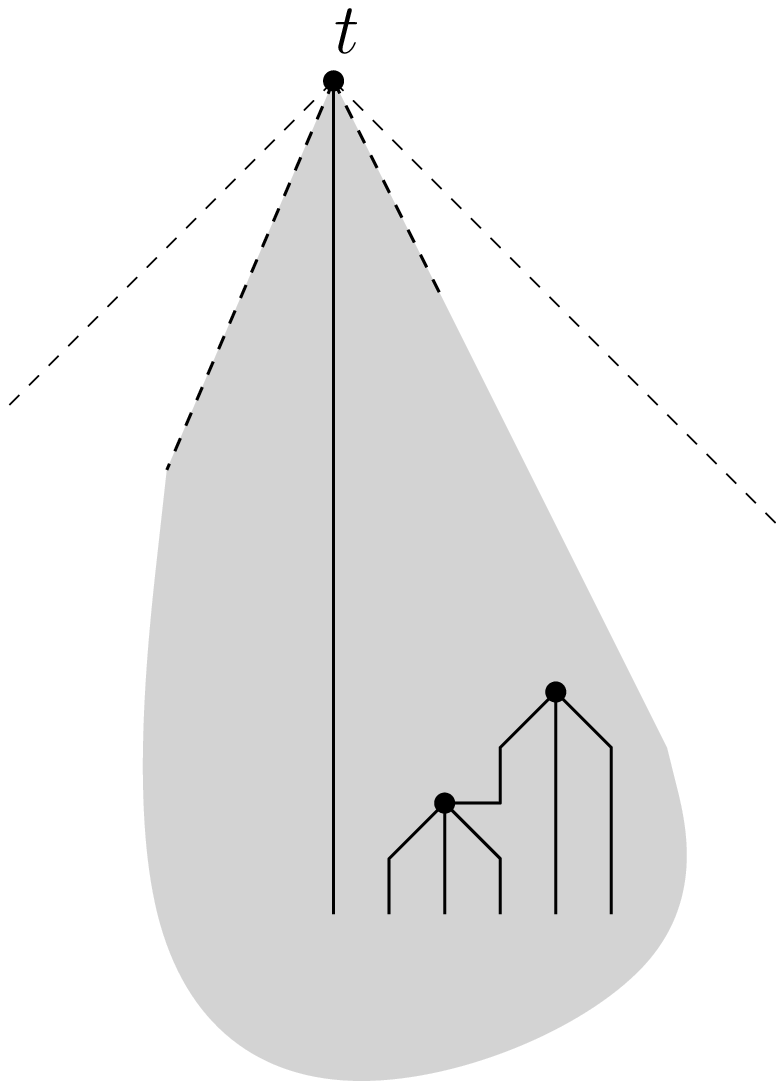}
						\label{fig:tcone1}
		}	
		
		 \caption{The $t$-wedge}
		 		\label{fig:tcone}
\end{figure}

To prove Theorem~\ref{twobendsimproved}, we show by induction that every connected planar graph with maximum degree $d\ge 6$ with an arbitrary $t$ vertex whose degree is strictly less than $d$ admits a good drawing that is contained in the $t$-wedge. Note that such a vertex always exist because of Euler's polyhedral formula, thus Theorem \ref{twobends} is indeed a direct consequence of this statement. First we show how the induction step goes for graphs that have a cut vertex, then (after a lot of definitions) we prove the statement also for biconnected graphs (without the induction hypothesis).

\begin{lem}\label{cutclaim} Let $G$ be a connected planar graph of maximum degree $d$, let $t\in V(G)$ be a vertex whose degree is strictly smaller than $d$, and let $v\in V(G)$ be a cut vertex. Suppose that for any connected planar graph $G'$ of maximum degree $d$, which has fewer than $|V(G)|$ vertices, and for any vertex $t'\in V(G')$ whose degree is strictly smaller than $d$, there is a good drawing of $G'$ that is contained in the $t'$-wedge. Then $G$ also admits a good drawing that is contained in the $t$-wedge.
\end{lem}
\begin{proof} Let $G_1, G_2,\ldots$ denote the connected components of the graph obtained from $G$ after the removal of the cut vertex $v$, and let $G^*_i$ be the subgraph of $G$ induced by $V(G_i)\cup\{v\}$.

If $t=v$ is a cut vertex, then by the induction hypothesis each $G_i^*$ has a good drawing in the $v$-wedge\footnote{Of course the $v$-wedges for the different components are different.}. After performing a suitable rotation for each of these drawings, and identifying their vertices corresponding to $v$, the lemma follows because the slopes of the segments incident to $v$ form a contiguous interval in each component.

If $t\ne v$, then let $G_j$ be the component containing $t$. Using the induction hypothesis, $G_j^*$ has a good drawing. Also, each $G_i^*$ for $i\ge 2$ has a good drawing in the $v$-wedge. As in the previous case, the lemma follows by rotating and possibly scaling down the components for $i\ne j$ and again identifying the vertices corresponding to $v$.
\end{proof}

In view of Lemma~\ref{cutclaim}, in the sequel we consider only biconnected graphs. We need the following definition.

\begin{defi} An ordering of the vertices of a graph, $v_1, v_2,\ldots,v_n$, is said to be an {\em $st$-ordering} if $v_1=s$, $v_n=t$, and if for every $1<i<n$ the vertex $v_i$ has at least one neighbor that precedes it and a neighbor that follows it.
\end{defi}

In \cite{LEC67}, it was shown that any biconnected graph has an $st$-ordering, for any choice of the vertices $s$ and $t$. In \cite{BK98}, this result was slightly strengthened for planar graphs, as follows.

\begin{lem}\label{stlemma} {\rm (Biedl-Kant)} Let $D_G$ be a drawing of a biconnected planar graph, $G$, with vertices $s$ and $t$ on the outer face. Then $G$ has an $st$-ordering for which $s=v_1$, $t=v_n$ and $v_2$ is also a vertex of the outer face and $v_1v_2$ is an edge of the outer face.
\end{lem}

We define $G_i$ to be the subgraph of $G$ induced by the vertices $v_1, v_2, \ldots, v_i$. Note that $G_i$ is connected. If $i$ is fixed, we call the edges between $V(G_i)$ and $V(G)\setminus V(G_i)$ the {\em pending edges}.
For a drawing of $G$, $D_G$, we denote by $D_{G_i}$ the drawing restricted to $G_i$ and to an initial part of each pending edge connected to $G_i$.
%We need some additional propositions about these drawings.

\begin{prop}\label{stremark0}
In the drawing $D_G$ guaranteed by Lemma \ref{stlemma}, $v_{i+1},\ldots v_n$ and the pending edges are in the outer face of $D_{G_i}$.
\end{prop}
\begin{proof}
Suppose for contradiction that for some $i$ and $j>i$, $v_j$ is not in the outer face of $D_{G_i}$. We know that $v_n$ is in the outer face of $D_{G_i}$ as it is on the outer face of $D_G$, thus $v_n$ and $v_j$ are in different faces of $D_{G_i}$. On the other hand, by the definition of $st$-ordering, there is a path in $G$ between $v_j$ and $v_n$ using only vertices from $V(G)\setminus V(G_i)$. The drawing of this path in $D_G$ must lie completely in one face of $D_{G_i}$. Thus, $v_j$ and $v_n$ must also lie in the same face, a contradiction.
Since the pending edges connect $V(G_i)$ and $V(G)\setminus V(G_i)$, they must also lie in the outer face.
\end{proof}

\begin{figure}[ht]
\centering
  \subfigure[The pending-order of the pending edges in $D_{G_i}$]
{		\includegraphics[scale=0.7]{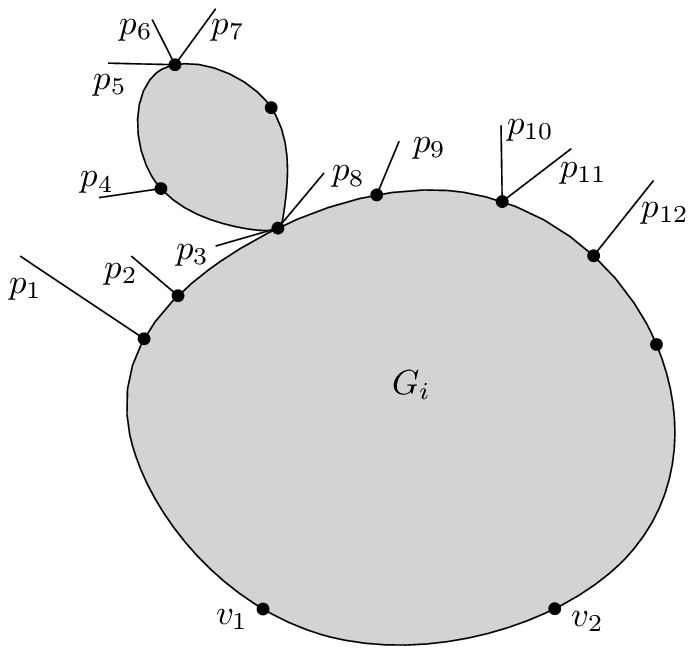}
\hspace{5mm}
				\label{fig:storder_a}
	}
       \hspace{2cm}
  \subfigure[The preceding neighbors of $v_{i+1}$ are consecutive in the pending-order]
{		
\includegraphics[scale=0.7]{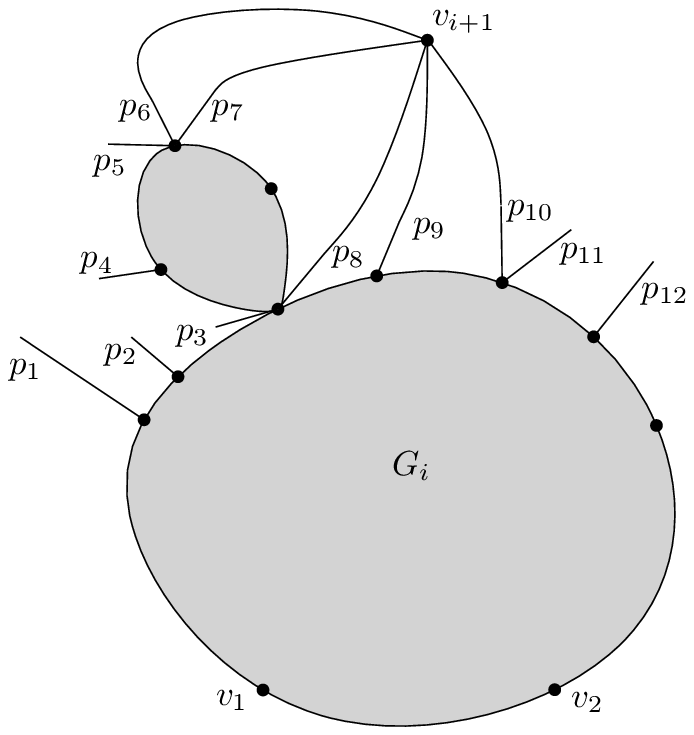}
\hspace{5mm}
				\label{fig:storder_b}
	}		 
		 \caption{Properties of the $st$-ordering}
\end{figure}

By Lemma \ref{stlemma}, the edge $v_1v_2$ lies on the boundary of the outer face of $D_{G_i}$, for any $i\ge 2$.
Thus, we can order the pending edges connecting $V(G_i)$ and $V(G)\setminus V(G_i)$ by walking in $D_G$ from $v_1$ to $v_2$ around $D_{G_i}$ on the side that does not consist of only the $v_1v_2$ edge, see Figure \ref{fig:storder_a}. We call this the {\em pending-order} of the pending edges between $V(G_i)$ and $V(G)\setminus V(G_i)$ (this order may depend on $D_G$). Proposition \ref{stremark0} implies 

\begin{prop}\label{stremark}
The edges connecting $v_{i+1}$ to vertices preceding it form an interval of consecutive elements in the pending-order of the edges between $V(G_i)$ and $V(G)\setminus V(G_i)$.
\end{prop}

For an illustration see Figure \ref{fig:storder_a}.

Two drawings of the same graph are said to be {\em equivalent} if the circular order of the edges incident to each vertex is the same in both drawings. Note that in this order we also include the pending edges (which are differentiated with respect to their yet not drawn end).

Now we are ready to finish the proof of Theorem \ref{twobends}, as the following lemma is the only missing step.

\begin{figure}[ht]
\centering
  \subfigure[Drawing $v_1$, $v_2$ and the edges incident to them]
  {
		\includegraphics[scale=0.36]{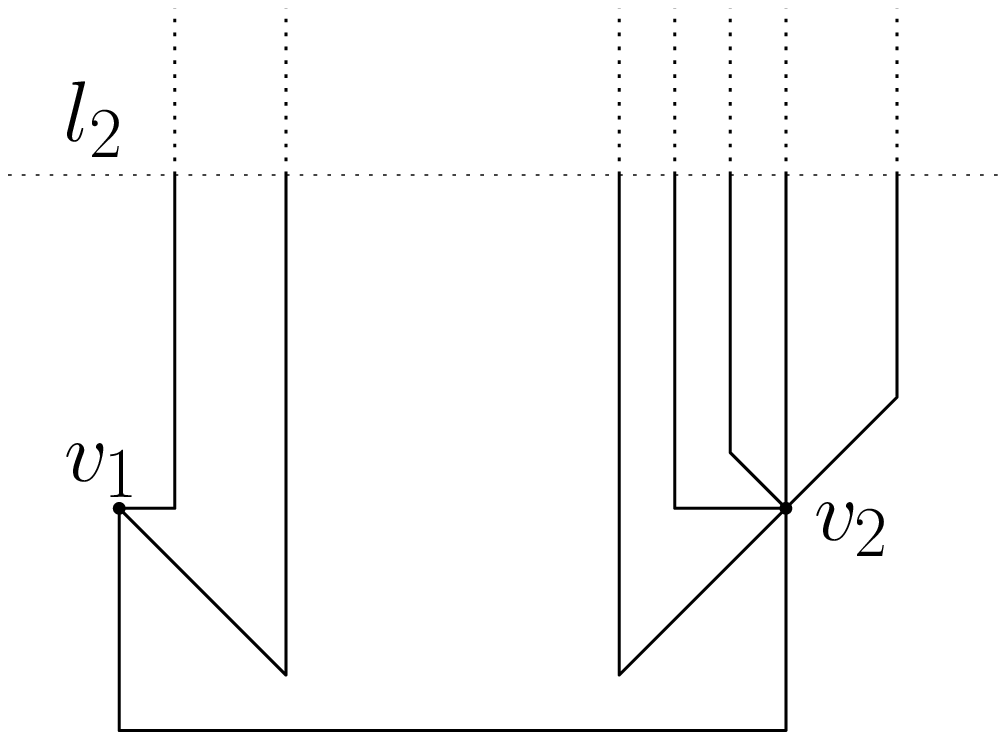}
               \hspace{5mm}
		 \label{fig:twobends}
      }
       \hspace{1cm}
      \subfigure[Adding $v_{i}$; partial edges added in this step are drawn with dashed lines]{
		\includegraphics[scale=0.36]{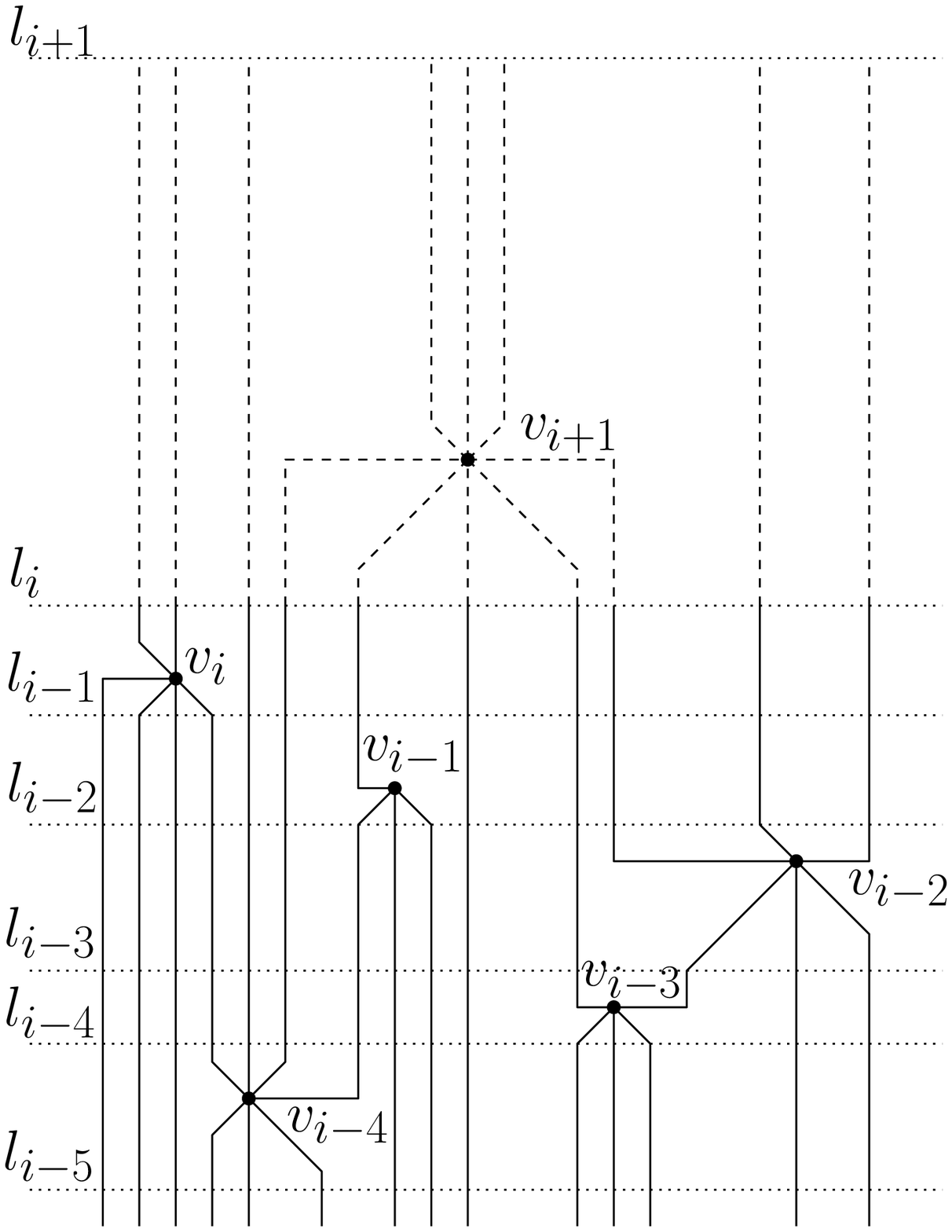}
               \hspace{5mm}
		 \label{fig:twobendsstep}
      }
      \caption{Drawing with at most two bends}
      \label{fig:twobendsboth}
\end{figure}

\begin{lem}
For any biconnected planar graph $G$ with maximum degree $d\ge 6$ and for any vertex $t\in V(G)$ with degree strictly less then $d$, $G$ admits a good drawing that is contained in the $t$-wedge.
\end{lem}
\begin{proof}%[Proof of Theorem \ref{twobends}.]
Take a planar drawing $D_G$ of $G$ such that $t$ is on the outer face and pick another vertex, $s$, from the outer face. Apply Lemma \ref{stlemma} to obtain an $st$-ordering with $v_1=s, v_2,$ and $v_n=t$ on the outer face of $D_G$ such that $v_1v_2$ is an edge of the outer face. We will build up a good drawing of $G$ by starting with $v_1$ and then adding $v_2,v_3,\ldots, v_n$ one by one to the outer face of the current drawing. As soon as we add a new vertex $v_i$, we also draw the initial pieces of the pending edges, and we make sure that the resulting drawing is equivalent to the drawing $D_{G_i}$.

Another property of the good drawing that we maintain is that every edge consists of precisely three pieces. (Actually, an edge may consist of fewer than 3 segments, because two consecutive pieces are allowed to have the same slope and form a longer segment) The middle piece will always be vertical, except for the middle piece of $v_1v_2$.

Suppose without loss of generality that $v_1$ follows directly after $v_2$ in the clockwise order of the vertices around the outer face of $D_G$. Place $v_1$ and $v_2$ arbitrarily in the plane so that the $x$--coordinate of $v_1$ is smaller than the $x$--coordinate of $v_2$. Connect $v_1$ and $v_2$ by an edge consisting of three segments: the segments incident to $v_1$ and $v_2$ are vertical and lie below them, while the middle segment has an arbitrary non-vertical regular slope.
Draw a horizontal auxiliary line $l_2$ above $v_1$ and $v_2$. Next, draw the initial pieces of the other (pending) edges incident to $v_1$ and $v_2$, as follows. For $i=1,2$, draw a short segment from $v_i$ for each of the edges incident to it (except for the edge $v_1v_2$, which has already been drawn) so that the directed slopes of the edges (including $v_1v_2$) form a contiguous interval and their circular order is the same as in $D_G$. Each of these short segments will be followed by a vertical segment that reaches above $l_2$. These vertical segments will belong to the middle pieces of the corresponding pending edges. Clearly, for a proper choice of the lengths of the short segments, no crossings will be created during this procedure. So far this drawing, including the partially drawn pending edges between $V(G_2)$ and $V(G)\setminus V(G_2)$, will be equivalent to the drawing $D_{G_2}$. As the algorithm progresses, the vertical segments will be further extended above $l_2$, to form the middle segments of the corresponding edges. For an illustration, see Figure \ref{fig:twobends}.

The remaining vertices $v_{i}, i>2$, will be added to the drawing one by one, while maintaining the property that the drawing is equivalent to $D_{G_i}$ and that the pending-order of the actual pending edges %incident to the outer face in $D_{G_i}
coincides with the order in which their vertical pieces reach the auxiliary line $l_i$. At the beginning of step $i+1$, these conditions are obviously satisfied. Now we show how to place $v_{i+1}$. 

Consider the set $X$ of intersection points of the vertical (middle) pieces of all pending edges between 
$V(G_i)$ and $V(G)\setminus V(G_i)$ with the auxiliary line $l_i$. By Proposition \ref{stremark}, the intersection points corresponding to the pending edges incident to $v_{i+1}$ must be consecutive elements of $X$. Let $m$ be (one of) the median element(s) of $X$. Place $v_{i+1}$ at a point above $m$, so that the $x$-coordinates of $v_{i+1}$ and $m$ coincide, and connect it to $m$. (In this way, the corresponding edge has only one bend, because its second and third piece are both vertical.) We also connect $v_{i+1}$ to the upper endpoints of the appropriately extended vertical segments passing through the remaining elements of $X$, so that the directed slopes of the segments leaving $v_{i+1}$ form a contiguous interval of regular slopes. For an illustration see Figure \ref{fig:twobendsstep}.
Observe that this step can always be performed, because, by the definition of $st$-orderings, the number of edges leaving $v_{i+1}$ is strictly smaller than $d$. This is not necessarily true in the last step, but then we have $v_n=t$, and we assumed that the degree of $t$ was smaller than $d$.
To complete this step, draw a horizontal auxiliary line $l_{i+1}$ above $v_{i+1}$ and extend the vertical portions of those pending edges between $V(G_i)$ and $V(G)\setminus V(G_i)$ that were not incident to $v_{i+1}$ until they hit the line $l_{i+1}$. (These edges remain pending in the next step.)
Finally, in a small vicinity of $v_{i+1}$, draw as many short segments from $v_{i+1}$ using the remaining directed slopes as many pending edges connect $v_{i+1}$ to $V(G)\setminus V(G_{i+1})$. Make sure that the directed slopes used at $v_{i+1}$ form a contiguous interval and the circular order is the same as in $D_G$. Continue each of these short segments by adding a vertical piece that hits the line $l_{i+1}$. The resulting drawing, including the partially drawn pending edges, is equivalent to $D_{G_{i+1}}$.

In the final step, if we place the auxiliary line $l_{n-1}$ high enough, then the whole drawing will be contained in the $v_n$-wedge and we obtain a drawing that meets the requirements.
\end{proof}

\subsection{Lower Bounds}
In this section, we construct a sequence of planar graphs, providing a nontrivial lower bound for the planar slope number of bounded degree planar graphs. They also require more than the trivial number ($\lceil d/2\rceil$) slopes, even if we allow one bend per edge.
Remember that if we allow {\em two} bends per edge, then, by Theorem \ref{twobends}, for all graphs with maximum degree $d\ge 3$, except for the octahedral graph, $\lceil d/2\rceil$ slopes are sufficient, which bound is optimal.

\medskip

\begin{thm} For any $d\ge 3$, there exists a planar graph $G_d$ with maximum degree $d$, whose planar slope number is at least $3d-6$.
In addition, any drawing of $G_d$ with at most one bend per edge requires at least $\frac 34(d-1)$ slopes.
\end{thm}

\medskip

\begin{figure}[ht]
\centering
  \subfigure[A straight line drawing of $G_6$]{
		\includegraphics[scale=0.5]{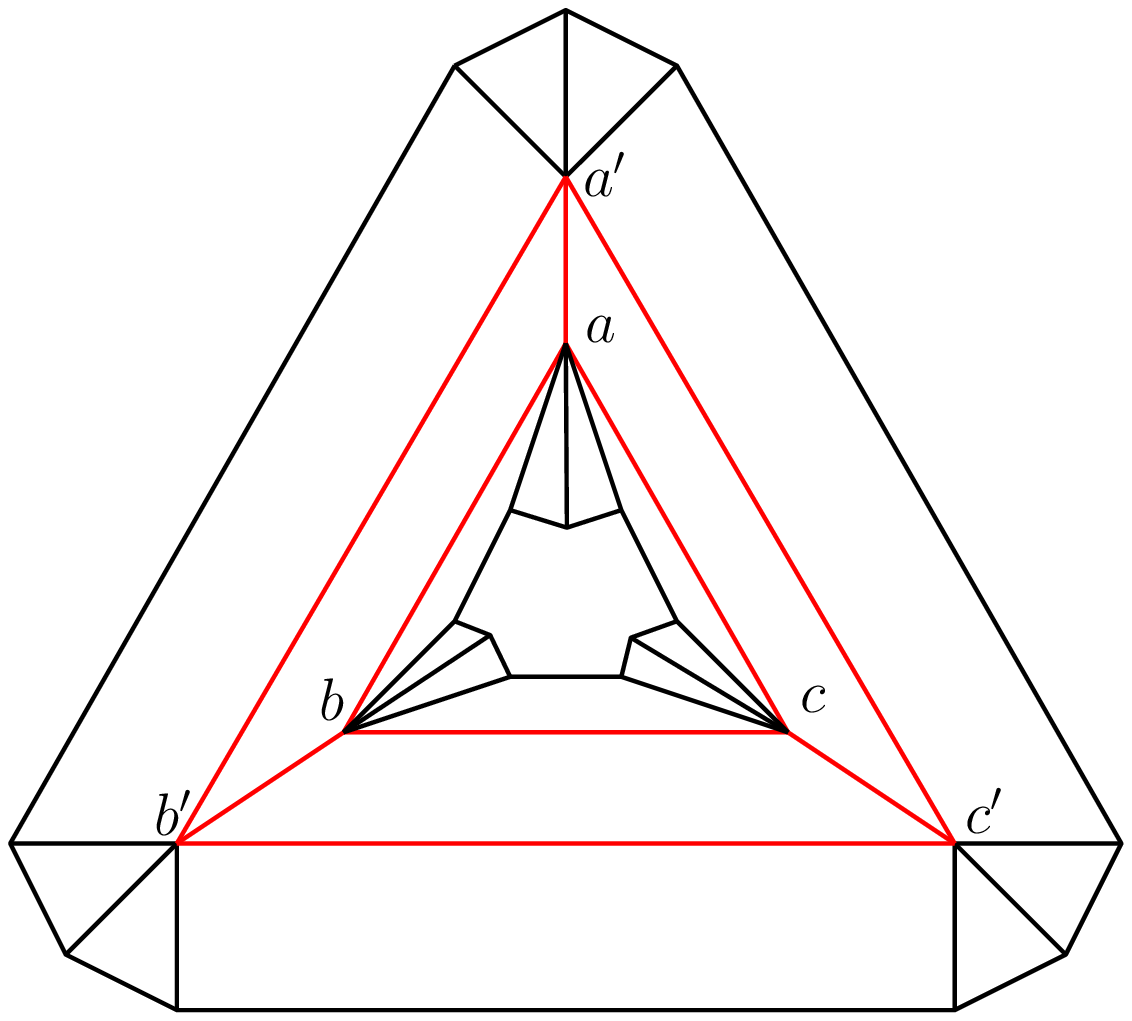}
       \label{fig:counterex}
      }		
	\hspace{2cm}
  \subfigure[At most four segments starting from $a,b,c$ can use the same slope in a drawing of $G_d$ with one bend per edge]
  {
		\includegraphics[scale=0.4]{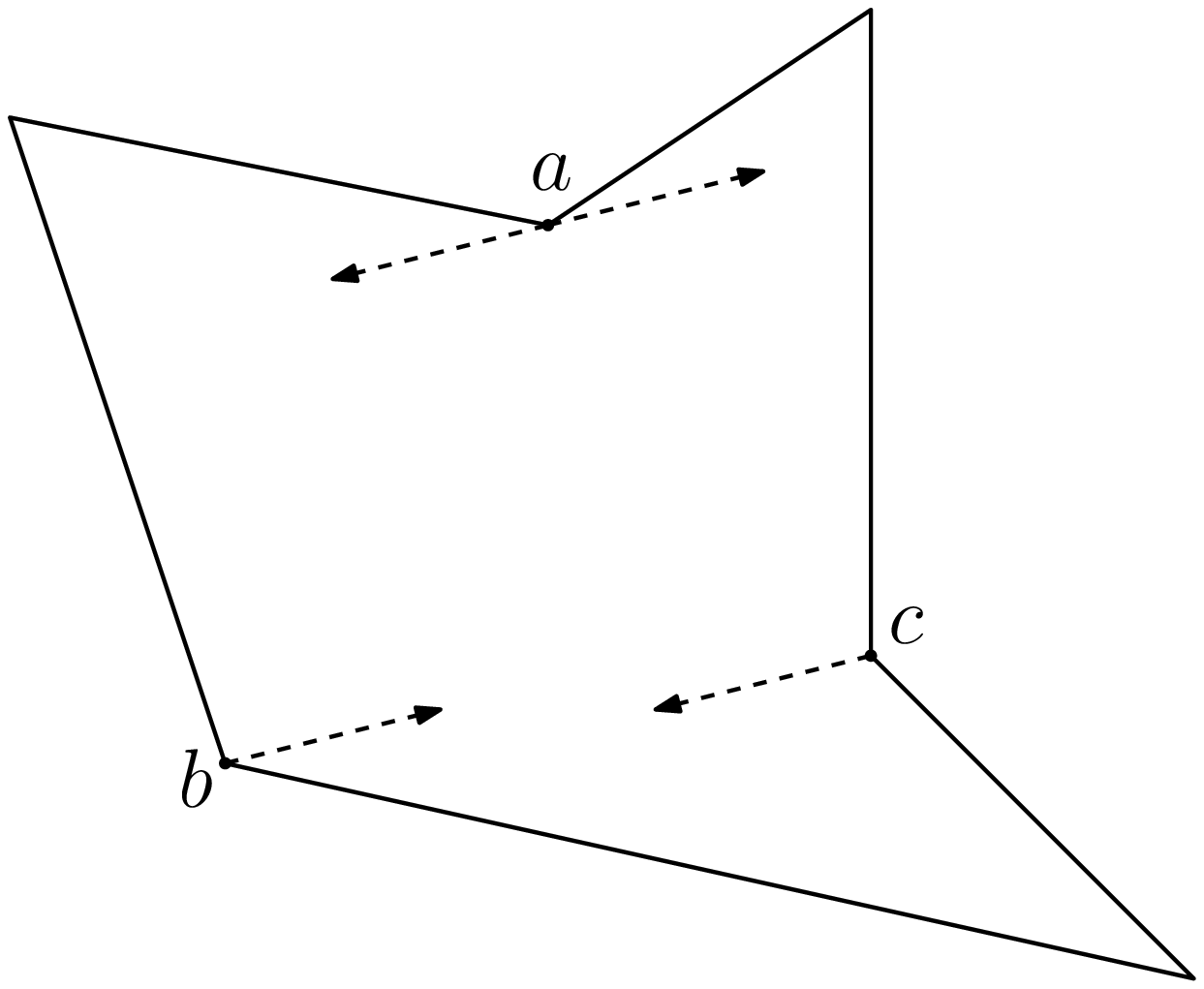}
      \label{fig:counterex1bend}
      }

      \caption{Lower bounds}
\end{figure}

\begin{proof}
The construction of the graph $G_d$ is as follows. Start with a graph of $6$ vertices, consisting of two triangles, $abc$ and $a'b'c'$, connected by the edges $aa'$, $bb'$, and $cc'$ (see Figure \ref{fig:counterex}). Add to this graph a cycle $C$ of length $3(d-3)$, and connect $d-3$ consecutive vertices of $C$ to $a$, the next $d-3$ of them to $b$, and the remaining $d-3$ to $c$. Analogously, add a cycle $C'$ of length $3(d-3)$, and connect one third of its vertices to $a'$, one third to $b'$, one third to $c'$. In the resulting graph, $G_d$, the maximum degree of the vertices is $d$.

In any crossing-free drawing of $G_d$, either $C$ lies inside the triangle $abc$ or $C'$ lies inside the triangle $a'b'c'$. Assume by symmetry that $C$ lies inside $abc$, as in Figure \ref{fig:counterex}.

If the edges are represented by straight-line segments, the slopes of the edges incident to $a,b,$ and $c$ are all different, except that $aa', bb',$ and $cc'$ may have the same slope as some other edge. Thus, the number of different slopes used by any straight-line drawing of $G_d$ is at least $3d-6$.

Suppose now that the edges of $G_d$ are represented by polygonal paths with at most one bend per edge. Assume, for simplicity, that every edge of the triangle $abc$ is represented by a path with exactly one bend (otherwise, an analogous argument gives an even better result). Consider the $3(d-3)$ polygonal paths connecting $a$, $b$, and $c$ to the vertices of the cycle $C$. Each of these paths has a segment incident to $a$, $b$, or $c$. Let $S$ denote the set of these segments, together with the $6$ segments of the paths representing the edges of the triangle $abc$.

\begin{claim} \label{4slopelemma}
The number of segments in $S$ with any given slope is at most $4$.
\end{claim}
\begin{proof}
The sum of the degrees of any polygon on $k$ vertices is $(k-2)\pi$.
Every direction is covered by exactly $k-2$ angles of a $k$-gon (counting each side $1/2$ times at its endpoints). Thus, if we take every other angle of a hexagon, then, even including its sides, every direction is covered at most $4$ times. (See Figure \ref{fig:counterex1bend}.)
\end{proof}

The claim now implies that for any drawing of $G$ with at most one bend per edge, we need at least $(3(d-3)+6)/4=\frac {3}{4}(d-1)$ different slopes.
\end{proof}

%% file: conj.tex
\subsection{Questions about Decomposition of Coverings}

\noindent {\bf Conjecture. (Pach)} {\em All planar convex sets are
cover-decomposable.}

\medskip

\noindent {\bf Conjecture \ref{conj:quad}.} {\em There is a constant $m$ such that any
$m$-fold covering of the plane with translates of a convex quadrilateral
can be decomposed into two coverings.}

\medskip

\noindent {\bf Conjecture \ref{conj:multipoly}.} {\em For any cover-decomposable polygon $P$, $m_k(P)=O(k)$.}

\medskip

In fact, the following, more general version seems to be open.

\begin{quest} Is it true that for any set $P$, $m_k(P)=O(k)$?
%(Where again the constant hidden in the $O$ notation depends on $P$).
\end{quest}

We cannot even say anything about decomposition into three coverings.

\begin{quest}\label{quest:m3} Is it true that for any set $P$, if $m_2(P)$ exists then $m_3(P)$ also exists?
\end{quest}

This latter can be asked about abstract sets instead of geometric sets in the following way.

Suppose we have a finite system of sets, $\mathcal F$.
We say that a multiset $\mathcal M$ is a multiset of $\mathcal F$ if its elements are from $\mathcal F$.
We say that $\mathcal M$ is $t$-fold if for every element of the ground set there are at least $t$ sets from $\mathcal M$ that contain it.
We say that $\mathcal M$ is $k$-wise decomposable if we can color the sets from $\mathcal M$ with $k$ colors such that every element is contained in a set of each color.
Define $m_k$ as the smallest number such that if a multiset of $\mathcal F$ is $m_k$-fold, then it is also $k$-wise decomposable.
This number always exists as it is easy to see that $m_k\le (k-1)|{\mathcal F}|+1$.

\begin{quest}
  Can we bound $m_3({\mathcal F})$ with some function of $m_2({\mathcal F})$?\\ (Independently from $\mathcal F$.)
\end{quest}

\begin{quest}
  Is it true that $m_k({\mathcal F})=O(k\cdot m_2({\mathcal F}))$?
\end{quest}

Tardos \cite{T09} proved a strongly related result. For any $m$ he constructed a set system, $\mathcal F$, that covers the ground set $m$-fold and any $2$-fold covering of the ground set with a subsystem of $\mathcal F$ is decomposable into two coverings but $\mathcal F$ cannot be decomposed into three coverings.

\medskip

\noindent {\bf Conjecture \ref{conj:threed}.} {\em Three-dimensional convex sets are not cover-decomposable.}

\medskip

\noindent {\bf Conjecture \ref{conj:closed}.} {\em  Closed, convex polygons are cover-decomposable.}

\medskip

\noindent {\bf Question \ref{quest:pcd}.} {\em  Are there polygons that are not totally-cover-decomposable but plane-cover-decomposable?}

\clearpage

Finally, another finite problem that has a strong connection to cover-decomposition. 

\begin{quest}
For $A\subset [n]$ denote by $a_i$ the $i^{th}$ smallest element of $A$.

For two $k$-element sets, $A,B\subset [n]$, we say that $A\le B$ if $a_i\le b_i$ for every $i$.

A $k$-uniform hypergraph ${\mathcal H}\subset [n]$ is called a {\em shift-chain} if for any hyperedges, $A, B \in {\mathcal H}$, we have $A\le B$ or $B\le A$. (So a shift-chain has at most $k(n-k)+1$ hyperedges.)

Is it true that shift-chains have Property B\footnote{A hypergraph  ${\mathcal H}$ has Property B if we can color its vertices with two colors such that no hyperedge is monochromatic.} if $k$ is large enough?

\end{quest}

An affirmative answer would  be a huge step towards Pach's conjecture, that all planar convex sets are cover-decomposable. To see this, for any fixed convex set $C$ and natural $k$, and any $y$ real number, define  $C(k;y)$ as the translate of $C$ which\\
(1) contains exactly $k$ points of a given point set, $S$,\\
(2) the center of $C$ has $y$-coordinate $y$,\\
(3) the center of $C$ has minimal $x$-coordinate,\\
if such a translate exists.
If we associate $i\in [n]$ to the element of $S$ with the $i^{th}$ smallest $y$-coordinate, then an easy geometric argument shows that ${\mathcal H}=\{C(k;y)\cap S|y\in {\mathbb R}\}$ is a shift-chain.\\

For $k=2$ there is a trivial counterexample to the question: (12),(13),(23).\\

A magical counterexample was found for $k=3$ by a computer program by Fulek \cite{F10}:

(123),(124),(125),(135),(145),(245),(345),(346),(347),(357),

(367),(467),(567),(568),(569),(579),(589),(689),(789).\\

If we allow the hypergraph to be the union of two shift-chains (with the same order), then the construction in Section \ref{sec:concave} gives a counterexample for any $k$, so arguments using that the average degree is small (like the Lov\'asz Local Lemma) probably fail.

\clearpage
\subsection{Questions about Slope Number of Graphs}

\medskip

\noindent {\bf Conjecture \ref{slope4}.} {\em  The slope number of graphs with maximum degree 4 is unbounded.}

\begin{quest} Can every cubic graph be drawn with the four basic directions\footnote{Vertical, horizontal and the two diagonal ($45^\circ$) directions.}?

What if no three vertices can be collinear?
\end{quest}

\noindent {\bf Question \ref{polygrid}.} {\em Is it possible to draw all cubic graphs with a bounded number of slopes on a polynomial sized grid?}

\begin{quest} Does the planar slope number of planar graphs with maximum degree $d$ grow exponentially or polynomially with $d$?
\end{quest}

\begin{conj} We can fix $O(d)$ slopes such that any planar graph with maximum degree $d$ can be drawn with these slopes if each edge can have one bend. 
\end{conj}

%% file: CV.tex
%\title{\bf CURRICULUM VITAE}
%\author{}
%\date{}
%\maketitle

\vspace{2cm}
\noindent

\!\!\!\!\!\!\!\!\!\!\!\!\!\!\!\!\!\!\!\!\!\!\!\!\!\!{\bf \underline{PERSONAL DATA}}\\

\noindent
First name: {\bf D\"om\"ot\"or}\\
Surname: {\bf P\'alv\"olgyi}\\
Date of birth: $4^{th}$ July, $1981$.\\
Nationality: Hungarian\\
Other languages: English (fluent), French (good), Italian, German, Telugu (basic)\\
e-mail: {\it doemoetoer.palvoelgyi@epfl.ch}\\

\bigskip
{\bf \underline{STUDIES}}\\

\indent
{\bf 2008-- \indent Ecole Polytechnique F\'ed\'erale de Lausanne (EPFL)}\\
\indent
\textit{Assistant-doctorant of J\'anos Pach}\\

{\bf 2000--2008 Eötvös Loránd University (ELTE), Budapest}\\
\indent
\textit{Graduate Studies in Pure Mathematics, Theoretical Computer Science}\\
\noindent
-Supervisor: Zolt\'an Kir\'aly.\\
-State scholarship for research in Communication Complexity, Complexity Theory, Combinatorial Geometry, Extremal Combinatorics, Network Information Flows since 2005.\\
-Member of the Communication Networks Laboratory.\\

%{\bf 2000--2005 Eötvös Loránd University (ELTE), Budapest}\\
\indent
\textit{Undergraduate Studies at Faculty of Science, Department of Mathematics}\\
\noindent
-Master thesis in Communication Complexity 2005.\\
-Holder of the State Scholarship Award (2004--2005).\\ 
-Holder of the Outstanding Student of the Faculty Award (2004).\\

{\bf 1994--2000          Fazekas Mihály High School, Budapest}\\
\indent
\textit{Class specialized in mathematics}\\
\noindent-Hungarian baccalaureate exams (qualification: outstanding), 2000.\\
-International Mathematical Olympiad, 
South Korea, silver medal, 2000.\\    
-National Mathematics Competition (OKTV),
Hungary, 2nd place, 2000\\
-International Hungarian Mathematics Competition,
Slovakia, 1st prize, 2000.\\

\bigskip
{\bf \underline{TEACHING EXPERIENCE}}\\
\noindent

{\bf 2008--}\\
-Teaching Assistant for Linear Algebra, Calculus, Geometric Graphs at EPFL.\\

{\bf 2005--2008}\\
-Lecturer and Teaching Assistant for Complexity Theory, Theory of Algorithms at ELTE.\\

{\bf 2002--2005}\\
-Teaching Assistant for Calculus, Real Analysis, Measure Theory at ELTE.\\

\bigskip
{\bf \underline{PUBLICATIONS}}\\

%{\bf In preparation:}\\
%\noindent

{\bf Submitted/Accepted:}\\
\noindent
-Bin Packing via Discrepancy of Permutations (with Friedrich Eisenbrand and Thomas Rothvoß), submitted.\\
-Almost optimal pairing strategy for Tic-Tac-Toe with numerous directions (with Padmini Mukkamala), submitted.\\
-Unique-maximum and conflict-free colorings for hypergraphs and tree graphs. (with Panagiotis Cheilaris and Balázs Keszegh), submitted.\\
-Indecomposable coverings with concave polygons, to appear in: Discrete and Computational Geometry.\\
-Permutations, hyperplanes and polynomials over finite fields (with András Gács, Tamás Héger and Zoltán Lóránt Nagy), to appear in: Finite Fields and Their Applications. (Earlier version: 22nd British Combinatorial Conference.)\\

{\bf 2010}\\
\noindent
-On weakly intersecting pairs of sets (with Zoltán Király, Zoltán Lóránt Nagy and Mirkó Visontai), presented at: 7th International Conference on Lattice Path Combinatorics and Applications.\\
-Vectors in a Box (with Kevin Buchin, Jiri Matousek and Robin A. Moser), presented at: Coimbra Meeting on 0-1 Matrix Theory and Related Topics.\\
-Consistent digital line segments (with Tobias Christ and Milos Stojakovic), in: SoCG 2010.
-Convex polygons are cover-decomposable (with G. Tóth), in: Discrete and Computational Geometry.\\
-Testing additive integrality gaps (with Friedrich Eisenbrand, Nicolai Hähnle and Gennady Shmonin), in: SODA 2010.\\
-Cubic Graphs Have Bounded Slope Parameter (with B. Keszegh, J. Pach, and G. Tóth), in: J. Graph Algorithms Appl. 14(1): 5-17 (2010).
(Earlier version: Proceedings of Graph Drawing 2008, 50--60.)\\
-Finding the biggest and smallest element with one lie (with D. Gerbner, B. Patkós and G. Wiener), in: Discrete Applied Mathematics 158(9): 988-995 (2010). (Earlier version: International Conference on Interdisciplinary Mathematical and Statistical Techniques - IMST 2008 / FIM.)\\
-Polychromatic Colorings of Arbitrary Rectangular Partitions (with D. Gerbner, B. Keszegh, N. Lemons, C. Palmer and B. Patkós), in: Discrete Math. 310, No. 1, 21-30 (2010). (Earlier version: 6th Japanese-Hungarian Symposium on Discrete Mathematics and Its Applications.)\\

{\bf 2009}\\
\noindent
-2D-TUCKER is PPAD-complete, in: WINE 2009.\\
-Combinatorial necklace splitting, in: Electronic J. Combinatorics 16 (1) (2009), R79.\\
-Deciding Soccer Scores and Partial Orientations of Graphs, in: Acta Universitatis Sapientiae, Mathematica, 1, 1 (2009) 35--42. (Earlier version in: EGRES Technical Reports 2008.)\\

{\bf 2008}\\
\noindent
-Drawing cubic graphs with at most five slopes (with B. Keszegh, J. Pach, and G. Tóth), in: Comput. Geom. 40 (2008), no. 2, 138--147.
(Earlier version: Graph Drawing 2006, Lecture Notes in Computer Science 4372, Springer, 2007, 114--125.)\\

{\bf 2007}\\
\noindent
-Revisiting sequential search using question-sets with bounded intersections, in: Journal of Statistical Theory and Practice Vol 1, Num 2 (2007).\\

{\bf 2006}\\
\noindent
-P2T is NP-complete, in: EGRES Quick-Proofs 2006.\\
-Bounded-degree graphs can have arbitrarily large slope numbers (with J. Pach), Electronic J. Combinatorics 13 (1) (2006), N1.\\

{\bf 2005}\\
\noindent
-Baljó S Árnyak, in: Matematikai Lapok (in Hungarian). English title: A short proof of the Kruskal-Katona theorem.\\
-Communication Complexity (Master's Thesis). Supervisor: Z. Kir\'aly.\\

\bigskip
{\bf \underline{ATTENDED CONFERENCES AND WORKSHOPS}}\\

{\bf 2010}\\
\noindent
-1st Eml\'ekt\'abla Workshop in Gy\"ongy\"ostarj\'an, July 26--29.\\
-8th Gremo's Workshop on Open Problems in Morschach (SZ), June 30--July 2.\\
-Coimbra Meeting on 0--1 Matrix Theory and Related Topics in Coimbra, June 17--19.\\
-3\`eme cycle romand de Recherche Op\'erationnelle in Zinal, January 17--21.\\

{\bf 2009}\\
\noindent
-WINE 2009 in Rome, December 16--18.\\
-Workshop on Combinatorics: Methods and Applications in Mathematics and Computer Science at IPAM, UCLA, September 7--November 7.\\
-The 14th International Conference on Random Structures and Algorithms in Poznan, August 3--7.\\
-7th Gremo's Workshop on Open Problems in Stels (GR), July 6--10.\\
-Algorithmic and Combinatorial Geometry in Budapest, June 15--19.\\
-6th Japanese-Hungarian Symposium on Discrete Mathematics and Its Applications in Budapest, May 16--19.\\
-3\`eme cycle romand de Recherche Op\'erationnelle in Zinal, January 18--22.\\

{\bf 2008}\\
\noindent
-16th International Symposium on Graph Drawing (GD 2008) in Heraklion, September 21--24.\\
-Building Bridges (L. Lovász 60), Fete of Combinatorics and Computer Science in Budapest and in Keszthely, Lake Balaton, August 5 -- 15.\\
-Bristol Summer School on Probabilistic Techniques in Computer Science, July 6--11.\\
-International Conference on Interdisciplinary Mathematical and Statistical Techniques
(IMST 2008 / FIM XVI) at University of Memphis, May 15--18.\\
-Expanders in Pure and Applied Mathematics at IPAM, UCLA, February 11--15.\\

{\bf 2007}\\
\noindent
-HSN Spring Workshop in Balatonkenese, May 31--June 1.\\
-Extremal Combinatorics Workshop at Rényi Institute of Mathematics in Budapest, June 4--8.\\
-Advanced Course on Analytic and Probabilistic Techniques in Combinatorics at Centre de Recerca Matem\`atica in Bar\-ce\-lo\-na, January 15--26.\\

{\bf 2006}\\
\noindent
-Final Combstru Workshop at Universitat Polit\'ecnica de Catalunya in Bar\-ce\-lo\-na, September 26--28.\\
-14th International Symposium on Graph Drawing (GD 2006) in Karlsruhe, September 18--20.\\ 
-Horizon of Combinatorics, EMS Summer School and Conference at Rényi Institute of Mathematics in Budapest and in Balatonalmádi, July 10--22.\\
-HSN Spring Workshop in Balatonkenese, May 23--24.\\

{\bf 2005}\\
\noindent
-European Conference on Combinatorics, Graph Theory, and Applications (EuroComb 2005) at Technische Universität in Berlin, September 5--9.\\

{\bf Summer 2004    RIPS, UCLA}\\
\indent
I participated in an REU (Research \& Education for Undergraduates) and I was assigned to a project sponsored by BioDiscovery. Our goal was to create an algorithm that would reveal which genes are responsible for cancer from an input database.\\
	
{\bf Spring 2002	    Hungarian-Dutch Exchange Program, Eindhoven}\\
\indent
The program was sponsored by Philips, and we were asked to work on various projects.